\documentclass{article}

\usepackage[T1]{fontenc}
\usepackage[utf8]{inputenc}
\usepackage[english]{babel}
\usepackage{lmodern}
\usepackage[margin=3cm]{geometry}

\usepackage[backend=biber,style=alphabetic,sorting=nty,doi=false,isbn=false,url=false,minalphanames=2,maxbibnames=4,maxcitenames=4,giveninits=true,date=year]{biblatex}
\AtEveryBibitem{\clearlist{language}}

\usepackage{empheq}

\usepackage{authblk}
\usepackage[hidelinks, colorlinks=false]{hyperref}
\usepackage[shortlabels]{enumitem}
\usepackage{alltt,bigints,bbm,slashed,mathtools,amssymb,amsmath,amsfonts,amsthm} 
\usepackage[nottoc]{tocbibind}
\usepackage{graphicx}
\usepackage[font={small,it}]{caption}
\usepackage{subcaption}
\usepackage{aliascnt}
\usepackage{cases}
\usepackage{comment}	
\usepackage[capitalize,nameinlink]{cleveref}
\usepackage{tikz}
\usetikzlibrary{tikzmark, patterns}
\usepackage{pgfplots}
\pgfplotsset{compat=1.18}
\usepackage{slashed}

\crefname{figure}{Fig.}{Fig.}
\Crefname{figure}{Fig.}{Fig.}
\crefname{subfigure}{Fig.}{Fig.}
\Crefname{subfigure}{Fig.}{Fig.}

\usepackage{xargs}   
\usepackage[colorinlistoftodos,prependcaption,textsize=tiny]{todonotes}
\newcommandx{\typo}[2][1=]{\todo[linecolor=red,backgroundcolor=red!25,bordercolor=red,#1]{#2}}
\newcommandx{\change}[2][1=]{\todo[linecolor=blue,backgroundcolor=blue!25,bordercolor=blue,#1]{#2}}
\newcommandx{\answer}[1]{\todo[linecolor=pink,backgroundcolor=pink!25,bordercolor=pink]{#1}}
\newcommandx{\unsure}[2][1=]{\todo[linecolor=green,backgroundcolor=green!25,bordercolor=green,#1]{#2}}
\newcommandx{\improve}[2][1=]{\todo[linecolor=violet,backgroundcolor=violet!25,bordercolor=violet,#1]{#2}}
\newcommandx{\thiswillnotshow}[2][1=]{\todo[disable,#1]{#2}}

\usepackage{tocloft}
\crefformat{equation}{(#2#1#3)}
\numberwithin{equation}{section}

\theoremstyle{definition}

\newtheorem*{theorem*}{Theorem}
\newtheorem*{conjecture*}{Conjecture}

\newtheorem{definition}{Definition}[section]

\theoremstyle{plain}
\newtheorem{theorem}{Theorem}[section]
\newtheorem{corollary}[theorem]{Corollary}
\newtheorem{lemma}[theorem]{Lemma}

\newtheorem{prop}[theorem]{Proposition}
\newtheorem{conjecture}[theorem]{Conjecture}
\newtheorem{cor}[theorem]{Corollary}
\newtheorem*{claim}{Claim}

\theoremstyle{remark}
\newtheorem{remark}{Remark}[section]
\newtheorem{notation}{Notation}[section]

\crefname{lemma}{Lemma}{Lemmas}
\crefname{prop}{Proposition}{Proposition}
\crefname{conjecture}{Conjecture}{Conjecture}
\crefname{cor}{Corollary}{Corollary}
\crefname{remark}{Remark}{Remark}
\crefname{ass}{Assumption}{Assumptions}
\crefname{defi}{Definition}{Definition}
\crefname{equation}{}{}
\crefname{enumi}{}{}
\crefname{appendix}{}{}

\newcommand{\Mhat}{\hat{\mathcal{M}}}
\newcommand{\Shat}{\hat{\Sigma}}

\newcommand{\dd}{\mathop{}\!\mathrm{d}}

\graphicspath{{figures}}

\newcommand{\Vt}{\mathcal{V}_{\mathrm{top}}}
\newcommand{\Vb}{\mathcal{V}_{\mathrm{b}}}
\newcommand{\Vtop}{\mathcal{V}_{\mathrm{top}}}
\newcommand{\M}{\mathcal{M}}
\newcommand{\Ab}{X_{\infty,\infty}}
\newcommand{\A}{\mathcal{A}}
\newcommand{\Hb}{{X}_{2,2}}
\newcommand{\Hc}{{X}_{\infty,2}}
\newcommand{\Hs}{{X}_{\mathrm{s}}}
\newcommand{\Diffb}{\mathrm{Diff}_{\mathrm{b}}}

\newenvironment{nalign}{
	\begin{equation}
		\begin{aligned}
		}{
		\end{aligned}
	\end{equation}
	\ignorespacesafterend
}

\newcommand{\N}{\mathbb{N}}

\newcommand{\R}{\mathbb{R}}

\newcommand{\T}{\mathbb{T}}

\newcommand{\abs}[1]{\left\lvert #1\right\rvert}

\newcommand{\norm}[1]{\left\lVert #1\right\rVert}

\newcommand{\floor}[1]{\lfloor #1 \rfloor}

\newcommand{\BB}{\mathbb{B}}

\newcommand{\D}{\mathcal{D}}

\DeclareMathOperator{\supp}{supp}

\newcommand{\blue}[1]{{\color{blue}{#1}}}



\DeclareFontFamily{U}{matha}{\hyphenchar\font45}
\DeclareFontShape{U}{matha}{m}{n}{
      <5> <6> <7> <8> <9> <10> gen * matha
      <10.95> matha10 <12> <14.4> <17.28> <20.74> <24.88> matha12
      }{}
\DeclareSymbolFont{matha}{U}{matha}{m}{n}
\DeclareFontSubstitution{U}{matha}{m}{n}

\DeclareFontFamily{U}{mathx}{\hyphenchar\font45}
\DeclareFontShape{U}{mathx}{m}{n}{
      <5> <6> <7> <8> <9> <10>
      <10.95> <12> <14.4> <17.28> <20.74> <24.88>
      mathx10
      }{}
\DeclareSymbolFont{mathx}{U}{mathx}{m}{n}
\DeclareFontSubstitution{U}{mathx}{m}{n}

\DeclareMathDelimiter{\vvvert}{0}{matha}{"7E}{mathx}{"17}

\usepackage{csquotes}
\usepackage{mathrsfs}

\title{Scattering and stability for ODE-type blow-up surfaces for focusing nonlinear wave equations}
\author[1]{Istvan Kadar}
\affil[1]{\small \textit{Department of Mathematics, ETH Zurich, R\"amistrasse 101, 8006 Zurich, Switzerland}}
\author[2,3]{Warren Li}
\affil[2]{\small \textit{Department of Mathematics, Stanford University, 450 Jane Stanford Way Building 380, \protect\\ Stanford, CA 94305, USA}}
\affil[3]{\small \textit{Trinity College, University of Cambridge, Cambridge CB2 1TQ, United Kingdom}}

\newcommand{\tstar}{\mathbf{t}}
\newcommand{\Mt}{\mathcal{M}_{\tstar_1}}
\newcommand{\J}{\mathbb{J}}
\DeclareMathOperator{\Jac}{\mathcal{J}}
\begin{document}

\maketitle
	
\begin{abstract}
    We study the focusing power nonlinear wave equation with any power, in Minkowski space of any spacetime dimension. We present a complete understanding of the local stability and scattering theory (both in high regularity spaces) for solutions exhibiting \emph{ODE type blow-up} on spacelike hypersurfaces, with the blow-up at each point modelled by the explicit solution $\phi_{\mathrm{model}} = c_p t^{-\alpha_p}$.

    Given a sufficiently regular spacelike hypersurface $\Sigma_f$, together with auxiliary scattering data $\psi$, we construct the unique corresponding solution to the nonlinear wave equation that (locally) forms an ODE type singularity on $\Sigma_f$ attaining $\psi$ as scattering data. Conversely, we show that such ODE type singularities are (locally) stable to suitably regular perturbations away from the singularity, and that the blow-up surface and scattering data remain regular, in a continuously dependent manner, following such perturbations.
\end{abstract}

\setcounter{tocdepth}{2}
\tableofcontents

\section{Introduction} \label{sec:intro}

In this paper, we consider the focusing nonlinear wave equation
\begin{equation}\label{in:eq:main}
    P[\phi]\coloneqq\Box\phi+|\phi|^{p-1}\phi=(-\partial_t^2+\Delta_x)\phi+ |\phi|^{p-1} \phi=0,
\end{equation}
for any exponent $p > 1$ in Minkowski space $(\R^{n+1}, \eta)$ of any spatial dimension $n \geq 1$. It is well known that \cref{in:eq:main} may form singularities in finite time given by the explicit homogeneous solution
\begin{equation} \label{in:eq:model}
    \phi_{\mathrm{model}} = c_p t^{- \alpha_p}, \quad \text{ where } \alpha_p = \frac{2}{p-1}, \; c_p^{p-1} = \frac{2(p+1)}{(p-1)^2}.
\end{equation}
Such a solution is often denoted to have \emph{ODE type} blow up at the hyperplane $\{ t = 0\}$. Via a Poincar\'e transformation, it is straightforward to show the existence of a solution that forms a singularity on any spacelike hyperplane, either forwards or backwards in time. That is, for $|\mathbf{v}| < 1$ and $T \in \R$, we can construct solutions of the form
\[
    \phi = c_p (1 - |\mathbf{v}|^2)^{\frac{1}{p-1}} \cdot (T \pm t - \mathbf{v} \cdot x)^{-\alpha_p}.
\]

In this article, we achieve two related goals. Firstly, we show that we can elevate the existence statement to show the existence of a solution to \eqref{in:eq:main} that blows up on any given smooth \emph{spacelike hypersurface}; in fact we characterize all such solutions via \emph{auxiliary scattering data}. Secondly, we show that such ODE type blow up solutions are stable to perturbations of its induced Cauchy data on spacelike hypersurfaces away from the singularity; in particular, after such perturbations the blow up surfaces and auxiliary scattering data remain regular and close to their initial values.

To state our theorems, it will be helpful to introduce the quantities $\beta_p, \gamma_p$ and $\kappa_p$, as
\begin{equation} \label{in:eq:betadelta}
    \beta_p = \frac{2p}{p-1}, \quad \gamma_p = \beta_p (\beta_p - 1) = \frac{2p(p+1)}{(p-1)^2}, \quad \kappa_p = \frac{\beta_p + \alpha_p}{2} = \frac{p+1}{p-1}.
\end{equation}
We also use the notation $\BB_r(x)$ to denote the ball in Euclidean space $\R^n$ of radius $r > 0$ and center $x$, and $\BB_r = \BB_r(0)$. Then the first of our main results is as follows.

\begin{theorem}[Rough version of \cref{in:thm:main1_precise}]\label{in:thm:main}
    Let $s\geq s_0\coloneqq30 \kappa_p +\frac{n+1}{2}+5$ and $f\in H^{s}(\BB_3)$ $\psi\in H^{s-\floor{2\kappa_p}}(\BB_3)$ with $\abs{\partial f}<1/10$ and $f(0)=0$. Write $\tstar=t-f(x)$.
    Then there exist $\tstar_1 \in (0, 1)$, depending on $\norm{f}_{H^{s_0}(\BB_3)}$ and $\norm{\psi}_{H^{s_0-\floor{2\kappa_p}}(\BB_3)}$, and a singular solution $\phi\in C^{s-\frac{n+1}{2}-12\kappa_p}(\mathcal{M}_{\tstar_1})$ to \eqref{in:eq:main} in the domain $\mathcal{M}_{\tstar_1} = \{|x|/2 < 1 - t: \tstar\in(0,\tstar_1)\}$ attaining $(\psi,f)$ as ``data'' at $\tstar = 0$, as defined in \cref{geo:def:scattering_sol}.
		
     In particular $\phi$ forms an ODE type singularity at the hypersurface $\Sigma_f=\{t=f(x)\}$ as 
    \begin{equation}\label{in:eq:phi_scat_rough}
    	\phi = (1 - |\partial f|^2)^{\frac{1}{p-1}} \cdot c_p \tstar^{-\alpha_p} + o(\tstar^{- \alpha_p}), \qquad \text{ as } \tstar \to 0.
    \end{equation}
    See \cref{fig:backward} for a visual representation.
\end{theorem}
For simplicity, we have, for now, omitted from \cref{in:eq:phi_scat_rough} the precise way the data $\psi$ is used in determining $\phi$.
In \cref{in:sec:ideas}, we will discuss the relevance of this term for the spatially homogeneous (and thus ODE) analogue of \eqref{in:eq:main}, and refer the reader to \cref{in:eq:expansion} for precise asymptotics.

\begin{figure}[ht]
    \centering
    \begin{tikzpicture}[scale=0.5]
\begin{axis}[
    axis lines=none,
    domain=-3:3,
    samples=100,
    width=15cm,
    height=9cm,
    xmin=-2.5, xmax=2.5,
    ymin=-0.5, ymax=2.5,
    thick
]

    \fill [lightgray, opacity=0.5]
        (-2, -0.3)
        -- plot[domain = -2:-0.4] ({\x}, {(1/50)*(\x-3)*(\x-1)*(\x+1) + 1 - abs(\x)/2})
        -- plot[domain = -0.4:0.4] ({\x}, {(1/50)*(\x-3)*(\x-1)*(\x+1) + 1 - 0.2})
        -- plot[domain = 0.4:2] ({\x}, {(1/50)*(\x-3)*(\x-1)*(\x+1) + 1 - abs(\x)/2})
        -- plot[domain = 2:-2] ({\x}, {(1/50)*(\x-3)*(\x-1)*(\x+1)})
        -- (-4, 5)
        -- cycle;

    \addplot[domain=-2:-0.4] {(1/50)*(x-3)*(x-1)*(x+1) + 1 - abs(x)/2} node[above left, midway] {$\mathscr{C}$};
    \addplot[domain=0.4:2] {(1/50)*(x-3)*(x-1)*(x+1) + 1 - abs(x)/2} node[above right, midway] {$\mathscr{C}$};

    \addplot[black, dotted, very thick] {(1/50)*(x-3)*(x-1)*(x+1)};
    \addplot[black, very thick, domain=-0.4:0.4] {(1/50)*(x-3)*(x-1)*(x+1) + 0.8} node[above, midway] {$\Sigma_{\mathbf{t}_1}$};
    \addplot[black, dashed, domain=-1.6:1.6] {(1/50)*(x-3)*(x-1)*(x+1) + 0.2} node[above, midway] {$\Sigma_{\mathbf{t}}$};

    \draw[->, thick] (axis cs:-2.3,1.15) -- (axis cs:-1.9,1.15) node[right, font=\small] {$x^i$};
    \draw[->, thick] (axis cs:-2.1,1.15) -- (axis cs:-2.1,1.35) node[above, font=\small] {$t$};

    \node at (axis cs:0,-0.15) {$\mathscr{S} = \{ t = f(x) \}$};
\end{axis}
    \end{tikzpicture}
    \qquad
    \begin{tikzpicture}[scale=0.5]
\begin{axis}[
    axis lines=none,
    domain=-3:3,
    samples=100,
    width=15cm,
    height=9cm,
    xmin=-2.5, xmax=2.5,
    ymin=-0.5, ymax=2.5,
    thick
]

    \fill [lightgray, opacity=0.5]
        (-2, -0.3)
        -- plot[domain = -2:-0.4] ({\x}, {1 - abs(\x)/2})
        -- plot[domain = -0.4:0.4] ({\x}, {+ 1 - 0.2})
        -- plot[domain = 0.4:2] ({\x}, {1 - abs(\x)/2})
        -- plot[domain = 2:-2] ({\x}, {0})
        -- cycle;

    \addplot[domain=-2:-0.4] {1 - abs(x)/2} node[above left, midway] {$\mathscr{C}$};
    \addplot[domain=0.4:2] {1 - abs(x)/2} node[above right, midway] {$\mathscr{C}$};

    \addplot[black, dotted, very thick] {0};
    \addplot[black, very thick, domain=-0.4:0.4] {0.8} node[above, midway] {$\Sigma_{\mathbf{t}_1}$};
    \addplot[black, dashed, domain=-1.6:1.6] {0.2} node[above, midway] {$\Sigma_{\mathbf{t}}$};

    \draw[->, thick] (axis cs:-2.3,1.15) -- (axis cs:-1.9,1.15) node[right, font=\small] {$y^i$};
    \draw[->, thick] (axis cs:-2.1,1.15) -- (axis cs:-2.1,1.35) node[above, font=\small] {$\mathbf{t}$};

    \node at (axis cs:0,-0.15) {$\mathscr{S} = \{ \tstar = 0 \}$};
\end{axis}
\end{tikzpicture}
\captionsetup{justification = centering}
\caption{An illustration of \cref{in:thm:main}. We show the existence of a solution to \eqref{in:eq:main} in the shaded domain, with the left picture showing the region in $(t, x)$ coordinates and the right picture showing the region in transformed $\tstar = t - f(x), y^i = x^i$ coordinates. The past boundary of the region is such that the solution $\phi$ exhibits ODE blow-up on $\{ t = f(x) \}$, with suitable scattering data $\psi$.}
\label{fig:backward}
\end{figure}
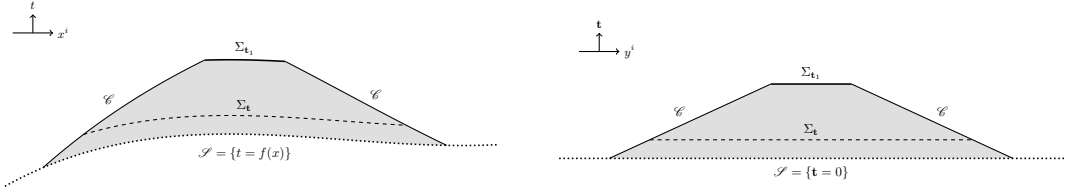

For our second result, concerning the stability of the singularity, we first present the theorem for perturbations of the model blow-up solution \eqref{in:eq:model}, which has induced Cauchy data $\phi = c_p$ and $\partial_t \phi = - \alpha_p c_p$ at $t = 1$. For the more general result see \cref{in:cor:general}.

\begin{theorem}[Rough version of \cref{in:thm:main2_precise}] \label{in:thm:main2}
    There exists $\epsilon_* > 0$ such that, for all $0 < \epsilon < \epsilon_*$, the following holds: suppose $s \geq s_0\coloneqq202 \kappa_p + \frac{n}{2} + 2$, and let $(\phi_0, \phi_1) \in H^{s+2}(\BB_6) \times H^{s+1}(\BB_6)$, such that
    \begin{equation}\label{in:eq:forward_ass}
        \| \phi_0 - c_p \|_{H^{s_0+2}} + \| \phi_1 + \alpha_p c_p \|_{H^{s_0+1}} \leq \epsilon. 
    \end{equation}
    Consider the solution $\phi$ to \eqref{in:eq:main} with initial data $\phi|_{t = 1} = \phi_1$ and $\partial_t \phi|_{t = 1} = \phi_0$. Then there exists functions $f, \psi\in H^{s-202\kappa_p}(\BB_4)$ such that $\phi$ can be extended to the domain $\mathcal{M} = \{ |x| < 3(1 + t),  f(x) < t \leq 1 \}$, and exhibits ODE type blow-up on the spacelike hypersurface $\Sigma = \{ t = f(x) \} \cap \mathcal{M}$, with auxiliary scattering data $\psi$. See \cref{fig:forward} for a visual representation.
\end{theorem}

Throughout this work, we will refer to \cref{in:thm:main} as the scattering construction and \cref{in:thm:main2} as the stability problem. 
Note that via the time reflection symmetry of the main equation \eqref{in:eq:main}, these constructions can apply equally to a singularity to the future or to the past of the hypersurface on which initial data is prescribed. 

\begin{figure}
    \begin{tikzpicture}[scale=0.5]
\begin{axis}[
    axis lines=none,
    domain=-3:3,
    samples=100,
    width=15cm,
    height=9cm,
    xmin=-2.5, xmax=2.5,
    ymin=-0.5, ymax=2.5,
    thick
]

    \fill [lightgray, opacity=0.5]
        (-2.2, 0.8) -- (2.2, 0.8) 
        .. controls (1.5, 0.5) .. (0.7, 0.03)
        -- plot[domain = 0.7:-0.7] ({\x}, {(1/50)*(\x-3)*(\x-1)*(\x+1)})
        -- (-0.7, 0.03)
        .. controls (-1.5, 0.4) .. (-2.2, 0.8);

    \draw[dotted] (-2.2, 0.8) -- (-0.3, 0.05);
    \draw[dotted] (2.2, 0.8) -- (0.3, 0.05);
    \draw (-2.2, 0.8) .. controls (-1.5, 0.4) .. (-0.7, 0.03) node[below left, midway] {$\mathcal{C}$};
    \draw (+2.2, 0.8) .. controls (1.5, 0.5) .. (0.7, 0.03) node[below right, midway] {$\mathcal{C}$};

    \addplot[black, dotted, very thick] {(1/50)*(x-3)*(x-1)*(x+1)};
    \addplot[black, very thick, domain=-2.2:2.2] {0.8} node[above, midway] {$\{ t = 1 \}$};
    \addplot[black, dashed, domain=-1.0:1.0] {(1/50)*(x-3)*(x-1)*(x+1) + 0.2} node[above, midway] {$\Sigma_{\tau}$};

    \draw[->, thick] (axis cs:-2.3,1.15) -- (axis cs:-1.9,1.15) node[right, font=\small] {$x^i$};
    \draw[->, thick] (axis cs:-2.1,1.15) -- (axis cs:-2.1,1.35) node[above, font=\small] {$t$};

    \node at (axis cs:0,-0.15) {$\mathcal{S}$};
\end{axis}
\end{tikzpicture}
\qquad
\begin{tikzpicture}[scale=0.5]
\begin{axis}[
    axis lines=none,
    domain=-3:3,
    samples=100,
    width=15cm,
    height=9cm,
    xmin=-2.5, xmax=2.5,
    ymin=-0.5, ymax=2.5,
    thick
]

    \fill [lightgray, opacity=0.5]
        (-2.2, 0.8)
        -- plot[domain=-2.2:2.2] ({\x}, {0.8+0.01*(\x-2.2)*(\x+2.2)*(\x-0.5)*(\x+0.1)})
        -- (2.2, 0.8)
        -- (0.7, 0)
        -- (-0.7, 0)
        -- cycle;

    \draw[dotted] (-2.2, 0.8) .. controls (-1.1, 0.4) .. (-0.3, 0);
    \draw[dotted] (2.2, 0.8) .. controls (1.1, 0.42) .. (0.3, 0);
    \draw (-2.2, 0.8) -- (-0.7, 0) node[below left, midway] {$\mathcal{C}$};
    \draw (+2.2, 0.8) -- (0.7, 0.) node[below right, midway] {$\mathcal{C}$};

    \addplot[black, dotted, very thick] {0};
    \addplot[black, very thick, domain=-2.2:2.2] {0.8 + 0.01*(x-2.2)*(x+2.2)*(x-0.5)*(x+0.1)} node[above, midway] {$\{ t = 1 \}$};
    \addplot[black, dashed, domain=-1.0:1.0] {0.2} node[above, midway] {$\Sigma_{\tau}$};

    \draw[->, thick] (axis cs:-2.3,1.15) -- (axis cs:-1.9,1.15) node[right, font=\small] {$z^i$};
    \draw[->, thick] (axis cs:-2.1,1.15) -- (axis cs:-2.1,1.35) node[above, font=\small] {$\tau$};

    \node at (axis cs:0,-0.15) {$\mathcal{S}$};
\end{axis}
\end{tikzpicture}
\captionsetup{justification = centering}
\caption{An illustration of \cref{in:thm:main2}. We show that regular perturbations of the model ODE-blow up data at $\{ t = 1 \}$ only has a perturbative effect on the (past) singular boundary $\mathscr{S}$.}
\label{fig:forward}
\end{figure}
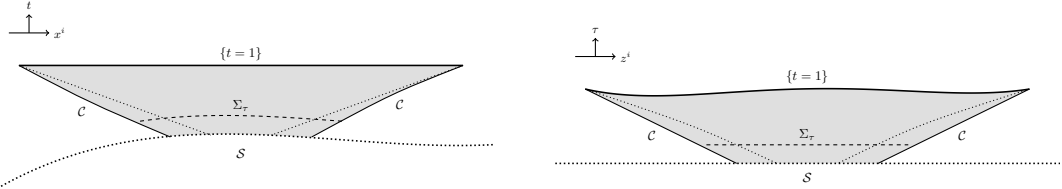

\subsection{Background on focusing semilinear waves}\label{in:sec:background}

The equation \cref{in:eq:main} has a large history, that we cannot review in all detail, so we mainly focus on the studies of ODE type singularity formation, but see \cite{strauss_nonlinear_1990,tao_nonlinear_2006,hillairet_smooth_2010, collot_type_2014,cazenave_solutions_2019,duyckaerts_soliton_2022,donninger_spectral_2023} and references therein for more details.
	
Given smooth Cauchy data for \cref{in:eq:main} that forms a singularity in finite time, either to the future or to the past, one can define $\mathcal{D}$ to be the maximal domain of existence and $\Sigma^{\pm} = \partial^{\pm} \mathcal{D}$ the future and past boundaries.
Let us discuss the geometry of $\Sigma^{\pm}$ in some detail.
In general, $\Sigma^{\pm}$ is only known to be of Lipschitz regularity (\cite{alinhac_blowup_1995}), see \cref{fig:singularity_picture}. For the sake of exposition, we focus on $\Sigma = \Sigma^{+}$ in the below discussion, with the past boundary $\Sigma^{-}$ being analogous.

We denote $\mathscr{S}\subset \Sigma$ to be the \emph{spacelike part} of $\Sigma$, given by points  with local Lipschitz constant strictly smaller\footnote{Provided that $\Sigma$ is more regular, say $C^1$, then this condition is easily seen to be equivalent to the usual requirement, that the normal is timelike.} than 1.
    It is not difficult to see that any two points $p_1,p_2\in\mathscr{S}$ become causally disconnected in finite time before the singularity, therefore in a geometric sense we cannot differentiate whether blow up happens at $p_1$ or $p_2$ first, even if with respect to the usual Minkowski time coordinate we can order these points.
	Therefore, connected components of $\mathscr{S}$ should be regarded as forming their singularities at the same time.
    
    We also introduce the noncharacteristic points in $\Sigma$, following \cite{merle_determination_2003}, given by points $(t_1,x_1)\in\mathscr{R}\subset\Sigma$ with the property that there is a truncated spacelike cone, $\{(t-t_1)=\epsilon\abs{x-x_1},t_0<t<t_1\}$ for some $\epsilon<1$, that is a subset of the domain of existence $\mathcal{D}$.
    A priori, we have the inclusion $\mathscr{S}\subset \mathscr{R}$.
    
    In one dimension, the works \cite{merle_determination_2003,merle_openness_2008,merle_existence_2009,cote_construction_2012} show that $\mathscr{R}$ is $C^1$ regular open subset of $\Sigma$ and characteristic points $\mathscr{R}^{\mathrm{c}}=\Sigma\setminus\mathscr{R}$ form a discrete set. Moreover, it holds that $\mathscr{I}=\mathscr{R}$, i.e.~\emph{all} noncharacteristic points are spacelike.
    In \cite{nouaili_1_2008}, the $C^1$ regularity of $\mathscr{R}$ was upgraded to have $C^{1, \alpha}$ regularity, for some $\alpha > 0$.
    In particular, they show:
    \begin{theorem}[\cite{merle_openness_2008,nouaili_1_2008}]\label{in:thm:MZ}
    	Let $\phi$ be a finite energy solution to \cref{in:eq:main}, i.e.: $(\phi|_{t=0},\partial_t\phi|_{t=0})\in \dot{H}^1\times L^2$ with maximum domain of existence $\D$.
    	Then the boundary of maximal existence $\Sigma^+=\partial^+\D$ is a non-empty Lipschitz curve with characteristic points forming a discrete subset of $\Sigma$ and the non-characteristic part of $\Sigma$ is open and $C^{1,\alpha}$ regular for some $\alpha>0$.
    \end{theorem}
    While similar results \cite{merle_n_setup, merle_n_stability} have been shown for perturbations of the model ODE blow-up solution, under a \emph{subconformal} restriction on $p$ and $n$, the phenomenology of $\mathscr{S}$ in higher dimensions is much richer.

	For instance, we briefly note that for \cref{in:eq:main}, many of the singularities illustrated on \cref{fig:singularity_picture} do occur in large spacial dimensions.
	In particular, non-characteristic, non-spacelike singularities exist, where $\Sigma^+$ is the forward light cone of a point locally.
    For example, for  $d\geq 5, p=3$ there exists $c_1(d),c_2(d) > 0$ such that \cref{in:eq:main} has the following self-similar solution, in which $\Sigma^+$ is the forward Minkowskian light cone emanating from $(0, 0) \in \R^{d+1}$, see \cite{glogic_co-dimension_2021}.
    \begin{align} \label{in:eq:self_similar}
        \phi=\frac{1}{t} \cdot \frac{c_1(d)}{1 + c_2(d) |x|^2 t^{-2}}.
	\end{align}
    The (codimension-one) stability of these  for $d = 7$ were also studied in \cite{glogic_co-dimension_2021}, as well as a similar example for the $p=2$ case in \cite{csobo_blowup_2024}.
	These have been suggested to be threshold solutions between the stable ODE blow-up and dispersive regimes \cite{glogic_threshold_2020}.
	We refer to these as \emph{locally naked singularities} -- following \cite{moortel_breakdown_2022} -- due to the presence of a Cauchy horizon\footnote{We remark that while the explicit solution \cref{in:eq:self_similar} is smooth across $\{r=t\}$, perturbed solutions retaining this qualitative picture may be of limited regularity across this hypersurface.}, which implies that a timelike observer can leave the domain of dependence of the initial data, even though the solution remains regular enough to ensure local well-posedness\footnote{Stronger null singularities, across which well-posedness fails are not known to exist for \cref{in:eq:main}.} along the observer’s trajectory. 
	The faster-than-self-similar singularities constructed in \cite{krieger_slow_2009,hillairet_smooth_2010} are also expected to belong to this class.
    
	In \cref{in:thm:main}, we are interested in constructing solutions with arbitrary smooth $\mathscr{S}$ regions.
	In \cite{kichenassamy_blow-up_1993} it was already shown that $\Box\phi=-e^{\phi}$ equation can have arbitrary analytic spacelike hypersurface as a blow up region.
	This was extended to finite  Sobolev regularity and then to the cubic $p=3$ case of \cref{in:eq:main} as well as other nonlinearities, \cite[Chapter 10]{kichenassamy_fuchsian_2007}.
	In this work, Kichenassamy also proves the analogue of \cref{in:thm:main2} using a notion of scattering data and the Nash-Moser theorem.
	In 1 dimension, a general construction in finite regularity for \cref{in:eq:main} was proved by Kilip-Visan \cite{killip_smooth_2011}.
	In higher dimensions, again many cases are known, even for $p\notin\N$ \cite{cazenave_solutions_2019,cazenave_solutions_2020}.

    Some of the above works are in low regularity setting (with initial data in the \emph{energy class} $\dot{H}^{1}\times L^2$), thus limiting the range of $p$ that can be covered for the construction.
    In contrast, we work in a high regularity setting, which is motivated by the corresponding regularity that we can be show for the stability problem in \cref{in:thm:main2}.
	That is, \cref{in:thm:main}  concludes the construction of spacelike singularity formation in arbitrary dimensions and powers $p$, in the class $C^N$ for $N \gg 1$.
    \begin{figure}[htbp]
		\centering
        \begin{tikzpicture}[scale=0.8]
        \begin{axis}[
            axis lines=none,
            domain=-6:6,
            samples=100,
            width=11cm,
            height=7cm,
            xmin=-5.5, xmax=5.5,
            ymin=-3.5, ymax=3.5,
            thick
        ]
        
        \fill [lightgray, opacity=0.5]
                (-5.5, 0.2) .. controls (-4.5, -0.1) .. (-4, -0.6)
                -- (-3, -1.6) -- (-2.3, -0.9)
                .. controls (-1.7, -0.2) and (-1.1, -0.7) .. (-0.5, 0)
                .. controls (0.1, -0.9) and (0.9, -0.5) .. (1.5, -1.6)
                .. controls (2.1, -0.8) and (2.6, -0.4) .. (3.0, -0.6)
                -- (4.2, 0.6)
                .. controls (4.9, 1.3) .. (5.5, 1.3)
                -- (5.5, -3.0)
                -- (-5.5, -3.0) -- cycle;
        
        \draw[red] (axis cs:-5.5, 0.2) .. controls (axis cs:-4.5, -0.1) .. (axis cs:-4, -0.6);
        \draw[blue] (axis cs:-4, -0.6) -- (axis cs:-3, -1.6);
        \draw[blue] (axis cs:-3, -1.6) -- (axis cs:-2.3, -0.9);
        \draw[red] (axis cs:-2.3, -0.9) .. controls (axis cs:-1.7, -0.2) and (axis cs:-1.1, -0.7) .. (axis cs:-0.5, 0);
        \draw[red] (axis cs:-0.5, 0) .. controls (axis cs:0.1, -0.9) and (axis cs:0.9, -0.5) .. (axis cs:1.5, -1.6) node [above right, midway] {$\mathscr{S}$};
        \draw[red] (axis cs:1.5, -1.6) .. controls (axis cs:2.1, -0.8) and (axis cs:2.6, -0.4) .. (axis cs:3, -0.6);
        \draw[blue] (axis cs:3, -0.6) -- (axis cs:4.2, 0.6) node [above left, midway] {$\mathcal{I}$};
        \draw[red] (axis cs:4.2, 0.6) .. controls (axis cs: 4.9, 1.3) .. (axis cs:5.5, 1.3);
        
        \node (q1) at (axis cs:-3, -1.6) [circle, draw, fill=white, inner sep=0.5mm] {};
        \node [below=0.2mm of q1] {$q_1$};
        \node (q2) at (axis cs:1.5, -1.6) [circle, draw, fill=white, inner sep=0.5mm] {};
        \node [below=0.2mm of q2] {$q_2$};
        \node (p) at (axis cs:-0.5, 0) [circle, draw, fill=white, inner sep=0.5mm] {};
        \node [above=0.2mm of p] {$p$};
        \end{axis}
        \end{tikzpicture}
        \captionsetup{justification = centering}
		\caption{Possible a priori structure of a singular surface $\Sigma^+$, with the solution existing in the grey region. In red, we have indicated the spacelike part $\mathscr{S}$; in blue, the null segments. We have also depicted a characteristic point $p$ in $\Sigma^+$, while $q_1$ and $q_2$ are examples of non-characteristic and non-spacelike singularities.}
		\label{fig:singularity_picture}
	\end{figure}
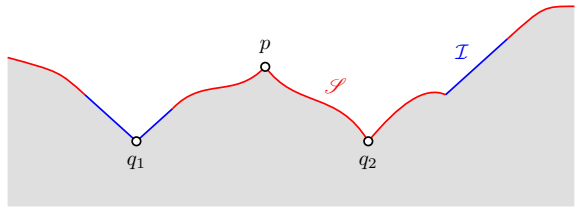

    On the other hand, in Theorem~\ref{in:thm:main2}, we are interested in the \emph{stability} and \emph{regularity} properties of the blow-up surface $\mathscr{S}$ upon perturbing initial data for \eqref{in:eq:main} away from $\Sigma$. We note that stability of a single blow-up point had previously been studied in \cite{Donninger2012, donninger_stable_2012, DonningerSupercritical2016, chatzikaleas_stable_2019}, concluding with a result for all $p > 1, d \in \N$ in \cite{Ostermann_stable_blowup}. However, these works leave open the understanding of the stability properties of the  blow-up curve $\Sigma$, and in particular do not show that the entire curve remains close to the original $\mathscr{S}$. The only previous works showing the later property are \cite{merle_openness_2008, nouaili_1_2008} for $d = 1$, where stability is shown in the $C^{1, \alpha}$ class, at least for perturbations of initial data in the energy class $\dot{H}^1 \times L^2$.

    In \cref{in:thm:main2}, for high regularity perturbations of a sufficiently regular ``background'' singular solution, we show that the blow-up curve remains $C^N$-close to the original blow-up hypersurface $\mathscr{S}$, for any $N \geq 0$, for all spacetime dimensions $d \geq 1$ and any power $p>1$. In particular, for smooth perturbations of the model ODE blow-up solution, the resulting blow-up curve is a smooth perturbation of the constant $t$-hypersurface we started with.

\subsection{Related equations}

We briefly review singularity formation for related nonlinear wave equations. While analogues of ODE blow-up -- and likely similar results to that of \cref{in:thm:main} and \cref{in:thm:main2} -- exist for many semilinear equations such as $\Box \phi = \pm (\partial_t \phi)^2$ or $\Box \phi = F(\phi) |\nabla \phi|_{\eta}^2$, we comment that
spacelike singularity formation is not known to occur for many other well-studied semilinear wave equations, such as the (hyperbolic) Yang--Mills equations or the wave-maps equations\footnote{Under the assumption of spatial homogeneity, the wave maps reduce to the equation describing geodesic motion, there is no ODE blow-up for closed target manifolds, and thus one expects no ``spacelike'' singularities, and similarly for all equations derived from any Hamiltonian system with non-negative energy functional.}.

Indeed, the most well-studied case of spacelike singularity formation for a nonlinear wave equation is for the Einstein equations possibly coupled to suitable matter models, see for instance 
\cite{fournodavlos_stable_2020,fournodavlos_asymptotically_2022,li_scattering_2024,li_kasner_2025}, where the expected behaviour has a much richer phenomenology than we could discuss here, see \cite{luk_singularities_2023}. In particular, even restricted to the spatially homogeneous setting, the associated ODE can produce interesting non-trivial, and in some settings even chaotic, behavior \cite{bkl_chaotic_1971, ringstrom_bianchi_2001}. 

Regarding other types of singularity formation for the aforementioned related equations, there is an abundant literature in both (asymptotically) self-similar and faster than-self-similar singularity formation.
For the wave-map equation from $\R^{d+1}$ to $\mathbb{S}^d$ for $d\geq3$, explicit self-similar solutions \cite{shatah_weak_1988,biernat_generic_2015} are known, which are smooth away from the point of singularity $(0,0)$, in a similar way to \eqref{in:eq:self_similar}.
That is, for this solution, $\Sigma$ is composed of a single non-characteristic point and a smooth null cone along which the solution may be smoothly extended.
The work of \cite{biernat_hyperboloidal_2019} indicate that this geometry for $\Sigma$ is stable, and thus is the ``generic'' form of singularity formation, in comparison to the focusing wave equation studied in the present article. 
Self-similar solutions have also been found for the Euler equation in \cite{merle_smooth_2019}.
Faster than self-similar singularity formation is also proved for many other hyperbolic problems as in \cite{krieger_renormalization_2008,raphael_stable_2009}.
For singularity formation to more general evolutionary equations, we refer the reader to \cite{raphael_quantized_2014,merle_blow_2019-1,kim_rigidity_2024} and references therein.

\subsection{Ideas of the proof}\label{in:sec:ideas}
We illustrate many of the ideas of the proof via the following ODE problem, which corresponds to spatially homogeneous solutions of \eqref{in:eq:main}, with methods that are translatable into the PDE setting:
\begin{equation}\label{ode:eq:ode}
	\partial_t^2\phi=\phi^p.
\end{equation}
Note that the model solution \eqref{in:eq:model} is also a solution to this ODE.

Next, upon considering $\phi = c_p t^{-\alpha_p}+ \psi t^{\beta_p}$ as an ansatz, where $\psi \in \R$, we observe that only the exponent $\beta_p$ defined in \eqref{in:eq:betadelta} yields a possible lower order perturbation. 
    A generic solution 
    that forms a singularity can be written as
	\begin{equation}\label{ode:eq:generic_sol}
		\phi=c_p (t-t_0)^{-\alpha_p}+\psi(t-t_0)^{\beta_p}+\mathcal{O}((t-t_0)^{\beta_p'}),\quad \beta_p'\coloneqq -\alpha_p+(\floor{2\kappa_p+1}+2\kappa_p)/2.
	\end{equation}
    Here, $t_0\in\R$ is the blow-up time, $\psi\in\R$ is the remaining scattering data and $\mathcal{O}$ term has infinite regularity with respect to $t\partial_t$, meaning that for all $N \in \N$, the $(t \partial_t)^N$ derivative of the remaining error term decays at the rate $O((t-t_0)^{\beta_p'})$.
	
    We discuss how to construct solutions that have the form \cref{ode:eq:generic_sol}, in analogy with the scattering construction of \cref{in:thm:main}.
    We fix $t_0=0$, and keep $\psi \in \R$ arbitrary.
    Writing $\phi=\phi_0+\bar{\phi}=c_p t^{-\alpha_p}+\bar{\phi}$, the ODE \eqref{ode:eq:ode} reduces to
	\begin{equation}\label{ode:eq:lin=nonlin}
        (\partial_t^2-\gamma_p t^{-2})\bar{\phi}=\mathcal{N}[\bar{\phi}], \quad \text{ where } \mathcal{N}_{\phi_0}[\bar{\phi}]=\phi^p- c_p^p t^{-p \alpha_p}-\gamma_p t^{-2}\bar{\phi}.
	\end{equation}
    Note that $\mathcal{N}_{\phi_0}$ is smooth in $\bar{\phi}$ for $\abs{\bar{\phi}/\phi_0}<1/2$, and vanishes quadratically: $|\mathcal{N}_{\phi_0}[\bar{\phi}]| \lesssim t^{-2 + \alpha_p} \bar{\phi}^2$.
	
	\paragraph{Ansatz:}
    The first step in our scattering instruction is that, for fixed $\psi$ and $N\in\R$ arbitrary, one can construct approximate solutions $\phi_N\in C^{\infty}((0,1))$ of the form \cref{ode:eq:generic_sol} satisfying
    \begin{equation} \label{ode:eq:approx}
		f_N\coloneqq\partial_t^2 \phi_N-\phi^p_N=\mathcal{O}(t^N).
	\end{equation}
    Roughly speaking, this follows from the fact that one can insert first $\bar{\phi} = \psi t^{ \beta_p}$ into \eqref{ode:eq:lin=nonlin}, which would generate an inhomogeneity $\mathcal{N}[\psi t^{- \beta_p}] = O(t^{\beta_p'})$. But, 
    from the fact that, on functions of the form $t^a$ for $a>\beta_p$, the linearised operator on the left hand side of \cref{ode:eq:lin=nonlin} is invertible, we can repeatedly correct the nonlinear term until we have a $O(t^N)$ error.
	In particular, we note the following computations
	\begin{subequations}\label{ode:eq:invert}
		\begin{align}
			(\partial_t^2-\gamma_pt^{-2})t^{q}=\big(q(q-1)-\gamma_p\big)t^{q-2},\\
			(\partial_t^2-\gamma_pt^{-2})t^{\beta_p}\log t=(2\beta_p-1)t^{\beta_p-2}.
		\end{align}
	\end{subequations}

	\paragraph{Exact solution:}
    Given an approximate solution $\phi_N$ solving \eqref{ode:eq:approx} for $N$ sufficiently large, we obtain an actual solution $\phi$ by an energy estimate.
	We write $\phi=\phi_N+\bar{\phi}$, so that
	\begin{equation}\label{ode:eq:energy_eq}
        (\partial_t^2 - \gamma_p t^{-2}) \bar{\phi} = \Big(f_N + \mathcal{N}_{\phi_N}[\bar{\phi}] - \underbrace{(p \phi_N^{p-1} - \gamma_p t^{-2})}_{\mathcal{O}(t^{-2 + 2 \gamma_p})} \bar{\phi}\Big).
	\end{equation}
    where the nonlinearity $\mathcal{N}_{\phi_N}[\bar{\phi}]$ is defined by $\mathcal{N}_{\phi_N} [\bar{\phi}] = \phi^p - \phi_N^p - p \phi_N^{p-1} \bar{\phi}$ and satisfies a quadratic bound $|\mathcal{N}_{\phi_N}[\bar{\phi}]| \lesssim t^{- 2 + \alpha_p} |\bar{\phi}|^2$.

	We can now construct $\bar{\phi}$ by a limiting argument.
	Truncating $f_N$ by a cutoff function $\chi(t/2^{-n})$, we solve \cref{ode:eq:energy_eq}  with prescribing $\bar{\phi}=0$ for $t \ll 2^{-n}$, and proving uniform estimates such as $\bar\phi=\mathcal{O}(t^N)$, then taking $n \to \infty$ and using a suitable compactness argument.

	To get such a uniform estimate, one may (i) multiply \cref{ode:eq:energy_eq} with $t^{-q}\partial_t\bar{\phi}$ and integrate by parts, then (ii) integrate $\alpha\partial_t (t^{-q-2}\bar{\phi}^2)$ by parts, and thus obtain
	\begin{nalign}\label{ode:eq:energy_est}
        \int_{t_0}^{t_1} t^{-q-1}q(\partial_t\bar{\phi})^2-t^{-q-3}(q+2)\gamma_p\bar\phi^2+\big[t^{-q}(\partial_t\bar{\phi})^2-\gamma_p t^{-q-2}\bar{\phi}^2\big]_{t_0}^{t_1}&=\int_{t_0}^{t_1} t^{-q}\bar{f}_N\partial_t \bar{\phi}, \\
       \alpha\int_{t_0}^{t_1} t^{-q-3}(q+2)\bar\phi^2-2t^{-q-2}\bar{\phi}\partial_t\bar{\phi}+\big[\alpha t^{-q-2}\bar{\phi}^2\big]_{t_0}^{t_1}&=0,
	\end{nalign}
	where $\bar{f}_N$ is the right hand side of \cref{ode:eq:energy_eq}.
    Let us take $q>2\beta_p$ so that $q(q+2)>16\gamma_p$, and we can take pick $\alpha$ satisfying $4\alpha^2<q(q+2)(\alpha-\gamma_p)$.
	Summing the two equations with this choice yields a coercive control over $\bar{\phi}$ on the left hand side.
	Using Cauchy-Schwartz gives the control in terms of $\bar{f}_N$.

	Applying Sobolev embedding, we can also absorb the terms $\bar{\phi}$ dependent terms of $\bar{f}_N$ into the left hand side as long as $t_1$ is sufficiently small.
	Therefore, we obtain uniform boundedness for $\bar{\phi}$, and the existence of a solution follows by compactness.
    
	Finally, let us note, that we can commute \cref{ode:eq:energy_eq} by $t\partial_t$, and since $f_N$ has extra regularity with respect to these vectorfields, so does $\bar\phi$.

    \paragraph{Uniqueness:} 
    Assume there exists two solutions $\phi_1,\phi_2$ with the same blow-up time $t_0=0$, and scattering data $\psi$.
    By construction, the ansatz part of $\phi_1,\phi_2$ must coincide.
    Writing $\bar{\phi}=\phi_1-\phi_2$, we conlcude that $\bar{\phi}=\mathcal{O}(t^N)$ for any $N$.
    Then, we write
    \begin{equation}
    	(\partial_t^2-\gamma_p t^{-2})\bar{\phi}=(\phi_1+\bar{\phi})^p-\phi_1^p-p\phi_1^{p-1}\bar{\phi},
    \end{equation}
    and apply the energy estimate from \cref{ode:eq:energy_est} to obtain
    \begin{equation}
    \norm{t^{-N}\bar{\phi}}_{L^2(0,t_1)}\lesssim \norm{t^{-N+1}\bar{\phi}}_{L^2(0,t_1)}^2.
    \end{equation}
    Taking $t_1$ sufficiently small yields that $\bar{\phi}=0$.
    This completes the sketch of \cref{in:thm:main}.

    \paragraph{Aside, solutions \emph{without} scattering data:}
    We present an alternative way to obtain \emph{non-unique} solutions to \cref{ode:eq:ode} forming an ODE type singularity at $t=0$ for a better comparison to previous works in the literature, such as \cite{cazenave_solutions_2019,cazenave_solutions_2020}.
    We may write the linearisation \cref{ode:eq:lin=nonlin} as
    \begin{equation}\label{in:eq:ode_product}
        t\partial_t t^{2\beta_p-1}t\partial_t t^{-\beta_p}\bar\phi=t^{\beta_p+1}\bar{f},
    \end{equation}
    where $\bar{f}=\mathcal{N}[\bar\phi]$.
    Note that the kernel of the linearised operator on the left is the span of $\{t^{\beta_p}, t^{-\beta_p+1}\}$, where $\beta_p-1>\alpha_p$.
    Solving \cref{in:eq:ode_product} for $\bar{f}=\mathcal{O}(t^{-c-2})$ with $c<\alpha_p$, and looking for a solution $\bar\phi$ with $ t^{-\beta_p+1}$ part vanishing, we obtain
    \begin{equation}
        t\partial_t t^{-\beta_p}\bar\phi=t^{-2\beta_p+1}\int_0^t\frac{\dd \bar{t}}{\bar{t}}\bar{t}^{\beta_p+1}\bar{f}(\bar{t})=\mathcal{O}(t^{-c-\beta_p}).
    \end{equation}
    Posing initial data for $\bar\phi$ at some $t_1$, we obtain that
    \begin{equation}
        \phi=t^{\beta_p} (t^{-\beta_p}\phi)|_{t=t_1}-\int_t^{t_1}\frac{\dd \tilde{t}}{\tilde{t}}\tilde{t}^{-2\beta+1}\int_0^{\tilde{t}}\frac{\dd \bar{t}}{\bar{t}}\bar{t}^{\beta_p+1}\bar{f}(\bar{t}),
    \end{equation}
    where the second part is in $\mathcal{O}(t^{-c})$.
    Using that $\mathcal{N}:\mathcal{O}(t^{-c})\to\mathcal{O}(t^{-c-2+(\alpha_p-c)})$ for $c<\alpha_p$, one can use the above to construct a (non-unique) solution forming a singularity at 0.

	\paragraph{Stability:}
    We next show how to prove stability for the singularity formation of the form \cref{ode:eq:generic_sol} upon applying perturbations away from $t = t_0$, i.e~the homogeneous analogue of \cref{in:thm:main2}. 
	Instead of studying \cref{ode:eq:ode} directly, we introduce the variable $\tau$ given implicitly by $c_p \tau^{-\alpha_p} =\phi$, and the quantity $\mathring{\Omega}^2 = (\partial_t\tau)^2 -1$.
	For such quantities, the expansion \cref{ode:eq:generic_sol} implies
    \begin{subequations}
        \begin{gather} \label{ode:eq:tau_asymp}
            \tau=(t - t_0) -\frac{\psi}{\alpha_pc_p}(t-t_0)^{2\kappa_p+1}+\mathcal{O}((t - t_0)^{2\kappa_p+2}),\\[0.3em] \label{ode:eq:omega_asymp}
        \text{and }\quad\mathring{\Omega}^2 =-\frac{2 \psi(2\kappa_p+1)}{c_p\alpha_p}t^{2\kappa_p}+\mathcal{O}(t^{2\kappa_p+1}).
    \end{gather}
\end{subequations}
	
    A straightforward computation using \cref{ode:eq:ode} yields the equation
	\begin{equation}\label{ode:eq:Omega}
        \tau \frac{\partial}{\partial \tau} \mathring{\Omega}^2 = 2 \kappa_p \mathring{\Omega}^2,
	\end{equation}
    which yields that $\mathring{\Omega}^2$ is in fact an exact multiple of $\tau^{2 \kappa_p}$, say $\hat{\psi} \tau^{2 \kappa_p}$. This implies that
    \[
        \frac{\partial t}{\partial {\tau}} = \sqrt{ 1 + \hat{\psi} \tau^{2 \kappa_p} },
    \]
    which -- upon applying a suitable inverse function theorem -- can be used to recover \eqref{ode:eq:tau_asymp} and \eqref{ode:eq:omega_asymp} for $t_0$ and $\psi$ suitably related to the (perturbed) initial data for $\tau(t)$ away from the singularity.

    \paragraph{Overcoming derivative loss:}
    In the PDE problem studied in \cref{in:thm:main2}, we do not simply study the ODE \eqref{ode:eq:Omega} but instead a PDE that we caricature as
    \[
        \tau \partial_{\tau} \mathring{\Omega}^2 = 2 \kappa_p \mathring{\Omega}^2 + \tau \partial_z J,
    \]
    where one expects $J$, which represents components of the Jacobian matrix of the transformation from $(t, x)$ to $(\tau, x)$ coordinates, to be smooth and bounded. As a result, this additional term is $\mathcal{O}(\tau)$ and is a ``good error term'' from the perspective of the ODE.

    However, the fact that a $\partial_z$ appears in this error term means that one must first ``overcome derivative loss'' and justify why one is allowed to estimate this term as $O(\tau)$. This is done using energy estimates, via equations that we caricature for now as
    \[
        \partial_{\tau}^2 \partial_z J = \kappa_p \tau^{-2} \mathring{\Omega}^2 \partial_z J + \kappa_p J \tau^{-1} \partial_{\tau} \partial_z J + \cdots.
    \]
    In the stability problem, we must integrate such an equation backwards in $\tau$. Indeed, we use a $\tau^{q} \partial_{\tau}$ multipler and integrate in parts to show that, for $q > 0$ chosen sufficiently large, $\tau^q \partial_z J$ is bounded.
    
    Unfortunately, this is not good enough to show that $\tau \partial_z J = O(\tau)$ as we previously desired. One instead closes the argument by showing that one can use similar energy estimates to show that $\tau^q \partial_z^K J$ is bounded, for $q > 0$ chosen independently of the order $K$. Then one may close the argument by using an interpolation argument to estimate $\partial_z J$ in terms of $J$ and $\partial_z^K J$, which is enough to yield $\tau \partial_z J = O(\tau^{1 - \frac{q}{K} -})$, which is also a ``good error term'' from the persepective of the ODE, so long as $K$ is chosen sufficiently large. This roughly completes our sketch of \cref{in:thm:main2}.

\subsection{Precise statements}
Let us state \cref{in:thm:main,in:thm:main2} precisely, explaining the role played by the additional scattering data $\psi$.

\begin{theorem}\label{in:thm:main1_precise}
    Let $s,f,\psi,\tstar,\tstar_1,\Mt$ be as in \cref{in:thm:main} and let $\Sigma_{\tilde{\tstar}}=\{\tstar=\tilde{\tstar}\}\cap\Mt$.
    Then, there exists a \underline{unique} solution of \cref{in:eq:main} obeying the following asymptotics:
    \begin{equation}\label{in:eq:expansion}
        \phi=\phi_0+\tstar^{-\alpha_p}\sum_{1 \leq j\leq\floor{2\kappa_p}}\tstar^jH^{s-j}(\Sigma_{\tstar})+\psi \tstar^{\beta_p}+\mathfrak{e}+\begin{cases}
            \tstar^{\beta_p}\log \tstar H^{s-2\kappa_p}(\Sigma_\tstar)&\text{ if }2\kappa_p\in\N\\
            0& \text{ if } 2\kappa_p\notin\N
        \end{cases}
    \end{equation}
    where $\phi_0=c_p(1-\abs{\partial f}^2)^{\frac{1}{p-1}}\tstar^{-\alpha_p}$ and $\mathfrak{e}\in \tstar^{\beta_p'}L^{\infty}((0,\tstar_1); H^{s-12\kappa_p}(\Sigma_\tstar))$ for some $\beta_p'>\beta_p$.
\end{theorem}

\begin{remark}[Existence interval]
    Although the model ODE solution \cref{in:eq:model} exists in the region $\Mt$ with any value of $\tstar_1 > 0$ (i.e.~it exists globally to the future), it is easy to argue that one expects a degeneration $\tstar_1 \to 0$ for large values of $\psi$ and $|\partial f|$. 
    Even in the spatially homogeneous setting with $f \equiv 0$ and $\psi \equiv \psi_{\mathrm{const}} \neq 0$, one may use the ODE \eqref{ode:eq:ode} to find that the time of existence scales as $\tstar_1 \sim \psi_{\mathrm{const}}^{-\frac{p-1}{p+1}}$.
\end{remark}

\begin{remark}[Regularity]
    The number of derivatives in \cref{in:thm:main1_precise} is far from sharp and we expect that already within the framework of the present article a smaller value can be achieved.
    Furthermore, note that the constructed solution in \cref{in:eq:expansion} has $s-12\kappa_p$ regularity, whereas we require $s>30\kappa_p+\frac{n+1}{2}+5$.
    The discrepancy is due to the way we prove uniqueness.
\end{remark}

\begin{remark}[Solutions without scattering data]
    Notice that in \cref{in:thm:main1_precise}, we need to use $s\sim (p-1)^{-1}$ number of derivatives as $p\to1$.
    Following an approach similar to \cite{cazenave_solutions_2019}, it is possible to construct solutions with bounded number of derivatives even in the $p\to1$ limit, using a compactness argument as done around \cref{in:eq:ode_product}.\footnote{In particular, the energy estimates of \cref{back:sec:energy} imply a bound $\norm{t^{\alpha}\Vb^k\mathring{\phi}}_{L^2}\lesssim \norm{t^{\alpha}\mathring{\phi}}_{L^2}$ for $\alpha>\alpha_p-1/2$ and $\mathring{\phi}=\phi-\phi_0$. We can combine this with an invertibility statement of $(\partial_t^2-p\phi_0^{p-1})$ in $t^{\alpha}L^2$ for $\alpha\in(\alpha_p-1/2,\beta_p-1/2)$ to improve the estimate to $\norm{t^{\alpha}\Vb^k\mathring{\phi}}_{L^2}\lesssim \norm{t^{\alpha-1}\mathring{\phi}}_{L^2}$.
    Existence follows by compactness.}
    However, these solutions are \emph{not} characterised by scattering data $\psi$, and thus are not unique as shown by \cref{in:thm:main1_precise}.
    We believe that for a scattering theory and corresponding uniqueness
    the $s\to\infty$ degeneration for $p\to1$ is necessary.
\end{remark}

\begin{theorem}\label{in:thm:main2_precise}
    Let $s,\epsilon_*$ be as in \cref{in:thm:main2}.
    Then, there exist $f\in H^{s-200\kappa_p}(\BB_4)$ and $\psi\in H^{s-202\kappa_p}(\BB_4)$, such that the unique solution $\phi$ to \cref{in:eq:main} with initial data $\phi|_{t=1}=\phi_0$ and $\partial_t\phi|_{t=1}=\phi_1$ extends to $\M = \{ |x| < 3(1 + t),  f(x) < t \leq 1 \}$, such that, for $\tstar = t - f(x)$, $\phi$ satisfies an asymptotic expansion of the following form:
    \begin{equation}\label{in:eq:expansion2}
        \phi=\phi_0+\tstar^{-\alpha_p}\sum_{j\leq\floor{2\kappa_p}}\tstar^jH^{s-200\kappa_p-j}(\Sigma_{\tstar})+\psi \tstar^{\beta_p}+\mathfrak{e}+\begin{cases}
            \tstar^{\beta_p}\log \tstar H^{s-2\kappa_p}(\Sigma_\tstar)&\text{ if }2\kappa_p\in\N,\\
            0&\text{ if } 2\kappa_p \notin\N,
        \end{cases}
    \end{equation}
    where $\phi_0=c_p(1-\abs{\partial f}^2)^{\frac{1}{p-1}}(1-f)^{-\alpha_p}$ and $\mathfrak{e}\in \tstar^{\beta_p'}L^{\infty}((0,\tstar_1); H^{s-202\kappa_p-1}(\Sigma_\tstar))$ for some $\beta_p'>\beta_p$.
    
    For $(\phi_0^{(i)},\phi_1^{(i)})$ with $i\in\{1,2\}$ satisfying \cref{in:eq:forward_ass} it furthermore holds that
    \begin{equation}\label{in:eq:sing_Cauchy}
        \norm{f^{(1)}-f^{(2)},\psi^{(1)}-\psi^{(2)}}_{H^{s -2 - 200\kappa_p} \times H^{s-2-202\kappa_p}}\lesssim\norm{\phi_0^{(1)}-\phi_0^{(2)},\phi_1^{(1)}-\phi_1^{(2)}}_{H^{s+2}\times H^{s+1}}.
    \end{equation}
\end{theorem}

\begin{remark}[Spectral theory]
    We note that the work \cite{Ostermann_stable_blowup}, which is the culmination of a series of works \cite{Donninger2012, donninger_stable_2012, chatzikaleas_stable_2019},  also establishes the stability of the ODE blow-up as in \cref{in:thm:main2}, or more precisely the stability of a single blow-up point on the surface. 
    The approach of \cite{Ostermann_stable_blowup} is to study a (non self-adjoint) spectral problem in suitable self-similar coordinates near the blow-up point and to establish a spectral gap in order to ascertain stability, which can be down at a much smaller regularity assumption of $s > \frac{n}{2}$.\footnote{A similar spectral problem also appears for the study of \emph{naked singularities} as described in \cref{in:sec:background}, but we expect that in that case the spectrum there has physical consequences, such as the \emph{limited} regularity of the Cauchy horizon.}
    In our approach, we are not required to study a spectral problem, at least beyond understanding the stability of fixed points of ODEs.
\end{remark}

Finally, since \cref{in:thm:main1_precise,in:thm:main2_precise} can be translated to purely local statements, we obtain the stability of any spacelike singularity of ODE blow-up type, with asymptotics as in \cref{in:thm:main}.
See \cref{fig:cor1.6} for an illustration.

\begin{cor}\label{in:cor:general}
	Let $s\geq s_0 \coloneqq 232\kappa_p+10+n$, and $f\in H^s(\BB_2),\psi\in H^{s-\floor{2\kappa_p}}(\BB_2)$ with
	\begin{equation}\label{in:eq:cor_assumptions}
		\abs{\partial f}<1/10,\qquad \norm{f}_{H^s},\norm{\psi}_{H^{s-\floor{2\kappa_p}}}\leq1,
	\end{equation}
	and let $\phi,\tstar,\M_{\tstar_1}$ be as in \cref{in:thm:main}.
	Then there exists a constant $\epsilon>0$, depending on $\phi$,  such that the following holds:
    Let $\phi_0, \phi_1$ be functions on $\Sigma_{\tstar_1/2}=\{\tstar=\tstar_1/2\}$ satisfying
	\begin{equation}\label{in:eq:gen_pert}
		\norm{(\phi_0-\phi)|_{\Sigma_{\tstar_1/2}}}_{H^{s-30\kappa_p+1}}+\norm{(\phi_1-\partial_t\phi)|_{\Sigma_{\tstar_1/2}}}_{H^{s-30\kappa_p}}\leq\epsilon_{s},
	\end{equation}
	with $\epsilon_s<\epsilon$. 
    Then there exists $\bar{f},\bar{\psi}\in H^{s-232\kappa_2}(\BB_2)$ and a solution, $\bar{\phi}$, of \cref{in:eq:main} with initial data $(\phi_0,\phi_1)$ in $\Mhat=\{\abs{x}\leq 5(t+1/3)\}\cap \{\tstar\leq\tstar_1/2\}\cap\{t>\bar{f}\}$ with $\bar{\phi}$ exhibiting ODE blow up towards $\{t=\bar{f}\}$ with auxiliary scattering data $\bar\psi$ which moreover satisfy
    \begin{equation}\label{in:eq:close}
        \norm{f-\bar{f},\psi-\bar\psi}_{H^{s-232\kappa_p}}\lesssim\epsilon_s.
    \end{equation}
\end{cor}

We also obtain the smoothness of \emph{all} spacelike singularities (with a regular enough singularity formation) that start from smooth initial data:
\begin{cor}\label{in:cor:smoothness}
	Let $s_0=202\kappa_p+10+\frac{n}{2}$ and $f\in H^{s_0-200\kappa_p}(\BB_2)$ with $f(0)=0$ and $\abs{\partial f}<1/10$.
	Let $\phi$ be a solution of \cref{in:eq:main} in $\M = \{ |x| < 3(1 + t),  f(x) < t \leq 1 \}$ with an expansion as in \cref{in:eq:expansion}.
	Assume furthermore that on a spacelike hypersurface $\Sigma\subset \M$ it holds that $(\phi,\partial_t\phi) \in H^{s+1}(\Sigma)\times H^{s}(\Sigma)$.
	Then the following improved regularity holds: $f\in H^{s-200\kappa_p}(\BB_4)$.
\end{cor}

In fact, it suffices to assume that $f\in C^{1,\alpha}(\BB_2)$ for some $\alpha>0$, and that $\phi$ is a solution of \cref{in:eq:main} with an expansion as in \cref{in:eq:phi_scat_rough}.

\begin{cor}\label{in:cor:alpha}
	Let $f\in C^{1,\alpha}(\BB_4)$ for some $\alpha>0$ with $f(0)=\partial f(0)=0$ as well as $\abs{\partial f}<1/10$, and set $\tstar=t-f(x)$.
	Let's set $\tilde{\phi}_0=c_p\tstar^{-\alpha_p}(1-\abs{\partial f}^2)^{\frac{1}{p-1}}$ and $\phi$ a solution to \cref{in:eq:main} in $\M=\{\abs{x}\leq(3(1+t)), f(x)\leq t\leq1\}$ satisfying
    \begin{equation} \label{eq:weak_asymptotics}
		\phi=\tilde{\phi}_0+\tstar^{-\alpha_p+\alpha}L^\infty(\M). 
	\end{equation}
    Suppose, for some $\frac{1}{2} \leq \tilde{t} \leq 1$, that upon restricting $\phi$ to the hypersurface $\Sigma = \{ t = \tilde{t} \} \cap \mathcal{M}$, one has $\phi|_{\Sigma} \in H^{s+1}, \partial_t \phi |_{\Sigma} \in H^s$ with $s\geq (200\kappa_p+5+\frac{n}{2})\left(1+\frac{2n}{\alpha}\right)$, then $f\in H^{s-200\kappa_p}(\BB_4)$.
\end{cor}

Hence, using \cref{in:thm:MZ}, we may conclude that the singular surface in $\R^{1+1}$  is smooth up to a discrete set of points:
\begin{cor}
	Let $d=1$ and $\phi$ be a smooth finite energy solution to \cref{in:eq:main}, i.e.: $(\phi|_{t=0},\partial_t\phi|_{t=0})\in \dot{H}^1\times L^2\cap C^\infty\times C^\infty$.
	Then the non-characteristic part of the singular boundary is smooth.
\end{cor}

\subsection{Outlook and overview}
We conclude the introduction by discussing several open questions that we do not address, but hope that some of the techniques and ideas developed here might be relevant towards, and give an overview for the rest of the paper.

\paragraph{Maximal globally hyperbolic development of finite energy solutions:}
Given any inextendible spacelike hypersurface $\Sigma$, \cref{in:thm:main} gives a way to construct a solution that forms an ODE type singularity towards $\Sigma$.
However, such solutions, as already noted in \cite{cazenave_solutions_2020}, need not be of finite energy.
We thus leave the construction of solutions to \cref{in:eq:main} of finite $H^1 \times L^2$ energy forming an ODE type singularity along an inextendible spacelike hypersurface as the following open problem:
\begin{conjecture}
    There exists an inextendible spacelike hypersurface $\Sigma\subset\R^{n+1}$ and finite energy smooth Cauchy data $(\phi_0,\phi_1)\in H^1(\R^n)\times L^2(\R^n) \cap C^{\infty}(\R^n) \times C^{\infty}(\R^n)$ such that the solution to
	\begin{equation}\label{in:eq:ivp}
		P[\phi]=0,\qquad (\phi|_{t=0},\partial_t\phi|_{t=0})=(\phi_0,\phi_1)
	\end{equation}
	forms an ODE type singularity along all of $\Sigma$.
\end{conjecture}

\paragraph{Threshold behaviour:} 
To illustrate what we mean by threshold behaviour, let us consider the specific case of the cubic focusing nonlinear wave with $p=3$ in \eqref{in:eq:main}, in \emph{supercritical dimensions}\footnote{See \cite{beceanu_center-stable_2014,krieger_center-stable_2015} for examples of threshold behaviour in energy critical problems.} $n \geq 5$. In this case, it is straightforward to show that the trivial solution $\phi \equiv 0$ is stable against smooth and localised perturbations of the initial data $(\phi|_{t=0}, \partial_t \phi|_{t = 0}) = (0, 0)$, meaning that for small and localised initial data solutions to $P[\phi] = 0$ exist globally in time and scatter to a linear wave. On the other hand, for sufficiently large (though still finite energy) initial data, we expect ODE type singularity formation, for example the explicit model ODE blow up solution \eqref{in:eq:model}; further, this qualitative blow-up behaviour is also stable to small and localised perturbations of initial data, as proven in \cref{in:thm:main2}.

An interesting problem is to understand the \emph{transition} between the open small data regime (which exhibits global in time existence and scattering) and the open large data regime (which, in our case, means ODE-type blow up). That is, we study the dynamics of a one parameter family of initial data interpolating between the small data regime (e.g.~trivial initial data) and the large data regime (e.g.~that of the model ODE blow-up). Following \cite{glogic_co-dimension_2021, glogic_threshold_2020}, we make the following conjecture, with the nontrivial self-similar solution \eqref{in:eq:self_similar} being a threshold between the two regimes.

\begin{conjecture}
    There exists a smooth\footnote{By smoothness, we mean this family is both smooth in the spatial variable $x \in \R^n$ and with respect to the interpolating parameter $\omega$. We will put the initial data at $t = 1$ in analogy the setup of \cref{in:thm:main2}, where the singularity is near $t = 0$.} one-parameter family $(\phi_{0, \omega}, \phi_{1, \omega})$ with $\omega \in [0, 1]$, serving as initial data for the nonlinear wave $P[\phi] = 0$ with $p = 3$ and $n \geq 5$ in the sense $(\phi|_{t = 1}, \partial_t \phi|_{t = 1}) = (\phi_{0, \omega}, \phi_{1, \omega})$, such that the following hold, for some $\omega^* \in (0, 1)$:
    \begin{enumerate}
        \item We have $(\phi_{0,0}, \phi_{1, 0}) = (0, 0)$ and, for $0 \leq \omega < \omega^*$, the corresponding solution exists globally in time to the past.
        
        \item We have $(\phi_{0, 1}, \phi_{1, 1}) \equiv (c_3, - \alpha_3 c_3)$, at least in some sufficiently large ball and, for $\omega^* < \omega \leq 1$, the corresponding solution to $P[\phi] = 0$ blows up to the past, and some portion of the past boundary of the maximal domain of existence is a smooth spacelike hypersurface $\Sigma^-_{\omega}$ along which the solution exhibits ODE-type blow up.

        \item At the threshold $\omega = \omega^*$, the past development of the data $(\phi_{0, \omega^*}, \phi_{1, \omega^*})$ blows up at the point $(t, x) = (0, 0)$, and, within the future light cone of this point, the solution approaches the nontrivial self-similar solution \eqref{in:eq:self_similar} as $t \downarrow 0$. The solution should exhibit locally naked singularity formation.
    \end{enumerate}
\end{conjecture}

We refer the reader to \cite{glogic_co-dimension_2021, glogic_threshold_2020} for further commentary on such threshold behaviour, which is also relevant to other related semilinear equations, as well as the phenomenon of gravitational collapse and naked singularity formation for the Einstein equations, see \cite{gundlach_critical_2007}.

\paragraph{Overview:}
We begin by introducing the analytic and geometric framework for both the scattering construction and the stability problem in \cref{sec:setup}.
We prove \cref{in:thm:main1_precise,in:thm:main2} in \cref{sec:backward,sec:forward} respectively.
Finally, we prove \cref{in:cor:smoothness,in:cor:alpha} regarding the more general spacelike singularities in \cref{sec:local}.

\paragraph{Acknowledgements}

The authors would like to thank Birgit Sch\"orkhuber for helpful comments on the manuscript. Part of the work was completed while the first author was at the Princeton Gravity Initiative and the second author was a graduate student at Princeton University.
The first author acknowledges the support of the SNSF starting grant TMSGI2 226018. 
We thank the Erwin Schr\"odinger Institute for Mathematics and Physics at the University of Vienna for hosting us at the workshop on ``Nonlinear Waves and Relativity'', June 17--21 2024, where this work was initiated.

\section{Geometric and Analytic Setup} \label{sec:setup}
In this section, we introduce the geometric formalism and function spaces that will be used in the proof of both the scattering construction of \cref{in:thm:main} and the stability result of \cref{in:thm:main2} in \cref{no:sec:backward,sub:forward_setup} respectively.
As already discussed in the introduction, stability is proved in an alternative $(\tau,z)$ coordinate system, and we connect these to $(x,t)$ in \cref{no:sec:compare}

Throughout the present, we only use integer order regularity $s,K\in\N$, often without explicitly stating this restriction.

\subsection{Setup for the scattering construction}\label{no:sec:backward}

Consider a function $f\in C^\infty(\R^n)$ satisfying the following conditions:
\begin{equation}
	\sup \abs{\partial f}<1/10,\quad f(0)=0.
\end{equation}
We set $\tstar = \tstar(t, x) =t-f(x)$ and introduce the following spacetime regions for some $\tstar_1 > 0$:
\begin{nalign} \label{not:eq:region_back}
	\Mt\coloneqq\{\tstar<\tstar_1,\abs{x}<2(1-t)\},\quad  \Mt(s)\coloneqq\{\tstar\in(0,s)\},\quad \Sigma_s\coloneqq\{\tstar=s\}.
\end{nalign}
See \cref{fig:backward} for a detailed description.
In particular, let us note that the boundary $\partial \Mt(s)$ consists of future directed spacelike components  at $\{ \tstar = \tstar_1 \} \cap \partial \Mt(s)$ and the truncated cone $\mathscr{C} = \{ |x| = 2 (1 - t) \} \cap \partial \Mt(s)$, and a past directed spacelike boundary $\mathscr{S} = \{ \tstar = 0 \} \cap \partial \Mt(s)$.

We begin by recording some geometric quantities that will be helpful for the construction of the scattering solution in \cref{in:thm:main}. In the proof, we will be using $\tstar=t-f(x)$ and $y^i=x^i$ as coordinates. With respect to these coordinates, the Minkowski metric $\eta$, the usual coordinate vector fields $T = \partial_t$ and $\partial_{x^i}$, and the wave operator $\square$, can be written in the following form.
\begin{subequations}
	\begin{gather} \label{setup:eq:metric}
		\eta = - d\tstar^2 + 2 (\partial_{i} f) d \tstar d y^i + (\delta_{ij} - \partial_i f \partial_j f) dy^i dy^j, \\[0.5em] \label{setup:eq:vectorfields}
		T \coloneqq \partial_t|_x = \partial_\tstar|_{y}, \quad  \partial_{x^i}|_{t}\coloneqq\partial_{y^i}|_{\tstar} - (\partial_i f)\partial_\tstar|_{y}, \\[0.5em]
		\label{setup:eq:box}
		\Box=-(1-\abs{\partial f}^2)\partial_\tstar^2+\underbrace{(\Delta_x f-2\partial_i f\cdot \partial_{y^i})\partial_\tstar+\Delta_y}_{P_f}.
	\end{gather}
\end{subequations}

Let us introduce the two sets of operators $\Vt \coloneqq \{\tstar T,\tstar \partial_y,1\}$ and $\Vb\coloneqq\{\tstar T,\partial_y,1\}$.
These measure the top order regularity and the operators with which we will commute the equations respectively.
We write $\Diffb$ for the span of $\Vb$ over smooth function in $\Mt$ and $\Diffb^k$ for their $k$-fold product.
\begin{notation}
	For a function $f: \Mt \to \R$, will write $\Vb^k f$ for the $\R^{k(n+1)}$-valued function formed by acting with operators from $\Vb$:
	\begin{equation}
		\Vb^k f\coloneqq\{\Gamma_{\alpha_1}\Gamma_{\alpha_2}...\Gamma_{\alpha_k}f:\Gamma_i\in\Vb\}.
	\end{equation}
\end{notation}
We introduce the following $L^\infty$ and $L^2$ based function spaces: for any integer $s \geq 0$ and $q \in \R$,
\begin{nalign} \label{eq:X_spaces}
	\Ab^{s;q}(\Mt)&\coloneqq\{f\in \tstar^{q} L^\infty(\Mt): \Vb^s f\in L^\infty(\Mt)\},\\
	\Hb^{s;q}(\Mt)&\coloneqq\left\{f\in \tstar^{q} L^2(\Mt): \tstar^{-1/2}\Vb^s f\in L^2(\Mt)\right\},\\
	\Hc^{s;q}(\Mt)&\coloneqq\left\{f\in \tstar^{q} L^2(\Mt): \Vb^s f\in L^\infty\left((0,\tstar_1);L^2(\Sigma_\tstar)\right)\right\}.
\end{nalign}
with corresponding norms.
Here, $\Ab$ is used to most transparently state bounds in a coordinate independent way, $\Hb$ is used to close energy estimates and $\Hc$ to construct approximate solutions.
We also define spaces of finite regularity $s$ that have an expansion up to order $q$ in $\tstar$
\begin{equation}
	\Hs^{s;q}(\Mt)=\{f\in L^{2}(\Mt): f\in \sum_{j=0}^q \tstar^j H^{s-j}(\Sigma_0)+\sum_{j=q+1}^{Q} \tstar^j H^{s-q}(\Sigma_0)\},
\end{equation}
where we used the convention that $\sum^{q}=\sum^{\floor{q}}$.
The $s$ subscript stand for ``smooth'' indicating the behaviour in $\tstar$ coordinate.
Via Sobolev embedding, we observe the following inclusions, for any $\epsilon > 0$:
\begin{equation}\label{not:eq:Sob}
	\begin{gathered}
		\Ab^{s;q}(\Mt)\subset\Hc^{s;q}(\Mt)\subset\Ab^{s-\frac{n+1}{2};q}(\Mt)\\
        \Hc^{s;{q + \epsilon}}(\Mt)\subset\Hb^{s;q}(\Mt)\subset\Hc^{s-1;q {- \epsilon}}(\Mt).
	\end{gathered}
\end{equation}
Using the previously defined spaces of operators and functions, we trivially have
\begin{equation}
	\Diffb^k:\Ab^{s;q}(\Mt)\to\Ab^{s-k;q}(\Mt),\quad\Diffb^k:\Hc^{s;q}(\Mt)\to\Hc^{s-k;q}(\Mt).
\end{equation}

Let us now introduce the precise class of solutions of \cref{in:eq:main} that we study:
\begin{definition}\label{geo:def:scattering_sol}
	Let $s_p\coloneqq s-\floor{2\kappa_p}>\frac{n+2}{2}$ and $f\in C^{s+1}(\R^n),\psi\in C^{s_p}(\R^n)$.
    We say that a function $\phi\in \Ab^{2;-\alpha_p}(\Mt)$ is a \emph{singular solution of regularity $s$}  to \cref{in:eq:main} with scattering data $(f,\psi)$ if $\phi$ solves \cref{in:eq:main} and 
	\begin{nalign}\label{geo:eq:sol_form}
		\phi&=\phi_0+\tstar^{1-\alpha_p} \Hs^{s-1;2\kappa_p}(\Mt)+\tstar^{\beta_p}\log\tstar \phi_{\mathrm{log}}+\tstar^{\beta_p}\psi+\Hc^{s_p-1;\beta_p'}(\Mt), &\text{ if }2\kappa_p\in\N,\\
		\phi&=\phi_0+\tstar^{1-\alpha_p}\Hs^{s-1;2\kappa_p}(\Mt)+\tstar^{\beta_p}\psi+\Hc^{s_p-1;\beta_p'}(\Mt), &\text{ if }2\kappa_p\notin\N,
	\end{nalign}
    where $\phi_{\mathrm{log}}\in H^{s_p}(\Sigma_0)$, $\phi_0\coloneqq c_p(1-\abs{\partial f}^2)^{\frac{1}{p-1}}\tstar^{-\alpha_p}$ and some\footnote{the maximum value of $\beta_p'$ is limited by that in \cref{ode:eq:generic_sol}} $\beta_p'>\beta_p$. 
\end{definition}

    \subsection{Setup for the stability problem} \label{sub:forward_setup}

For the stability problem studied in \cref{in:thm:main2}, we recast the equation \eqref{in:eq:main} as an equation for a coordinate $\tau$, where the blow-up surface for $\phi$ will be represented exactly by the zero set of $\tau$.

\begin{lemma} \label{lem:newtime}
    Let $\phi$ be a solution of \cref{in:eq:main} and $\tau$ be such that $\phi \eqqcolon c_p \tau^{- \alpha_p}$. Then $\tau$ satisfies the equation
    \begin{equation} \label{eq:nlw_tau}
        \square \tau = \kappa_p \tau^{-1} \left( 1 + \nabla_{\alpha} \tau \, \nabla^{\alpha} \tau \right),
    \end{equation}
    where $\kappa_p$ is defined in \eqref{in:eq:betadelta}.
\end{lemma}

\begin{proof}
    Assuming $\phi > 0$, this is an explicit computation, as follows:
    \begin{gather*}
        \square \left( c_p \tau^{- \alpha_p} \right) = c_p \alpha_p \kappa_p \nabla_{\alpha} \tau \nabla^{\alpha} \tau \cdot \tau^{- \alpha_p - 2} - c_p \alpha_p \square \tau \cdot \tau^{- \alpha_p - 1}, \\[0.5em]
        |\phi|^{p-1} \phi = c_p^p \tau^{- p \alpha_p} = \alpha_p \kappa_p \tau^{- \alpha_p - 2},
    \end{gather*}
    where we used the explicit expressions for $\alpha_p$ and $c_p$ in \eqref{in:eq:model}.
\end{proof}

\begin{remark}
    We comment that equation \eqref{eq:nlw_tau} with $\gamma = 1$ is also a suitably transformed version of the focusing exponential nonlinear wave equation $\square \phi = - e^{\phi}$
    upon applying the change of variables $\tau = e^{\phi}$. In particular, one may easily adapt the arguments of the present article to study this problem. 
\end{remark}

In the study of the stability problem, we use coordinates $(\tau, z^i)$ where $\tau$ is as above and $z^i = x^i$. The following definitions will be used to describe to the Jacobian matrix of the coordinate transformation\footnote{We derive the equations for $W,V^i$ geometrically and thus often do not specify with respect to which coordinate system  $W,V^i$ are expressed. For the sake of readability, we prefer to use terminology such as $W$ ``expressed in $(x,t)$'' or  ``expressed in $(\tau,z)$'' coordinates rather than inserting pushforwards such as $\Phi^\star$.}  $\Phi:(t, x^i)\to(\tau, z^i)$.

\begin{definition} We define $W, V^i$ and the \emph{lapse} $\Omega^2$ by\footnote{we use the Einstein summation convention throughout the paper and always raise and lower Greek indices with $\eta$ and latin indices with $\delta_{ij}$ and $\delta^{ij}$.}
    \begin{equation} \label{eq:metric_qty}
        W \coloneqq - \nabla_{\alpha} \tau \, \nabla^{\alpha} t= - \frac{\partial \tau}{\partial t}, \quad V^i \coloneqq \nabla_{\alpha} \tau \, \nabla^{\alpha} x^i=\frac{\partial\tau}{\partial x^i}, \quad \Omega^2 \coloneqq - \nabla_{\alpha} \tau \, \nabla^{\alpha} \tau = W^2 - \delta_{ij} V^i V^j.
    \end{equation}
    We also use the notation $J = (W, V^i)^{\intercal} \in \R^{n+1}$.
\end{definition}

Thereby, we have the following linear transformations for coordinate vector fields and differentials:%
\begin{subequations}
\begin{gather} \label{eq:zx}
        T = \frac{\partial}{\partial t} = W \frac{\partial}{\partial \tau}, \quad \frac{\partial}{\partial x^i} = V_i \frac{\partial}{\partial \tau} + \frac{\partial}{\partial z^i}, \\[0.5em]
        \label{eq:xz}
        dt = W^{-1} \left( d \tau - V_i dz^i \right), \quad dx^i = dz^i.
    \end{gather}
    We also introduce the unit normal to the constant $\tau$ hypersurface as 
    \begin{equation}
        N \coloneqq -\frac{\nabla \tau}{\Omega}=\frac{W}{\Omega}\partial_{\tau} - \frac{V_i}{\Omega}\partial_{z^i}.
    \end{equation}
    With respect to the $(\tau, z^i)$ coordinate system, one can check that the Minkowski metric $\eta$ and the inverse Minkowski metric $\eta^{-1}$ may be written as follows.
    \begin{gather} 
        \label{eq:metric}
        \eta = - W^{-2} \left( d \tau - V_i dz^i \right) \left( d \tau - V_j dz^j \right) + \delta_{ij} dz^i dz^j, \\[0.5em]
        \label{eq:invmetric}
        \eta^{-1} = - \Omega^{2} \frac{\partial}{\partial \tau} \otimes \frac{\partial}{\partial \tau} + V^i \frac{\partial}{\partial \tau} \otimes \frac{\partial}{\partial z^i} + V^i \frac{\partial}{\partial z^i} \otimes \frac{\partial}{\partial \tau} + \delta^{ij} \frac{\partial}{\partial z^i} \otimes \frac{\partial}{\partial z^j},
    \end{gather}
    Finally, the wave operator $\square$ will be -- using $\det \eta = W^{-2}$ -- given by
    \begin{align}
        \tau^2 \square \Phi
        &= \frac{\tau^2}{\sqrt{- \det \eta}} \partial_{\mu} \left( \sqrt{ - \det \eta} \cdot  (\eta^{-1})^{\mu \nu} \partial_{\nu} \Phi \right) \nonumber \\[0.5em]
        &= (\eta^{-1})^{\mu\nu} \cdot \tau^2 \partial_{\mu} \partial_{\nu} \Phi + W \tau \partial_{\mu} (W^{-1} (\eta^{-1})^{\mu\nu})) \cdot \tau \partial_{\nu} \Phi. \label{eq:wave_op}
    \end{align}
\end{subequations}

\begin{remark}[Model ODE values]\label{not:rem:boosted_model}
    For the model ODE blowup $\phi_{\mathrm{model}}$ in \eqref{in:eq:model} the above quantities take the values $W=1, V=0,\Omega^2 = 1$.
    Using a Lorentz boost in the direction $\mathbf{v}\in \BB_1$, these quantities change to $W=(1-\abs{\mathbf{v}}^2)^{-1/2}$, $V_i=\mathbf{v}_iW$, $\Omega^2=1$.
\end{remark}

For the stability problem, we again utilise the operators $\Vt = \{ \tau \partial_{\tau}, \tau \partial_z, 1 \}$ to measure the top order regularity, while our commutator vector fields will always be of the form $\partial_z$, and in particular are \emph{chosen to commute with operators in $\Vt$}. For a function $\Phi$, we use $\Vt \Phi$ to denote the vector $(\tau \partial_{\tau} \Phi, \tau \partial_{z} \Phi, \Phi)^{\intercal}$, while $\Vt J$ will denote the concatenation of $\Vt W$ and $\Vt V^i$ for all $i$. 

Next, we introduce the following spacetime regions which will be helpful for the proof, where the hat notation is used to signify that we are working with respect to a $(\tau, z)$ coordinate system.
\begin{equation} \label{eq:forward_regions}
    \Mhat_{\tau_1} \coloneqq \{ 0 < \tau \leq \tau_1, |z| < 2(\tau + 2) \}, \quad
    \Mhat_{\tau_0, \tau_1} \coloneqq \{ \tau_0 < \tau < \tau_1 \} \cap \Mhat_{\tau_1}, \quad
    \Shat_{\tilde{\tau}} = \{ \tau = \tilde{\tau} \} \cap \Mhat_{\tau_1}.
\end{equation}
The boundary of the region $\Mhat_{\tau_0, \tau_1}$ consists of two past-directed spacelike portions $\{ \tau = \tau_0 \} \cap \partial \Mhat_{\tau_0, \tau_1}$ and $\mathscr{C} = \{ |z| = 2 (\tau + 2) \} \cap \partial \Mhat_{\tau_0, \tau_1}$, and the future-directed spacelike portion $\{ \tau = \tau_1 \} \cap \partial \Mhat_{\tau_0, \tau_1}$, see \cref{fig:forward}.
We also make use of the following $L^2$ and $L^{\infty}$ based function spaces on the hypersurfaces $\Sigma_{\tau}$ for $s \in \N, q \in \R$.
\begin{nalign}
    &H^{s; q}(\Shat_{\tau})\coloneqq \{f\in \tau^{q} L^2(\Shat_{\tau}): \partial_z^{\alpha} f\in \tau^{q} L^2(\Shat_{\tau}) \; \forall \; 0 \leq |\alpha| \leq s \}, \\
    &C^{s; {q}}(\Shat_{\tau}) \coloneqq \{ f \in \tau^{q} L^{\infty}(\Shat_t): \partial_z^{\alpha} f \in \tau^q L^2 (\Shat_{\tau}) \; \forall \; 0 \leq |\alpha| \leq s\},
\end{nalign}
with appropriate norms $\| \cdot \|_{H^{s;q}(\Shat_{\tau})}$ and $\| \cdot \|_{C^{s;q} (\Shat_{\tau})}$. 
When $q = 0$, we often omit the second superscript. 
Note that, in comparison to the $X^{s,q}_{\bullet, \bullet}$ spaces defined in \eqref{eq:X_spaces}, we do not include commutation with $\tau \partial_{\tau}$ here.
Nevertheless, we also define $X^{s,q}_{\bullet, \bullet}(\Mhat)$ as in \cref{eq:X_spaces} with $(\tstar,x)\mapsto(\tau,z)$ in order to facilitate translations between the two problems.

\begin{notation}
    $L^2$ norms on any hypersurface $\Shat_{\tau}$ will not include volume forms, so
    \[
        \| \Phi \|_{L^2 (\Shat_{\tau})}^2 = \int_{\Shat_{\tau}} |f|^2 dz = \int_{\BB_{2(\tau + 4)}} |f(\tau, z)|^2 dz.
    \]
    We use the Sobolev embedding: for $s > s' + \frac{n}{2}$, we have
    \begin{equation} \label{eq:Sob2}
        C^{s; q} (\Shat_{\tau}) \subset H^{s; q} (\Shat_{\tau}) \subset C^{s'; q} (\Shat_{\tau}),
    \end{equation}
    where any implicit constants will depend on $s, s'$ but not on $\tau$.
\end{notation}

\subsubsection{The system of equations for \texorpdfstring{$W$}{W} and \texorpdfstring{$V^i$}{V\^i}}

The analysis of the nonlinear wave equations \eqref{eq:nlw_tau} will be done via deriving suitable nonlinear equations for the quantities $W$, $V^i$ and $\Omega^2$. The first of these are a system of nonlinear wave equations:

\begin{lemma} \label{lem:nlw_jac}
    Let $\tau$ be a solution to \cref{eq:nlw_tau} and $W,V^i$ as defined in \cref{eq:metric_qty}.
    The Jacobian coefficients $W$ and $V^i$ satisfy the equations
    \begin{subequations}\label{eq:nlw_Jacobi}
    	    \begin{gather}
    		\label{eq:nlw_W}
    		\tau^2 \square W = - \kappa_p (1 - \Omega^2) W - \kappa_p W \tau \partial_{\tau} \Omega^2, 
    		\\[0.5em] \label{eq:nlw_V}
    		\tau^2 \square V^i = - \kappa_p ( 1 - \Omega^2 ) V^i - \kappa_p V^i \tau \partial_{\tau} \Omega^2 - \kappa_p \tau \partial_{z^i} \Omega^2.
    	\end{gather}
    \end{subequations}
\end{lemma}

\begin{proof} 
    Using the definition of $W$ in \eqref{eq:metric_qty}, the fact that $\nabla$ commutes with $\square$ and $\eta$, that $\nabla^2 t= 0$, and \cref{eq:nlw_tau}, we find that
    \begin{align*}
        \tau^2 \square W
        &= - \tau^2 \square \left( \eta ( \nabla \tau, \nabla t ) \right) 
        = - \tau^2 \eta ( \nabla \square \tau, \nabla t ) \\
        &= - \kappa_p \tau^2 \eta \left( \nabla \left( \tau^{-1} (1 - \Omega^2) \right), \nabla t \right) \\
        &= \kappa_p (1 - \Omega^2) \eta( \nabla \tau, \nabla t) + \kappa_p \tau \eta (\nabla \Omega^2, \nabla t),
    \end{align*}
    which gives \eqref{eq:nlw_W} upon using \eqref{eq:zx}. The derivation of \eqref{eq:nlw_V} is similar.
\end{proof}

On top of these second-order equations, it will later be useful when deriving pointwise estimates for $W, V^i, \Omega^2$ or constructing expansions, to consider the following system of \emph{first-order equations}. 
We derive a system of $N+1$ such equations, corresponding to the degrees of freedom of the Jacobian coefficients, with respect to the following ``linearly small quantities'' around $\phi_{\mathrm{model}}$ .

\begin{definition} \label{def:transport_variables}
    We define $U^i$ and $\mathring{\Omega^2}$ to be the following:
    \begin{equation} \label{eq:uomega}
        U^i = \frac{V^i}{W}, \qquad \mathring{\Omega}^2 = \Omega^2 - 1.
    \end{equation}
    Note that we may express $V^i, W$ and $\Omega^2$ in terms of these via
    \begin{equation} \label{eq:vwomega}
        \Omega^2 = \mathring{\Omega}^2 + 1, \quad W = \sqrt{ \frac{\Omega^2}{1 - \delta_{ij} U^i U^j} }, \quad V^i = W U^i.
    \end{equation}
\end{definition}

The first-order equations for $U^i$ and $\mathring{\Omega}$ are then as follows.

\begin{prop} \label{prop:transport}
    For $\tau$ satisfying \eqref{eq:nlw_tau}, and $W, V^i, \Omega^2$ as in \eqref{eq:metric_qty}, we have
    \begin{equation} \label{eq:transport_U}
        \tau \partial_{\tau} U^i = W^{-2} \tau \partial_{z^i} W,
    \end{equation}
    \begin{equation} \label{eq:transport_omega}
        \tau \partial_{\tau} \mathring{\Omega}^2 = 2 \kappa_p \mathring{\Omega}^2  + 2 \tau \partial_{z^i} V^i.
    \end{equation}
\end{prop}

\begin{proof}
    The equation \eqref{eq:transport_U} is an immediate consequence of the fact that the $1$-form $dt = W^{-1} d\tau - U_i dz^i$ from \eqref{eq:xz} is closed. To derive \eqref{eq:transport_omega}, we use \eqref{eq:nlw_tau} together with the expression for the wave operator \eqref{eq:wave_op} to get
    \begin{align*}
        \kappa_p (1 - \Omega^2) 
        &= \tau \square \tau \\
        &= W \tau \partial_{\mu} (W^{-1} (\eta^{-1})^{\mu \tau} ) \\
        &= W \tau \partial_{\tau} (- \Omega^2 W^{-1}) + W \tau \partial_{z^i} (W^{-1} V^i) \\
        &= - \tau \partial_{\tau} \Omega^2 + \Omega^2 \tau \partial_{\tau} \log W + W \tau \partial_{z^i} (W^{-1} V^i) \\
        &= - \frac{1}{2} \tau \partial_{\tau} \mathring{\Omega}^2 - \frac{\Omega^2}{2} \tau \partial_{\tau} \log(1 - \delta_{ij} U^i U^j) - V^i W^{-1} \partial_{z^i} W + \tau \partial_{z^i} V^i,
    \end{align*}
    where in the final line we used \eqref{eq:vwomega}.

    Using \eqref{eq:transport_U} to expand the second term, we see that
    \[
        - \frac{\Omega^2}{2} \tau \partial_{\tau} \log(1 - \delta_{ij} U^i U^j) = W^2 \delta_{ij} U^i \tau \partial_{\tau} U^j = V^i W^{-1} \tau \partial_{z^i} W.
    \]
    Inserting this in the above and simplifying, we get \eqref{eq:transport_omega}. The commuted equations \eqref{eq:transport_U_comm} and \eqref{eq:transport_omega_comm} are also immediate.
\end{proof}

Note, that geometrically $U^i$ is the partial derivatives of $t$ along constant $\tau$ hypersurfaces.
Since we want constant $\tau$ (equivalently, constant $\phi$) to correspond to hypersurfaces, this forces an additional integrability condition on $U^i$.
\begin{remark}
Using that $U^i=-\partial_{z^i}t$, the following integrability condition will be true at all times
    \begin{equation}\label{eq:integrability}
        \partial_{z_j}U^i-\partial_{z_i}U^j=0.
    \end{equation}
    We can consider \eqref{eq:integrability} as a ``constraint equation'' that is propagated by the remaining equations \eqref{eq:nlw_Jacobi}, and thus largely ignore it in the sequel.
\end{remark}

\begin{remark}[Degrees of freedom]
	The scattering data for $\phi$ clearly constitutes of two functions $(f,\psi)$.
	For $W,V^i$, it is clear that we can reconstruct $W,V^i$ from \cref{eq:transport_U,eq:transport_omega} in an expansion starting with the information $U^i|_{\Sigma_0}$ and the $\tau^{2\kappa}$ part of $ \mathring{\Omega}^2$.
    While the latter seems to be more initial data, it is easy to see that $U^i=\frac{\partial_i \tau}{\partial_t\tau}=\partial_i|_{\Sigma_0}t=\partial_i f$. Thus, in light of \eqref{eq:integrability}, there is no mismatch between the degrees of freedom.
\end{remark}

To conclude \cref{sub:forward_setup}, we introduce the notion of \emph{scattering data} at the singularity within the context of the $(\tau, z^i)$ coordinate system.

\begin{definition}\label{not:def:scattering_VW}
    Let $s > \frac{n}{2} + 2 + \lfloor 2 \kappa_p \rfloor$ and $s_p \coloneqq s  - \lfloor 2 \kappa_p \rfloor$, and let $\mathfrak{f} \in C^{s+1}(\BB_r),
    \mathfrak{w} \in C^{s_p}(\BB_r)$. We say that functions $W, V^i$ are
    \emph{scattering solutions of regularity $s$} to the nonlinear wave
    equations \eqref{eq:nlw_Jacobi}, with scattering data
    $(\mathfrak{f}, \mathfrak{w})$ if $W, V^i$ solve \eqref{eq:nlw_Jacobi} in
    the $(\tau, z^i)$ coordinate system, $W\in(1/2,2)$ and that moreover, the
    associated $(U^i, \mathring{\Omega}^2)$ satisfy the following asymptotics
    for all $0 < \tau < \tau_1$: 
    \begin{nalign} \label{eq:asymp:U}
        U^i &= \partial_i \mathfrak{f} + \tau \Hs^{s-1; 2\kappa_p}(\Mhat)+\tau^{2\kappa_p+1}\log\tau H^{s_p-1}(\Shat_0)+\tau^{2\kappa_p+1}H^{s_p-1}(\Shat_0)+\Hc^{s_p-2;2\kappa'+1} (\Mhat), &\text{ if }2\kappa_p\in\N\\
        U^i &= \partial_i \mathfrak{f} + \tau \Hs^{s-1; 2\kappa_p}(\Mhat)+\tau^{2\kappa_p+1}H^{s_p-1}(\Shat_0)+\Hc^{s_p-2;2\kappa'+1} (\Mhat), &\text{ if }2\kappa_p\notin\N.
    \end{nalign}
   for some $\kappa'>\kappa_p$ and
	\begin{nalign} \label{eq:asymp:Omega}
        \mathring{\Omega}^2 &= \tau \Hs^{s-1; 2 \kappa_p}(\mathcal{M}_{\tau_1}) + \tau^{2 \kappa_p} \log \tau \, H^{s_p}(\Shat_0) + \tau^{2 \kappa_p} \mathfrak{w} + \Hc^{s_p-1; 2\kappa'}(\Mhat), &\text{ if }2 \kappa_p \in \N, \\
        \mathring{\Omega}^2 &= \tau \Hs^{s-1; 2\kappa_p}(\mathcal{M}_{\tau_1}) + \tau^{2 \kappa_p} \mathfrak{w} + \Hc^{s_p-1;2\kappa'}(\Mhat), &\text{ if }2 \kappa_p \notin \N.
	\end{nalign}
\end{definition}

\subsection{Comparison between the two geometric setups}\label{no:sec:compare}

In this section, we record how solutions of \cref{in:eq:main} determine the constant $\phi$ hypersurfaces and geometric quantities such as $(W, V^i)$, as well as how to reconstruct $\phi$ from these geometric quantities.
We do this only in the case $2\kappa_p \notin \N$ for sake of notational simpliciaty and leave the other case to the reader.
We shall also be rather loose about domains; the following correspondences should only be made in intersections $\Mt \cap \Mhat_{\tau_1}$ but we are not precise about this in the sequel.

\begin{lemma}\label{not:lem:comparison}
	Fix $p>1$ such that $2\kappa_p\notin\N$ and define $s_p=s-\floor{2\kappa_p}>\frac{n+2}{2}+2$.
    \begin{enumerate}[(a)]
		\item Let $\phi$ be a singular solution to \cref{in:eq:main} of regularity $s$ as in \cref{geo:def:scattering_sol}.
		Then it holds that
		$\tau\in \Hs^{s;2\kappa_p}(\Mt)
		+\tstar^{2\kappa_p+1} H^{s_p}(\Sigma_{0})+\Hc^{s_p-1;\floor{2\kappa_p+1}}(\Mt)$.
		
        Furthermore, the geometric quantities $V^i,W$ in $(\tau,z)$ coordinates are a scattering solution of regularity $s-1$ as in \cref{not:def:scattering_VW}.
		
		\item 
		The reverse statements holds: given $V^i,W$ geometric scattering solution of regularity $s$ in $\Mhat$ as in \cref{not:def:scattering_VW}, it follows that: there exists $f\in H^s(\Sigma_0)$ and corresponding $\tstar=t-f(x)$ such that
		$\phi=\tau^{-\alpha_p}c_p$ with $\partial_t\tau=W, \partial_{x^i}\tau=V^i$ is a singular solution of regularity $s-1$.
	\end{enumerate}
\end{lemma}
\begin{proof}
    (a)
	We note that $(\phi_0/c_p)^{-1/\alpha_p}=\tstar(1-\abs{\partial f}^2)^{-1/2}$.
	Let us expand $\tau$ using \cref{geo:def:scattering_sol} as
	\begin{multline}\label{not:eq:tau_t}
        \tau=(\phi/c_p)^{-1/\alpha_p}=(\phi_0/c_p)^{-1/\alpha_p}\Big(1+\frac{\phi-\phi_0}{\phi_0}\Big)^{-1/\alpha_p}\\[0.5em] =\tstar(1-\abs{\partial f}^2)^{-1/2}
		+\tstar^2 \Hs^{s-1;2\kappa_p-1}(\M)-\tstar^{2\kappa_p+1}\frac{\psi}{c_p\alpha_p}+\Hc^{s_p-1;2\kappa'+1}(\M),
	\end{multline}
	where $\psi$ is the scattering data of \cref{geo:eq:sol_form} and $\kappa'>\kappa_p$. 
	Differentiating $\tau$ and using \cref{eq:metric_qty}, we also obtain
	\begin{subequations}
		\begin{align}
			V^i&=\partial_i f (1-\abs{\partial f}^2)^{-1/2}+\tstar\Hs^{s-2;2\kappa_p-1}(\Mt)-\tstar^{2\kappa_p}\partial_i f \frac{(2\kappa_p+1)\psi}{c_p\alpha_p}+\Hc^{s_p-2;2\kappa'}(\M),\\
			W&=(1-\abs{\partial f}^2)^{-1/2}+\tstar\Hs^{s-2;2\kappa_p-1}(\Mt)-\tstar^{2\kappa_p} \frac{(2\kappa_p+1)\psi}{c_p\alpha_p}+\Hc^{s_p-2;2\kappa'}(\M),\label{comp:eq:W}\\
            \text{ and thus } U^i&=\Hs^{s-1;2\kappa_p}(\M)+\tstar^{2\kappa_p+1}H^{s_p}(\Sigma_0)+\Hc^{s_p-2;2\kappa'}(\M),
		\end{align}
	\end{subequations}
	where in the final step we used a cancellation for the coefficients at the $\tstar^{2\kappa_p}$ level that follows from the explicit expressions.
    One may compare the leading order terms to \cref{not:rem:boosted_model}.
	A similar cancellation also occurs for the error term $\Hc^{s_p-2;2\kappa'}(\M)$ and thus it may be improved to $\Hc^{s_p-2;2\kappa'+1}(\M)$.
	
	It still remains to change coordinates from $(\tstar, y)$ to $(\tau, z)$. From \cref{comp:eq:W} it follows that for $\tstar_1$ sufficiently small $\partial_t\tau=W\in(1/2,2)$ and thus  $\Hb^{s';q'}(\M)=\Hb^{s';q'}(\Mhat)$ for all $s'\leq s_p-1$ and for all $q'$, with equivalent norms.
	Using the inclusion $\Hc^{s;q+\epsilon}(\M)\subset \Hb^{s;q}(\M)\subset \Hc^{s-1;q}(\M)$ for all $\epsilon>0$, it follows that we can invert $\tau(\tstar,x)$ as
	\begin{equation}
		\tstar(\tau,z)=\tau (1-\abs{\partial f}^2)^{1/2}+\tau^2\Hs^{s-1;2\kappa_p-1}(\Mhat)+\tau^{2\kappa_p+1}H^{s_p}(\Sigma_0)+\Hc^{s_p-2;2\kappa'+1-\epsilon}(\Mhat).
	\end{equation}
	Therefore, we obtain that $V^i,W$ are of the correct form to be scattering solutions in $(\tau,z)$ coordinates of regularity $s-1$.
	
	Let us also compute that $\tilde{\tau}\coloneqq\tstar(1-\abs{\partial f}^2)^{-1/2}$ satisfies
	$\eta(\dd\tilde{\tau},\dd\tilde{\tau})=-1+\tstar \Hc^{s-1;1}(\M)$.
	In order to recover the scattering data for $\mathring{\Omega}^2$, we compute
	\begin{equation}
		\mathring{\Omega}^2=\eta(\dd\tau,\dd\tau)+1=\mathbf{t}\Hs^{s-1;2\kappa_p-1}(\M)-\mathbf{t}^{2\kappa_p}\frac{2\psi(2\kappa_p+1)}{c_p\alpha_p}+\Hc^{s_p-2;2\kappa'}(\M).
	\end{equation}
	A similar expression therefore also holds in $(\tau,x)$ coordinates.
	
    (b)
	For the contrary, we use the definition of $W$ in \cref{eq:metric_qty} to express $t(\tau,y)$ from
	\begin{equation}\label{comp:eq:Jacobi_b}
		\begin{aligned}
			\partial_\tau t&=W^{-1}\in \Hs^{s;2\kappa_p}(\Mhat)+\tau^{2\kappa_p}H^{s_p}(\Sigma_0)+\Hc^{s_p-1;2\kappa'_p}(\Mhat),\\
            \partial_{z^i} t&=\frac{V^i}{W}=U^i\in \Hs^{s;2\kappa_p}(\Mhat)+\tau^{2\kappa_p+1}H^{s_p}(\Sigma_0)+\Hc^{s_p-1;2\kappa'_p+1}(\Mhat),
		\end{aligned}
	\end{equation}
	where the inclusion follows from the $L^\infty$ bound on $W$, and \cref{not:def:scattering_VW}.
    Thus, using the integrability conditions \eqref{eq:integrability} and \eqref{eq:transport_U} to determine $t$, we obtain
	\begin{equation}
		t\in \Hs^{s+1;2\kappa_p}(\Mhat)+\tau^{2\kappa_p+1}H^{s_p}(\Sigma_0)+\Hc^{s_p-1;2\kappa'+1}(\Mhat).
	\end{equation}
	Let us define $\tstar(\tau, z)= t(\tau, z) - t(0, z)$, so that
	\begin{equation}\label{comp:eq:t_tau}
		\tstar=W^{-1}\tau+\tau^2\Hs^{s-1;2\kappa_p}(\Mhat)+\tau^{2\kappa_p+1}H^{s_p}(\Sigma_0)+\Hc^{s_p-1;2\kappa_p'}(\Mhat).
	\end{equation}
	Using that $W\in(1/2,2)$, we know that $\tstar(\tau,z)$ is invertible function for each $z$.
	As in (a), we may use \cref{comp:eq:Jacobi_b,comp:eq:t_tau} to find the inverse
	\begin{equation}
		\tau= W\tstar+\tstar\Hs^{s-1;2\kappa_p}(\M)+\tstar^{2\kappa_p+1}H^{s_p}(\Sigma_0)+\Hc^{s_p-2;2\kappa'-\epsilon}(\M).
	\end{equation}
	From this, we can use the definition $\phi=c_p\tau^{-\alpha_p}$ to obtain that $\phi$ is a singular solution of regularity $s-1$.
\end{proof}	

    \section{The scattering construction: Proof of \texorpdfstring{\cref{in:thm:main}}{Theorem 1.1}}\label{sec:backward}

    i
We follow the strategy explained in \cref{in:sec:ideas}.
We begin in \cref{back:sec:map} by linearising \cref{in:eq:main} around the leading order ODE blow-up, namely $\phi_0$ as in \eqref{eq:phi_0}, and study properties of this linear operator and associated nonlinear operators between the spaces introduced in \cref{no:sec:backward}.
Next, in \cref{back:sec:ansatz}, we construct an ansatz $\phi_N$ that solves \cref{in:eq:main} up to an arbitrarily small error term $\mathcal{O}(\tstar^N)$ for $N \gg 1$.
Finally, in \cref{back:sec:energy}, we derive an energy estimate that yields coercive control for the linearised operator around such an approximate solution, thus yielding a good a priori estimate provided that the forcing is $\mathcal{O}(\tstar^N)$ for sufficiently large $N$. 
We find the nonlinear solution by a limiting argument.

\subsection{Mapping properties of linear and nonlinear operators}\label{back:sec:map}

Fix $f\in H^{s+1}(\BB_3)$ for $s>\frac{n+4}{2}$. We are looking for a solution that to leading order takes the form 
\begin{equation} \label{eq:phi_0}
	\phi_0:=(1-\abs{\partial f}^2)^{\frac{1}{p-1}}c_p\tstar^{-\alpha_p}.
\end{equation}
Using \cref{setup:eq:box}, the linearisation of $P[\phi]$ in \cref{in:eq:main} around $\phi_0$ then takes the form 
\begin{equation}
	\Box+p\phi_0^{p-1}=-(1-\abs{\partial f}^2)\big(\partial_\tstar^2-\gamma_p\tstar^{-2}\big)+P_f \eqqcolon P_0+P_f.
\end{equation}
We note that $P_0$ belongs to $\tstar^{-2}\Diffb^2$, while $P_f$ belongs to $\tstar^{-1}\Diffb^2$ with $H^{s-1}$ coefficients.
We also note that the dominant operator, namely $P_0$, is invertible on expressions of the form $g(y) \tstar^{\beta}$ for $\beta\neq \beta_p$, as shown in \cref{ode:eq:invert}.

We also introduce the following expressions, which represent the additional linear term and the nonlinear terms, which arise when one is ``linearising'' $P[\Psi + \Phi] = 0$ around $\Psi$:
\begin{equation}
	\mathfrak{E}_{\Psi}[\Phi]=-p(\phi_0^{p-1}-\Psi^{p-1})\Phi,\qquad \mathcal{N}_{\Psi}[\Phi]=(\Psi+\Phi)^p-\Psi^p-p\Psi^{p-1}\Phi.
\end{equation}
These are defined such that $P[\Psi + \Phi] = 0$ is exactly equivalent to
\begin{equation} \label{back:eq:lin_psi}
    P_0 \Phi + P_f \Phi + \mathfrak{E}_{\Psi}[\Phi] + \mathcal{N}_{\Psi}[\Phi] = - P[\Psi].
\end{equation}
In the following lemma, we record some mapping properties of $\mathcal{N},\mathfrak{E}$, following the analytic setup introduced in \cref{no:sec:backward}.
\begin{lemma}\label{map:lem:nonlin}	
	Fix $s > \frac{n+4}{2}$ and $s' \in \left[\frac{n+2}{2},s-1\right]$, $q\geq1$. 
	\begin{enumerate}
		\item Let $\Phi\in\Hc^{s';q-\alpha_p}(\Mt)$ and $\Phi_1\in\Hc^{s';1-\alpha_p}(\Mt)$.
		Then, for $\tstar_1<1$, it holds that
		\begin{equation}\label{map:eq:N_A}
			\mathfrak{E}_{\phi_0+\Phi_1}[\Phi],\mathcal{N}_{\phi_0+\Phi_1}[\Phi]\in\Hc^{s';q-\alpha_p-1}(\Mt).
		\end{equation}
		\item Let $\Phi\in\Hb^{s';q-\alpha_p}(\Mt)$ and $\Phi_1\in\Hb^{s';1-\alpha_p}(\Mt)$, then it holds that
		\begin{equation}\label{map:eq:N_H}
			\mathfrak{E}_{\phi_0+\Phi_1}[\Phi],\mathcal{N}_{\phi_0+\Phi_1}[\Phi]\in\Hb^{s';q-\alpha_p}(\Mt).
		\end{equation}
		\item 	Let $q\in\N$ and $\Phi\in\tstar^{q-\alpha_p} \Hs^{s';0}(\Mt)$ and $\Phi_1\in\tstar^{1-\alpha_p} \Hs^{s;s-s'}(\Mt)$, then it holds that
		\begin{equation}\label{map:eq:N_phg}
			\mathfrak{E}_{\phi_0+\Phi_1}[\Phi],\mathcal{N}_{\phi_0+\Phi_1}[\Phi]\in\tstar^{q-\alpha_p-1} \Hs^{s;s-s'}(\Mt).
		\end{equation}	
		\item 
            We have a nonlinear estimate for $\Phi_1\in\Hc^{s;1-\alpha_p}(\M)$ and $\norm{\Phi_2}_{\Hb^{s;q-\alpha_p}(\M)} ,\norm{\Phi_3}_{\Hb^{s;q-\alpha_p}(\M)}\leq 1$
		\begin{equation}\label{map:eq:N_linearised}
			\norm{\mathcal{N}_{\phi_0+\Phi_1}[\Phi_2]-\mathcal{N}_{\phi_0+\Phi_1}[\Phi_3]}_{\Hb^{s;q-\alpha}}\lesssim\norm{\Phi_2-\Phi_3}_{\Hb^{s;q-\alpha_p}}.
		\end{equation}
	\end{enumerate}
	
\end{lemma}
\begin{proof}
	The first inequality follows by noting that $\abs{(1+x)^p-1-px}\lesssim_p x^2$ for $x$ sufficiently small and expanding
	\begin{gather}
		\mathcal{N}_{\Psi}[\Phi]=\Psi^p\Big((1+\Phi/\Psi)^p-1-p\Phi/\Psi\Big)\in\Hc^{s;-p\alpha_p+2q}\subset\Hc^{s;q-\alpha_p-1},\\
		\mathfrak{E}_{\phi_0+\Phi_1}[\Phi]=\phi_0^{p-1}\big( (1+\Phi_1/\phi_0)^{p-1}-1\big)\Phi\in\Hc^{s;-(p-1)\alpha_p+1+q-\alpha_p}=\Hc^{s;q-\alpha_p-1},
	\end{gather}
	where we used $\alpha_p(p-1)=2,q\geq1$.
	The second follows from the same computation using the embedding \cref{not:eq:Sob} and product estimates; note we lose one power of decay from Sobolev but this will be ultimately harmless.
	   Next, \cref{map:eq:N_H} follows from noting that all the expansions we performed are smooth, so if $\Phi_1,\Phi$ are smooth, so is the result.

       Finally, \cref{map:eq:N_linearised} follows from the smoothness of $\mathcal{N}_\Psi[\Phi]$ in its argument for $\abs{\Phi}/\Psi<1/2$ and its quadratic nature
       \begin{equation}
           \abs{(1+x)^p-1-px}-\abs{(1+(x+y))^p-1-p(x+y)}\lesssim \abs{y},\qquad \forall x,y<1/2.
       \end{equation}
\end{proof}

\subsection{The singular ansatz}\label{back:sec:ansatz}

\begin{definition}
	We say that $\phi$ is a singular \emph{ansatz} of decay $N$ and regularity $s>N+\frac{n+1}{2}$ if it is of the form\footnote{We use the $1/2$ shift in the regularity so that the $p\in\N$ case, where logarithmic terms appear in the expansion, can be treated with the same notation.} \cref{geo:eq:sol_form} and
	\begin{equation}
		\Box\phi+\phi^p\in\Hc^{s-N-2;N-\alpha_p-3/2}(\Mt).
	\end{equation}
\end{definition}
One easily obtains that $\phi_0$ is a singular ansatz of decay $1/2$ and regularity $s-2$ whenever $f\in H^{s+1}(\BB_1)$.
To improve on this ansatz, we use the explicit inversion given by \cref{ode:eq:invert} for slowly decaying errors, $\mathcal{O}(\tstar^{q})$ with $q \leq \beta_p$.
For faster decaying errors, it is not necessary to keep track of the exact structure and we rather bound the errors and corrections in $\Hc(\M)$ spaces, using the following general ODE result:
\begin{lemma}\label{ode:lemma:invert}
	Given $f\in\Hc^{s;q}(\M)$ with $q>\beta_p$, there exists a unique $g\in \Hc^{s;\beta_p+}(\M)$ such that $(t^2\partial_t^2-\gamma_p)g=f$, which in fact satisfies $g \in\Hc^{s;q}(\M)$.
\end{lemma}
\begin{proof}
	To prove the lemma, it will suffice to show the following ODE\footnote{See \cite[Lemma 7.5]{hintz_stability_2020} for an $L^2$ based analogue.} result 
	\begin{equation}\label{an:eq:ode}
		\norm{\{1,t\partial_t\}^kg}_{t^qL^{\infty}((0,1])}\lesssim\norm{\{1,t\partial_t\}^k(t\partial_t-q') g}_{t^qL^{\infty}((0,1])},\qquad \forall q>q', g\in t^{q'+}L^{\infty}((0,1]).
	\end{equation}
	Observe that by commuting with $t^{q'}$, it suffices to prove \cref{an:eq:ode} for $q'=0$.
	Let us normalise $\norm{t\partial_t g}_{t^qL^\infty}=1$ and for $k=0$ compute
	\begin{equation}
		\norm{g}_{L^\infty}\leq \sup_{t\in[0,1)}t^{-q}\int_0^t\frac{\dd t}{t}\abs{t\partial_tg}\leq  \sup_{t\in[0,1)}t^{-q}\int_0^t\frac{\dd t}{t}t^q =q^{-1}
	\end{equation}
	
	We commute $(t^2\partial_t^2-\gamma_p)g=f$ with  $t^{-\beta_p}$ using \cref{in:eq:betadelta} to get
	\begin{equation}\label{ode:eq:commuted}
		(t^2\partial_t^2 + 2\beta_p t\partial_t)(t^{-\beta_p}\phi) = \big(t\partial_t + 2 \beta_p - 1\big)t\partial_t (t^{-\beta_p}\phi)=t^{-\beta_p}f.
	\end{equation}
	Applying \cref{an:eq:ode} twice yields the result.			
\end{proof}

It will be important when constructing the ansatz to separate the case when  $2\gamma_p$ is an integer.

\begin{prop}\label{an:prop:ansatz}
	Fix $N\in\N$, $s-N>\floor{2\kappa_p}+5+\frac{n+1}{2}$, $f\in H^{s+1}(\R^n)$ and $\psi\in H^{s-\floor{2\kappa_p}}(\R^n)$ scattering data.
	Then, there exists a singular ansatz $\phi$ of regularity $s$, decay order $N+\floor{2\kappa_p}$ and scattering data $\psi$ satisfying $P[\phi]\in \Hc^{s-N-\floor{2\kappa_p}-2;N-5/2}(\M)$.
\end{prop}
\begin{proof}
	This follows from iteratively inverting the leading operator via either the explicit formulae \cref{ode:eq:invert} or \cref{ode:lemma:invert}, similarly as done in \cite{cazenave_solutions_2020,kichenassamy_fuchsian_2007}.
	We give some details for completeness.

	Let us focus on the case $2\kappa_p\in\N$, the other case being easier.
We first find $\phi_q$ for $q< \floor{2\kappa_p}$ such that $P[\phi_q]\in\tstar^{q-\alpha_p-1}\Hs^{s-q-1}(\Mt)$.
	For $q=0$, we use $\phi_0$ from \cref{geo:def:scattering_sol}.
    For $q\geq1$, let us define $\phi_q=\phi_{q-1}+\psi_q\tstar^{q-\alpha_p}$ by solving $P_0[\psi_q \tstar^{q - \alpha_p}]$ to be the leading term of $P[\phi_{q-1}]$, or more precisely
	\begin{equation}
		\big((\alpha_p-q)(\alpha_p-q+1)-\gamma_p\big)(1-\abs{\partial f}^2)\psi_q:=-\tstar ^{\alpha_p+1-q}P[\phi_{q-1}]|_{\tstar=0}.
	\end{equation}
	We claim that $P[\phi_{q-1}]\in \tstar^{q-\alpha_p-2} \Hs^{s-q-2}(\Mt)$ and $\psi_q\in H^{s-q}(\Sigma_0)$.
	Indeed, we can write
	\begin{equation}
		P[\phi_q]= P[\phi_{q-1}]+ P_0[\psi_q]+\big(P_f[\psi_q]+\mathfrak{E}_{\phi_{q-1}}[\psi_q]+\mathcal{N}_{\phi_{q-1}}[\psi_q]\big),
	\end{equation}
	and the claim follows by induction using \cref{map:eq:N_A} of \cref{map:lem:nonlin} and the explicit formula \cref{ode:eq:invert}.
	
	Next, we define $\phi_{\mathrm{log}}$ as $(2\beta_p-1)\phi_{\mathrm{log}}=-\lim_{\tstar\to0}\tstar^{\beta_p+2}P[\phi_{2\kappa_p-1}]$, so that
	\begin{equation}
		P[\underbrace{\phi_{2\kappa_p-1}+\phi_{\log}\tstar^{\beta_p}\log\tstar+t^{\beta_p}\psi}_{\eqqcolon \phi_{\floor{2\kappa_p}}} ]\in  \Hc^{s-{2\kappa_p}-1;\beta_p-1-\epsilon}(\Mt)+ \Hc^{s-{2\kappa_p}-2;\beta_p-\epsilon}(\Mt),\quad \forall\epsilon>0.
	\end{equation}
	
	Next, we construct $\phi_q$ for $q>\floor{2\kappa_p}$ such that $P[\phi_q]\in\Hc^{s-q-1;q-\alpha_p-3/2}(\M)+\Hc^{s-q-2;q-\alpha_p-1/2}(\M)$.
	Let us write $\left(P[\phi_q]\right)_{1}$ and $\left(P[\phi_q]\right)_{2}$ for the decomposition into these summands.
	We define $\psi_q\in \Hc^{s-2q;q-\alpha_p}(\Mt)$ for $2\kappa_p< q<N$ via solving parametrically in $y$ the equation $P_0\psi_q=-\left(P[\phi_{q-1}]\right)_1$,  the solution of which is given by \cref{ode:lemma:invert}.
	Using $\phi = \phi_{N+\floor{2\kappa_p}}$ as the ansatz yields the result:
	\begin{multline}
		P[\phi]\in\Hc^{s-N-\floor{2\kappa_p}-1;N+\floor{2\kappa_p}-\alpha_p-3/2}(\M)+\Hc^{s-N-\floor{2\kappa_p}-2;N+\floor{2\kappa_p}-\alpha_p-1/2}(\M)\\
		\subset \Hc^{s-N-\floor{2\kappa_p}-2;N+\floor{2\kappa_p}-\alpha_p-3/2}(\M)\subset  \Hc^{s-N-\floor{2\kappa_p}-2;N-3/2}(\M).
	\end{multline}
\end{proof}

\subsection{Energy estimates and the nonlinear solution} \label{back:sec:energy}	

Next, we prove an energy estimate for the linearised operator $P_{\phi_0}=\Box+p\phi^{p-1}_0$. 
Albeit the $p\phi_0^{p-1}$ part of $P_{\phi_0}$ will be treated as an error term, we keep it explicit because it determines the value of $q$ such that we obtain invertibility of $P_{\phi_0}:\Hb^{s;q}(\M)\to\Hb^{s;q-2}(\M)$, just as in \cref{in:sec:ideas}.
We first give some general definitions.
\begin{definition} \label{def:energymomentum}
	For some function $\Phi: U \to \R$ defined on a domain $\mathcal{D}$ in Minkowski space, define its \underline{energy momentum tensor} to be
	\begin{equation} \label{eq:energymomentum}
		\mathbb{T}_{\mu\nu} [\Phi] \coloneqq \nabla_{\mu} \Phi \, \nabla_{\nu} \Phi - \frac{1}{2} \eta_{\mu\nu} \cdot \eta(\nabla \Phi, \nabla \Phi).
	\end{equation}
	
	Then, for some vector fields $Z^{\mu}$ and $Q^{\mu}$ in the same domain $\mathcal{D}$, we define the \underline{energy current} associated to the triple $(\Phi, Z^{\mu}, Q^{\mu})$ to be the vector field $\J^{\mu} = \J^{\mu}[\Phi, Z, Q]$, given by 
	\begin{equation} \label{eq:current_general}
		\J^{\mu} = ( \eta^{-1} )^{\mu\nu} \mathbb{T}_{\nu \sigma} [\Phi] Z^{\sigma} - Q^{\mu} \Phi^2.
	\end{equation}
\end{definition}

\begin{lemma} \label{back:lem:energy_est}
	Consider the vector fields $Z^{\mu}$ and $Q^{\mu}$ given by
	\begin{equation} \label{back:eq:multiplier}
		Z^{\mu} \coloneqq \tstar^{-2q} T^{\mu} = \tstar^{-2q} \left( \frac{\partial}{\partial \tstar} \right)^{\mu}, \quad Q^{\mu} \coloneqq \frac{\kappa_p^2}{2} \tstar^{-2(q+1)} T^{\mu}.
	\end{equation}
	Then, for $q \geq 4\kappa_p$ and $\J^{\mu} = \J^{\mu}[\Phi, Z, Q]$ as in \cref{eq:current_general}, one has
	\begin{equation} \label{back:eq:current}
		\nabla_{\mu} \J^{\mu} \geq \tstar^{-2q} \Box \Phi \, T \Phi + \frac{q}{4} \tstar^{-2q-3} | \Vt \Phi |_{\ell^2}^2+\frac{\kappa_p^2q\abs{\Phi}^2}{4},
	\end{equation}
	as well as the coercivity on constant $\tstar$ hypersurfaces: for $N^{\mu} = - \frac{d \tstar}{ - \eta(\nabla \tstar, \nabla \tstar)^{1/2}}$, 
	\begin{equation} \label{back:eq:coercivity}
		\frac{1}{4} | \Vt \Phi |_{\ell^2}^2 \leq \J^{\mu} N_{\mu} \leq 4 | \Vt \Phi |_{\ell^2}^2.
	\end{equation}
\end{lemma}

\begin{proof}
	It will help to first compute the gradient $ \nabla^{\mu} \tstar$ and the normal vector field $N^{\mu}$. Via the usual Minkowskian coordinates, we find
	\[
	d \tstar = dt - \partial_{i} f dx^i,
	\]
	so the gradient vector field $\nabla^{\mu} \tstar$ is given by
	\[
	- \nabla \tstar = \frac{\partial}{\partial t} + \partial_i f \frac{\partial}{\partial x^i} = (1 - |\partial f|^2) T + \partial_i f \frac{\partial}{\partial y^i},
	\]
	where we used \eqref{setup:eq:vectorfields}. It also follows that $N = (1 - |\partial f|^2)^{1/2} T + (1 - |\partial f|^2)^{-1/2} \partial_i f \frac{\partial}{\partial y^i}$.
	
	Now, since $T$ is a Killing vector field it follows that the deformation field $\mathcal{L}_Z \eta$ and $\mbox{div}\, Q$ are
	\begin{gather*}
		(\mathcal{L}_Z \eta)_{\mu\nu} = \nabla_{\mu} Z_{\nu} + \nabla_{\nu} Z_{\mu} = - 2 q \tstar^{-2q-1} (\nabla_{\mu} \tstar T_{\nu} + \nabla_{\nu} \tstar T_{\mu}), \\[0.3em]
		\mbox{div} \, Q = - \kappa_p^2(q + 1) \tstar^{-2q-3} T^{\mu} \partial_{\mu} \tstar = - \kappa_p^2(q+1) \tstar^{-2q-3}.
	\end{gather*}
	A classical computation then yields that
	\begin{align}
		\nabla_{\mu} \J^{\mu} 
		&= \frac{1}{2} \T_{\mu\nu} (\mathcal{L}_Z \eta)^{\mu\nu} - \mbox{div} \, Q \Phi^2 - 2 (Q \Phi) \Phi + \tstar^{-2q} \square \Phi \cdot T \Phi \nonumber \\
		&= - 2q \tstar^{-2q - 3} \cdot \tstar^2 \T( T, \nabla \tstar) + \kappa_p^2(q+1) \tstar^{-2q-3} \Phi^2 - \kappa_p^2\tstar^{-2q-3} (\tstar T \Phi) \Phi + \tstar^{-2q} \Box \Phi \cdot T \phi. \label{forward:energymomentum}
	\end{align}
	Next, using \eqref{setup:eq:metric} and the above to see that $\eta(T, \nabla \tstar) = -1$, we evaluate $\T(T, \nabla \tstar)$ to be
	\[
	\T(T, \nabla \tstar) = - \frac{1}{2} \left( (1 - |\partial f|^2) (T \Phi)^2 + \gamma^{ij} \partial_{y^i} \Phi \partial_{y^j} \Phi \right).
	\]
	Inserting this into \eqref{forward:energymomentum} and using that $|\partial f|^2 < \frac{1}{2}$, then using Young's inequality to bound the cross term $(\tstar T \Phi) \Phi$, the estimate \eqref{back:eq:current} follows. The coercivity \eqref{back:eq:coercivity} is similar, where we repeatedly make use of $|\partial f|^2 < \frac{1}{2}$, as well as $\eta(T, N) = - (1 - |\partial f|^2)^{-1/2}$.
\end{proof}

We now use this formalism to prove an energy estimate for the operator $P_{\phi_0} = \Box + p \phi_0^{p-1}$.
\begin{prop}\label{en:prop:main}
	Let $f\in H^{s+1}(\BB_3)$ for $s>\geq\frac{N+4}{2}$.
    Let $\psi,F$ be smooth functions supported in $\{ \tstar \geq \tstar_0 > 0\} \cap \Mt$ solving $P_{\phi_0}\psi=F$. Then, for $q>q_0=\frac{10\kappa_p}{1-\abs{\partial f}^2}$, and $k\in\N_{<s}$ and $j\in\N$ the following estimate holds
	\begin{equation}\label{back:eq:high}
		\norm{(\tstar \partial_\tstar)^j\Vt\psi}_{\Hb^{k;q+2}(\Mt)}^2 \lesssim_{k, q,j} \norm{(\tstar \partial_\tstar)^jF}_{\Hb^{k;q}(\Mt)}^2.
	\end{equation}
\end{prop}
\begin{remark}
	This proposition is far from optimal with regards to the minimum value of $q_0$.
	We believe that $q_0=\beta_p-2$ is possible with the same methods by optimising the constants and restricting $\abs{\partial f}$ to be smaller.
\end{remark}

\begin{proof}
	\emph{Step 1:}
	Let us start with the uncommuted estimate $k=0$ and fix $q>4\kappa_p$ sufficiently large to be determined during the proof.
	We apply the divergence theorem to the current $\J^{\mu} = \J^{\mu}[\psi, Z, Q]$ considered in \cref{back:lem:energy_est}. Using this lemma, we know that
	\[
	\nabla_{\mu} \J^{\mu}[\psi, Z, Q] \geq \tstar^{-2q-1} (- p \phi_0^{p-1} \psi + F) \tstar T \psi + \frac{q}{4} \tstar^{-2q-3} |\Vt \psi|_{\ell^2}^2+\frac{q\kappa_p^2\psi^2}{4}.
	\]
	
	Next, using that $\phi_0 = (1 - |\partial f|^2)^{\frac{1}{p-1}} c_p \tstar^{- \alpha_p}$, with $\alpha_p = \frac{2}{p-1}$ and $pc_p^{p-1} = \frac{2p (p+1)}{(p-1)^2} = \gamma_p<2\kappa_p^2$, we find that whenever $\frac{q\kappa_p}{8} > \gamma_p(1 - |\partial f|)^2$, we can absorb $\tstar^{-2q-1} p \phi_0^{p-1} |\psi \, \tstar T \psi| \leq \frac{q}{8} |\Vt \psi|_{\ell^2}^2+\frac{q\kappa_p^2}{8}\psi^2$, and we therefore have
	\[
	\nabla_{\mu} \J^{\mu} [\psi, Z, Q] \geq  \tstar^{-2q-3} \left( \frac{q}{8}|\Vt \psi|_{\ell^2}^2 + \tstar T \psi \cdot \tstar^2 F \right).
	\]
	
	
	Since both components of $\J^{\mu}$ are past directed causal, so is $\J^{\mu}$ itself. Therefore, upon applying the divergence theorem in $\Mt$, the contributions from any boundary integral with future-directed timelike normal is positive, and can be ignored. On the other hand, due to our assumption of compact support in $\{ \tstar \geq \tstar_0 \}$, the contribution from the boundary integral at $\Sigma_{\tstar_0}$ can also be assumed to vanish.
	
	As a result, the divergence theorem yields that
	\[
	\frac{q}{8} \int_{\Mt} \tstar^{-2q-3} |\Vt \psi|_{\ell^2}^2 \leq \int_{\Mt} \tstar^{-2q-1} F \cdot \tstar T \psi.
	\]
	Using Young's inequality yields the estimate \eqref{back:eq:high} with $k=0$ and $q\mapsto q+0.5$.

	\emph{Step 2:}		
    We now apply commutators to increase the possible value of $k$, recalling that we use the commutator vector fields $\partial_y$ and $\tstar \partial_{\tstar}$.
	Firstly, let us note the following commutation, 
	\begin{equation}
		[\partial_y,(1-\abs{\partial f}^2)^{-1}P_{\phi_0}] \psi = A \psi,
	\end{equation}
	where $A$ is a second order operator with $H^{s-1}$ coefficients in front of $\partial_{\tstar}\partial_y$ and $\partial_y^2$, whose size only depends on $\norm{f}_{C^2}$. In particular, applying the $k = 0$ case of \eqref{back:eq:high} to the equation 
	\[P_{\phi_0} \partial_y \psi = (1 - |\partial f|^2) A \left(  \psi \right) + (1 - |\partial f|^2) \partial_y ((1 - |\partial f|^2)^{-1} F), \]
	we find that
	\begin{multline*}
		\| \Vt \psi \|_{\Hb^{1; q+2}}^2 \lesssim_f \| F \|_{\Hb^{1; q}}^2 + \| ( \partial_y \partial_{\tstar} \psi, \partial_y^2 \psi) \| _{\Hb^{0; q}}^2\lesssim \| F \|_{\Hb^{1; q}}^2 + \| \tstar^{-1}\Vt\psi\| _{\Hb^{0; q}}^2\\
		\lesssim \| F \|_{\Hb^{1; q}}^2 + \tstar_1^2\| \Vt\partial_y\psi\| _{\Hb^{1; q+2}}^2,
	\end{multline*}
	where, in the final step, we used that $\tstar \leq \tstar_1$.
	Thus, for $\tstar_1$ sufficiently small, the last term can be absorbed and the estimate \eqref{back:eq:high} for $k=1$ follows. 
	Similarly, one can commute with $\partial_y^{\alpha}$ for any multi-index $\alpha$, yielding \eqref{back:eq:high} for any $k$, with implicit constant depending on $\| f \|_{H^{k+1}}$.

    In order to recover $\tstar \partial_{\tstar}$ derivatives as well, we  use the form of $\Box$ in \eqref{setup:eq:box} to express $\partial_{\tstar}^2 \psi$ in terms of other derivatives to get
	\begin{equation}
		\norm{(\tstar\partial_{\tstar})^2\psi}_{\Hb^{0;q+2}}\lesssim_f \norm{\psi}_{\Hb^{0;q+2}}+\norm{\Vt\partial_y\psi}_{\Hb^{0;q-1}}+\norm{F}_{\Hb^{0;q}}.
	\end{equation}
	Higher order $\tstar\partial_{\tstar}$ derivatives following similarly.
	
	\emph{Step 3:}
	To show that $\partial_{\tstar}$ commutations can be achieved independently of the $\partial_y$ commutations, we compute that
	\begin{equation}
		P_{\phi_0}\partial_{\tstar}\phi=-\frac{2}{\tstar^3}(1-\abs{\partial f}^2)\phi+\partial_{\tstar}F.
	\end{equation}
	Applying the \cref{back:eq:high} with $q\mapsto q+1$ and bounding the first term on the right hand side with the previous estimates yields the result.
\end{proof}

To apply the linear estimate above, we consider the linearised problem around an ansatz $\phi_1$ with decay order $q_0$, where $q_0$ is as given in \cref{en:prop:main}.
Let us write a solution to \cref{in:eq:main} as $\phi_1+\bar{\phi}$, where $\bar{\phi}$ satisfies
\begin{equation}\label{back:eq:nonlinear}
	-P_{\phi_0}\bar{\phi}=\mathfrak{E}_{\phi_1}[\bar\phi]+\mathcal{N}_{\phi_1}[\bar{\phi}]+\underbrace{\Box\phi_1+\phi_1^p}_{=:F_1}.
\end{equation}
We show that provided we cutoff $F_1$, there is a existence interval for $\bar{\phi}$ uniform in where we cutoff $F_1$, and thus, by compactness, a solution to \cref{back:eq:nonlinear}.

\begin{lemma}\label{en:lem:correction}  
	Let $k\geq5+\frac{n+1}{2}$ and $q>q_0+1$ for $q_0$ as in \cref{en:prop:main}.
    Fix $f\in\Hb^{k+1}(\BB_3)$,
	$F_1\in\Hb^{k;q}(\Mt)$ and $\phi_1=\phi_0+\Hc^{k;1-\alpha_p}(\M)$.
	There exists a $\tstar_1$ sufficiently small such that a unique solution $\bar{\phi}\in\Hb^{k-1;q-1}(\Mt)$ for \cref{back:eq:nonlinear} exists in $\Mt$.
	Here, uniqueness is within the space $\Hb^{k-1;q-1}(\Mt)$.
	
	Assuming $(\tstar\partial_{\tstar})^jF\in\Hb^{k;q}(\Mt)$ the same improvement holds for $\bar{\phi}$, i.e. $(\tstar\partial_{\tstar})^j\bar{\phi}\in\Hb^{k-1;q-1}(\Mt)$.
\end{lemma}
\begin{proof}
	We construct $\bar{\phi}$ via compactness from a sequence $\bar{\phi}_m$ defined as follows.
	Let $\bar{\phi}_m$ be a solution of \cref{back:eq:nonlinear} with $F_1\in\Hb^{k;q}(\Mt)$ multiplied by $\chi(\tstar/2^m)$ and $\bar{\phi}_m$ vanishing for $\tstar$ small.
	There exist $\tstar_1$ sufficiently small, such that for all $m>-\log \tstar_1$, $\bar{\phi}_m$ exists in $(0,\tstar_1)$ and the following uniform estimate holds
	\begin{equation}\label{back:eq:boundedness}
		\norm{\Vt\bar{\phi}_m(\Mt)}_{\Hb^{k;q}}\leq1.
	\end{equation}
	
	\emph{Step 1:}
	We apply \cref{en:prop:main} to obtain
	\begin{equation}
		\norm{\Vt\bar{\phi}_m}_{\Hb^{k;q+2}}\lesssim_f\norm{F_1}_{\Hb^{k;q}}+\norm{\mathfrak{E}_{\phi_1}[\bar{\phi}_m]+\mathcal{N}_{\phi_1}[\bar{\phi}_m]}_{\Hb^{k;q}}
	\end{equation}
	We use \cref{map:eq:N_H} from \cref{map:lem:nonlin} to bound the last two terms as
	\begin{equation}
		\norm{\mathfrak{E}_{\phi_1}[\bar{\phi}_m]+\mathcal{N}_{\phi_1}[\bar{\phi}_m]}_{\Hb^{k;q}}\lesssim_{\phi_1}\norm{\bar{\phi}_m}_{\Hb^{k;q+1}}
	\end{equation}
	under the assumption $\norm{\bar{\phi}_m}_{\Hb^{k;q}}\leq1$.
	By choosing $\tstar_1$ sufficiently small, we can make $\norm{F_1}_{\Hb^{k;q}(\M_{\tstar_1})}$ arbitrarily small, and obtain \cref{back:eq:boundedness} by bootstrap.
	
	\emph{Step 2:}
	Using that $\Hb^{k,q}\to\Hb^{k-1,q-1}$ is a compact embedding, we know that $\bar{\phi}_m\to\bar{\phi}$ in $\Hb^{k-1,q-1}$.
	This $\bar{\phi}$ solves \cref{back:eq:nonlinear} and satisfies $\norm{\bar{\phi}}_{\Hb^{k-1,q-1}}\leq1$.
	
	\emph{Step 3:}
	In order to show uniqueness, assume that $\bar{\phi}_1,\bar{\phi}_2\in \Hb^{k-1,q-1}(\M)$ and that both solve \cref{back:eq:nonlinear}.
	Let's write $\phi_\Delta=\bar{\phi}_1-\bar{\phi}_2$, so that
	\begin{equation}\label{back:eq:linearised}
        P_{\phi_0} \phi_{\Delta} -\mathfrak{E}_{\phi_1}[\phi_\Delta]=\mathcal{N}_{\phi_1}[\bar{\phi}_1]-\mathcal{N}_{\phi_1}[\bar{\phi}_2].
	\end{equation}
	The right hand side is linear in $\phi_\Delta$ captured by \cref{map:eq:N_linearised}.
	Let us multiply \cref{back:eq:linearised} with a cutoff $\chi_n:=\chi(\tstar/2^n)$, and
	apply the estimate \cref{back:eq:high} to \cref{back:eq:linearised} with $(k-1,q-1)$ to obtain
	\begin{equation}
		\norm{\chi_n\phi_\Delta}_{\Hb^{k-1;q+1}(\M_{\tstar_2})}\lesssim\tstar_2\norm{\chi_n\phi_\Delta}_{\Hb^{k-1;q+1}(\M_{\tstar_2})}+\sum_{i\in{1,2}}\norm{\left(\chi_{n+2}-\chi_n\right)\phi_i}_{\Hb^{k-1;q+1}(\M_{\tstar_2})}.
	\end{equation}
	The right hand term vanishes in the limit $n\to\infty$ and therefore, for $\tstar_2$ sufficiently small, we obtain $\phi_\Delta=0$.
	The uniqueness in $\{\tstar\in(\tstar_2,\tstar_1)\}$ follows from standard uniqueness for the wave equation.
\end{proof}

\begin{theorem}\label{back:thm:main}
		Let $s,f,\psi$ be as in \cref{in:thm:main}.
		Then, there exists a unique singular solution $\phi$ to \cref{in:eq:main} of regularity $s-12\kappa_p$ as in \cref{geo:def:scattering_sol}.
\end{theorem}

\begin{proof}
	Let $q_0$ be as in \cref{en:prop:main}.
	Using \cref{an:prop:ansatz}, we set $\phi_1$ be an ansatz of decay $q_0+\floor{2\kappa_p}+3$ and regularity $s>q_0+\floor{2\kappa_p}+10+\frac{n+2}{2}$, so that $P[\phi_1]\in\Hc^{s-q_0-\floor{2\kappa_p};q_0}$.
	Let us write the solution we are to construct as $\phi=\phi_1+\bar{\phi}$.
	Then \cref{en:lem:correction} yields $\bar{\phi}\in\Hb^{s-q_0-\floor{2\kappa_p}-1;q_0-1}$ the solution.

    For uniqueness, assume that two solutions $\phi_1,\phi_2$ exist with an expansion \cref{geo:eq:sol_form} with $\bar\phi_{1},\bar\phi_{2}$ denoting the remainder in $\Hb^{s-30\kappa_p;\beta_p'}$.
    We may write the equation satisfied by $\bar\phi_{\Delta}=\bar\phi_1-\bar\phi_2$ as
    \begin{equation}
        P_0\bar\phi_\Delta+P_f\bar\phi_\Delta=\mathcal{N}_{\phi_0}[\phi_2+\bar\phi_\Delta]-\mathcal{N}_{\phi_0}[\phi_1].
    \end{equation}
    Using the same argument as in \cref{back:sec:ansatz}, we can improve to get $\bar\phi_\Delta\in \Hb^{s-30\kappa_p-j;\beta_p'+j}$.
    Picking $j=12\kappa_p$ implies $\beta_p'+j>q_0$, thus we may apply the energy estimate of \cref{en:prop:main} to conclude that $\bar\phi_\Delta=0$.
\end{proof}

    \section{Stability: Proof of \texorpdfstring{\cref{in:thm:main2}}{Theorem 1.2}} \label{sec:forward}

In this section, we prove the stability of the singular hypersurface and the scattering data given in \cref{in:thm:main2}. As explained in \cref{in:sec:ideas}, the bulk of this proof will not use the usual Minkowskian time coordinate $t$, but instead the $\tau$ coordinate defined implicitly as $\phi = c_p \tau^{- \alpha_p}$, where the singularity $\mathscr{S}$ appears naturally as the surface $\{ \tau = 0 \}$.

That is, we use the $(\tau, z^i)$ coordinate system introduced in \cref{sub:forward_setup}. In \cref{for:sub:wave}, we collect computations involving commutator and multiplier vector fields for the nonlinear wave equations \eqref{eq:nlw_Jacobi}, which we then use to derive the main (degenerate) energy estimate in \cref{for:sub:energy}. We then use this energy estimate to get $L^{\infty}$ estimates in \cref{for:sub:transport}, where one crucially uses the ``transport'' system in \cref{prop:transport}. Finally, in \cref{for:sub:scattering} we recover the scattering data as in \cref{not:def:scattering_VW}.

\subsection{Commuted wave equations for Jacobian coefficients} \label{for:sub:wave}

We will use \eqref{eq:nlw_W} and \eqref{eq:nlw_V} to derive \emph{energy estimates} for $W$ and $V^i$. To capture the geometry of the problem and simplify computations, we once again use the geometric formalism of \cref{def:energymomentum}.
The key estimate will arise from applying the divergence theorem to the following energy current $\J^{\mu}[\Phi, Z, Q]$, which we simply refer to in the sequel as $\J^{\mu}[\Phi]$, via the following. 

\begin{lemma} \label{lem:current}
    For some $\tilde{q} > 0$, we define the multiplier vector field $Z^{\mu}$ and the vector field $Q^{\mu}$  to be
    \begin{equation} \label{eq:multiplier}
        Z^{\mu} \coloneqq \tau^{2(\tilde{q} + 1)} T^{\mu} = W \tau^{2(\tilde{q} + 1)} \left( \frac{\partial}{\partial \tau} \right)^{\mu}, \qquad Q^{\mu} \coloneqq - \frac{\kappa_p^2}{2} T^{\mu} = - \frac{\kappa_p^2}{2} W \tau^{2 \tilde{q}} \left( \frac{\partial}{\partial \tau} \right)^{\mu}
    \end{equation}
    Let $\J^{\mu}[\Phi] = \J^{\mu}[\Phi, Z, Q]$ be as in Definition~\ref{def:energymomentum}. Then one has the following divergence identity
    \begin{equation} \label{eq:divergence_J}
        \nabla_{\mu} \J^{\mu}[\Phi] = - \Omega^{-1} \tau^{2\tilde{q} - 1}
        \left[
            2 (\tilde{q} + 1) \tau^2 \mathbb{T}_{\mu\nu}[\Phi] T^{\mu} N^{\nu} + \tilde{q} \kappa_p^2 W \Phi^2 - \Omega W (\kappa_p^2 \Phi + \tau^2 \square \Phi) \cdot \tau \partial_{\tau} \Phi
        \right],
    \end{equation}
    where one may express $\T[\Phi](T, N) = \mathbb{T}_{\mu\nu}[\Phi] T^{\mu} N^{\nu}$ by the following coercive quantity: 
    \begin{equation}\label{eq:coercivity}
        \T_{\mu\nu}[\Phi]T^{\mu}N^\nu
        = \frac{W \Omega^{-1}}{2} \Big ( \Omega^2 (\partial_{\tau} \Phi)^2 + \delta^{ij} (\partial_{z^i} \Phi)(\partial_{z^j} \Phi) \Big).
    \end{equation}
\end{lemma}

\begin{proof}
    \emph{Divergence:}
    It will be useful to recall that the vector field $T^{\mu} = W \left(\frac{\partial}{\partial \tau}\right)^{\mu}$ is a Killing field of Minkowski space, and that, for
    $(\pi^Z)_{\mu\nu} = (\mathcal{L}_Z \eta)_{\mu\nu} = \nabla_{\mu} Z_{\nu} + \nabla_{\nu} Z_{\mu}$, one has
    \begin{equation} \label{eq:current_comp}
        \nabla_{\mu} \J^{\mu}[\Phi] = \frac{1}{2} \mathbb{T}_{\mu \nu} [\Phi] \left( \pi^Z \right)^{\mu \nu} - (\mbox{div} \, Q) \Phi^2 + \square \Phi \cdot Z^{\mu} \nabla_{\mu} \Phi - 2 \Phi \cdot Q^{\mu} \nabla_{\mu} \Phi.
    \end{equation}

    Now, since $T^{\mu}$ is a Killing field, it is easy to verify that
    \[
        (\mathcal{L}_Z \eta)^{\mu\nu} = 2 (\tilde{q} + 1) \tau^{2 \tilde{q} + 1} \left( \nabla^{\mu} \tau \, T^{\nu} + T^{\mu} \, \nabla^{\nu} \tau \right) \quad \text{ and } \quad \mbox{div} \, Q = -\tilde{q} \kappa_p^2 \tau^{2\tilde{q} - 1} T^{\mu} \,\nabla_{\mu} \tau.
    \]
    The result then follows upon inserting $\nabla^{\mu} \tau = - \Omega^{-1} N^{\mu}$ and using that $\eta(N, T) = - \Omega^{-1} W$.

    \vspace{0.5em}
    
    \emph{Coercivity:}
    To compute $\T_{\mu\nu}[\Phi] N^{\mu} T^{\mu}$ we shall use that $T = W \frac{\partial}{\partial \tau}$ and $N = \Omega \frac{\partial}{\partial \tau} - \Omega^{-1} V^i \frac{\partial}{\partial z^i}$, yielding that $\eta(T, N) = - \Omega^{-1} W$ and thus that
    \begin{align*}
        \T_{\mu\nu}[\Phi] N^{\mu} T^{\mu}
        &= (N \Phi) (T \Phi) + \frac{1}{2} \Omega^{-1} W \eta( \nabla \Phi, \nabla \Phi) \\[0.5em]
        &= W \frac{\partial \Phi}{\partial \tau} \left( \Omega \frac{\partial \Phi}{\partial \tau} - \Omega^{-1} V^i \frac{\partial \Phi}{\partial z^i} \right)
        + \frac{1}{2} \Omega^{-1} W \left \{ 
            - \Omega^2 \left( \frac{\partial \Phi}{\partial \tau} \right)^2 + 2 V^i \frac{\partial \Phi}{\partial \tau} \cdot \frac{\partial \Phi}{\partial z^i}
            + \delta^{ij} \frac{\partial \Phi}{\partial z^i} \frac{\partial \Phi}{\partial z^j}
        \right \},
    \end{align*}
    from which \cref{eq:coercivity} immediately follows.
\end{proof}

Next, we introduce some schematic notation, to simplify the outputs of certain vector field commutation computations.
Firstly, we introduce the quantities that we will eventually show to be uniformly bounded.
\begin{definition}
    We use $\vec{\Phi}$ to represent the following vector in $\R^{n+2}$:
    \[
        \vec{\Phi} = ( \kappa_p \Phi, \tau \partial_{\tau} \Phi, \tau \partial_{z^k} \Phi)^{\top}.
    \]
    Similarly, we use $\vec{W}$ and $\vec{V}^i$ to represent the vectors $(\kappa_p(W-1), \tau \partial_{\tau} W, \tau \partial_{z^k} W)^{\intercal}$ and $(\kappa_p V^i, \tau \partial_{\tau} V^i, \tau \partial_{z^k} V^i)^{\intercal}$ respectively. We will also use the notation $\vec{J}$ to represent the vector in $\R^{(n+1)(n+2)}$ given by the concatenation of $\vec{W}$ and $\vec{V}^i$ for all $i$.
\end{definition}

Next, we introduce the vector $\vec{J}_0$ capturing quantities that converge\footnote{one may compare the components of $\vec{J}_0$ to \cref{not:def:scattering_VW} for some explanation for their vanishing} to $0$ as $\tau \to 0$.

\begin{definition} \label{def:j_0}
    We also define $\vec{J}_0$ to represent the following, where we enumerate over all $i, j$.
    \[
        \vec{J}_0 = (\tau \partial_{\tau} W, \tau \partial_{z^j} W, \tau^2 \partial_{\tau}^2 W, \tau^2 \partial_{\tau} \partial_{z^j} W, \tau \partial_{\tau} V^i, \tau \partial_{z^j} V^i, \tau^2 \partial_{\tau}^2 V^i, \tau^2 \partial_{\tau} \partial_{z^j} V^i, \Omega^2 - 1)^{\intercal}.
    \]
\end{definition}

\begin{notation}\label{not:multiforms}
    In the following, bold font symbols such as $\mathbf{L}, \mathbf{B}, \mathbf{Q}_{\bullet}$ will represent linear, bilinear or multilinear forms on suitable vector spaces whose coefficients will be allowed to depend (smoothly) on the value of $W$ and $V^i$. The linear forms $\mathbf{L}$ will always act on top order terms, while the bilinear forms $\mathbf{B}$ will be such that one of its entries is a top order term, while the other entry will be $\vec{J}_0$. On the other hand, the entries of $\mathbf{Q}_{\bullet}$ will all be lower order terms in a suitable sense. The bullet subscript ${}_{\bullet}$ on $\mathbf{Q}_{\bullet}$ and sometimes on $\mathbf{B}_{\bullet}$  will be used as a shorthand to denote that these coefficients will depend on the indices $\alpha$ and $\alpha_{\ell}$ but it is not necessary to make this dependence explicit. 

    Due to the coefficients of $\mathbf{L}, \mathbf{B}, \mathbf{Q}_{\bullet}$ only depending on $V^i, W$ and not their derivatives, the bootstrap \eqref{eq:vw_bootstrap} to be introduced later will imply that these forms have bounded operator norms. Another bootstrap, \eqref{eq:firstorder_bootstrap}, will show that the objects in $\vec{J}_0$ will further have good vanishing properties towards $\tau \to 0$, which has the effect that the bilinear forms $\mathbf{B}(\cdot, \vec{J}_0)$ can be estimated easily.
\end{notation}

Next, we derive some lemmas that describe what happens when one commutes the equations \eqref{eq:nlw_W} and \eqref{eq:nlw_V} with the coordinate vector fields $\partial_z^{\alpha}$.
The nonlinear terms from commuting $\partial_z^{\alpha}$ with $\Box$ will be either lower order in derivatives, or contain top order terms which are multiplied by terms in $\vec{J}_0$ and thus vanish suitably as $\tau \to 0$.
However, $\partial_z^\alpha(\Box\Phi)$ contains leading order terms and bounding these will yield the lower bounds on the decay rate of energy estimates $\tilde{q}$.

\begin{lemma} \label{lem:commutation}
    Suppose that $V^i$ and $W$ are solutions of the system \eqref{eq:nlw_W} and \eqref{eq:nlw_V}. Then for $\alpha$ an arbitrary multi-index with $|\alpha| \geq 1$, one has that for $\Phi \in \{ W \} \cup \{ V^i: i = 1, \ldots, n\}$,
    \begin{equation} \label{eq:comm_J}
        [\tau^2 \square, \partial_z^{\alpha}] \Phi = \mathbf{B}_{comm} (\partial_z^{|\alpha|} \vec{J}, \vec{J}_{0}) + 
        \sum_{\substack{L \geq 2, \, |\alpha_{\ell}| \geq 1 \\ \sum|\alpha_{\ell}|= |\alpha|}}
        \mathbf{Q}_{comm, \bullet} (\partial_z^{\alpha_1} \vec{J}, \ldots, \partial_z^{\alpha_L} \vec{J}),
    \end{equation}
    where the notation $\partial_z^{|\alpha|} \vec{J}$ will represent the concatenation of $\partial_z^{\beta} \vec{J}$ for all multi-indices $\beta$ with $|\beta| = |\alpha|$.
\end{lemma}

\begin{proof}
    We recall the form of the wave operator $\Box$ in $(\tau, z^i)$ coordinates from \eqref{eq:wave_op}. Since $\partial_{z}^{\alpha}$ commutes with the operators $\tau \partial_{\mu}$, it follows that
    \begin{multline} \label{eq:commutation_expansion}
        [ \tau^2 \square, \partial_z^{\alpha} ] \Phi = - \partial_z^{\alpha} (\eta^{-1})^{\mu\nu} \cdot \tau^2 \partial_{\mu} \partial_{\nu} \Phi
        - \partial_z^{\alpha} ( W \tau \partial_{\mu} (W^{-1} (\eta^{-1})^{\mu\nu}))) \cdot \tau \partial_{\nu} \Phi \\[0.5em]
        + \sum_{\substack{|\beta| + |\gamma| = |\alpha| \\ |\beta|, |\gamma| \geq 1}} \partial_z^{\beta} (\eta^{-1})^{\mu\nu} * \partial_z^{\gamma} (\tau^2 \partial_{\mu} \partial_{\nu} \Phi)
        + \sum_{\substack{|\beta| + |\gamma| = |\alpha| \\ |\beta|, |\gamma| \geq 1}} \partial_z^{\beta} ( W \tau \partial_{\mu} (W^{-1} (\eta^{-1})^{\mu\nu}))) * \partial_z^{\gamma} (\tau \partial_{\nu} \Phi).
    \end{multline}

    Note that for $\Phi$ being one of $W$ or $V^i$, the expressions $\tau^2 \partial_{\mu} \partial_{\nu} \Phi$ and $\tau \partial_{\nu} \Phi$ in the first line may be written as $\tau \partial_z \vec{J}$, and are also included in $\vec{J}_0$. Therefore, upon expressing $(\eta^{-1})^{\mu\nu}$ in terms of $W$ and $V^i$ using \eqref{eq:invmetric}, it is straightforward to see the terms on the first line can either be included in either the $\mathbf{B}_{comm}$ or the $\mathbf{Q}_{comm, \bullet}$ terms in \eqref{eq:comm_J}. Similarly, the final term in \eqref{eq:commutation_expansion} can straightforwardly be included in the $\mathbf{Q}_{comm, \bullet}$ part of \eqref{eq:comm_J}.

    It remains to deal with the first term on the second line of \eqref{eq:commutation_expansion}. Using the wave equations \eqref{eq:nlw_W} and \eqref{eq:nlw_V} to express $\tau^2 \partial_{\tau}^2 \Phi$ in terms of expressions involving $\tau \partial_{z^i}$ and lower order terms, we can rewrite the object $\tau^2 \partial_{\mu} \partial_{\nu}$ schematically as
    \begin{equation} \label{eq:top_schematic}
        \tau^2 \partial_{\mu} \partial_{\nu} \Phi = \tau \partial_{z} \vec{J} \quad \text{ or } \quad \Omega^{-2} ( \tau \partial_z \vec{J} + \vec{J} * \tau \partial_z \vec{J} + \vec{J} + \vec{J} * \vec{J} * \vec{J}).
    \end{equation}
    Consider now the expression $\partial_z^{\beta} (\eta^{-1})^{\mu\nu} * \partial_z^{\gamma} (\tau^2 \partial_{\mu} \partial_{\nu} \Phi)$. Inserting \eqref{eq:top_schematic} and expanding, any terms not involving $\tau \partial_z \vec{J}$ are straightforwardly included in the $\mathbf{Q}_{comm, \bullet}$ part of \eqref{eq:comm_J}. For terms involving $\tau \partial_z \vec{J}$, we note that, since $|\beta| \geq 1$ and $(\eta^{-1})^{\mu\nu}$ is one of $V^i$ or $\sum (V^i)^2 - W^2$, there exists some $\beta'$ with $|\beta'| = |\beta| -1$ such that one may write
    \[
        \partial_z^{\beta} (\eta^{-1})^{\mu\nu} = \tau^{-1} \partial_z^{\beta'} \vec{J} \quad \text{ or } \tau^{-1} \partial_z^{\beta'} (\vec{J} * \vec{J}).
    \]

    This means we can distribute both the derivative and the $\tau$ from any $\tau \partial_z \vec{J}$ onto $\partial_z^{\beta} (\eta^{-1})^{\mu\nu}$. For example, we can rewrite $ \partial_z^{\beta} (\eta^{-1})^{\mu\nu} * \partial_z^{\gamma} \tau \partial_z \vec{J}$ as
    \[
        \partial_z^{\beta'} \vec{J} * \partial_z^{\gamma} \partial_z \vec{J} \quad \text{ or } \quad \partial_z^{\beta'} (\vec{J} * \vec{J}) * \partial_z^{\gamma} \partial_z \vec{J}.
    \]
    If $|\gamma| \leq |\alpha| - 2$, so that $|\beta'| \geq 1$, then such terms can be included in the $\mathbf{Q}_{comm, \bullet}$ part of \eqref{eq:comm_J}. On the other hand, if $|\gamma| = |\alpha| - 1$, so that $\beta'$ is empty, then we note that $\partial_z^{\beta}(\eta^{-1})^{\mu\nu}$ may further be written as $\vec{J}_0$ or $\vec{J}_0 * \vec{J}$, while the top order term $\partial_z^{\gamma} \partial_z \vec{J}$ is included in $\partial_z^{|\alpha|} \vec{J}$, so such expressions are instead included in the $\mathbf{B}_{comm}$ part of \eqref{eq:comm_J}. All other terms follow similarly.
\end{proof}

\begin{lemma} \label{lem:inhom}
    Suppose that $W$ and $V^i$ are solutions of \eqref{eq:nlw_W} and \eqref{eq:nlw_V}. Then, for $|\alpha| > 1$,
    \begin{gather} 
        \label{eq:inhom_W}
        \partial_z^{\alpha} \left( \tau^2 \square W \right) = \mathbf{L}_W(\partial_z^{\alpha} \vec{J})
        + \mathbf{B}_{in, \bullet}(\partial_z^{\alpha} \vec{J}, \vec{J}_0)  + 
        \sum_{\substack{J \geq 2, \, |\alpha_j| \geq 1 \\ \sum|\alpha_j| = |\alpha|}}
        \mathbf{Q}_{in, \bullet} (\partial_z^{\alpha_1} \vec{J}, \ldots, \partial_z^{\alpha_J} \vec{J}), \\[0.5em]
        \label{eq:inhom_V}
        \partial_z^{\alpha} \left( \tau^2 \square V^i \right) = \mathbf{L}_{V^i}(\partial_z^{\alpha} \vec{J})
        + \mathbf{B}_{in, \bullet}(\partial_z^{\alpha} \vec{J}, \vec{J}_0)  + 
        \sum_{\substack{J \geq 2, \, |\alpha_j| \geq 1 \\ \sum|\alpha_j| = |\alpha|}}
        \mathbf{Q}_{in, \bullet} (\partial_z^{\alpha_1} \vec{J}, \ldots, \partial_z^{\alpha_J} \vec{J}),
    \end{gather}
    where the linear terms $\mathbf{L}_W$ and $\mathbf{L}_{V^i}$ have coefficients explicitly depending on $V^i$ and $W$ as follows:
    \begin{gather}
        \label{eq:linear_W}
        \mathbf{L}_W(\partial_{z}^{\alpha} \vec{J}) = 2 \kappa_p W^2 ( \partial_z^{\alpha} W - \partial_z^{\alpha} \tau \partial_{\tau} W) - 2 \kappa_p \sum_i V^i W ( \partial_z^{\alpha} V^i - \partial_z^{\alpha} \tau \partial_{\tau} V^i), \\[0.5em]
        \label{eq:linear_V}
        \begin{split}
            \mathbf{L}_{V^i}(\partial_z^{\alpha} \vec{J}) =  
            2 \kappa_p W ( V^i \partial_z^{\alpha} W - V^i \partial_z^{\alpha} \tau \partial_{\tau} W - \partial_z^{\alpha} \tau \partial_{z^i} W) \hspace{10em} \\
            \hspace{10em} - 2 \kappa_p \sum_j V^j ( V^i \partial_z^{\alpha} V^j - V^i \partial_z^{\alpha} \tau \partial_{\tau} V^j - \partial_z^{\alpha} \tau \partial_{z^i} V^j).
        \end{split}
    \end{gather}
\end{lemma}

\begin{proof}
    Here, we simply apply $\partial_z^{\alpha}$ to the right hand sides of \eqref{eq:nlw_W} and \eqref{eq:nlw_V}. For instance, we have
    \[
        \partial_z^{\alpha} \left( - \kappa_p (1 - \Omega^2) W \right)
        = - \kappa_p (1 - \Omega^2) \partial_z^{\alpha} W + 2 \kappa_p W ( W \partial_z^{\alpha} W - \sum_i V^i \partial_z^{\alpha} V^i ) + \text{lower order terms. }
    \]
    Now, while we keep the contribution of $2 \kappa_p W ( W \partial_z^{\alpha} W - \sum_i V^i \partial_z^{\alpha} V^i)$ within the linear term \eqref{eq:linear_W}, we use that $\Omega^2 - 1$ is part of $\vec{J}_0$ to put $- \kappa_p (1 - \Omega^2) \partial_z^{\alpha} W$ on the $\mathbf{B}_{in, \bullet}$ part of \eqref{eq:inhom_W}. The lower order terms are then included in the multilinear contribution $\mathbf{Q}_{in, \bullet}$. The remaining computations follow similarly.
\end{proof}

\begin{corollary} \label{cor:wave_error}
    Suppose that the following hold in any region $\mathcal{D} \subset \mathcal{M}$.
    \begin{equation} \label{eq:apriori_bounds}
        W, \Omega, W^{-1}, \Omega^{-1} \leq \sqrt{2},  \quad \sum_i|V_i|^2 \leq \frac{1}{2}, \quad
        | \vec{J}_0 |_{\ell^2} \leq \epsilon^{1/2} \tau^{1/2}.
    \end{equation}

    Then, for $\Jac$ denoting the set $\{W-1, V^1, \ldots, V^n \}$, we have, for $\alpha$ an arbitrary multi-index with $|\alpha| = K \geq 1$, the following bound in the region $\mathcal{D}$, where $D > 0$ is some fixed constant. 
    \begin{multline} \label{eq:bulk_estimate}
        \sum_{\Phi \in \Jac }|\Omega W (\kappa_p^2 \partial_z^{\alpha} \Phi + \tau^2 \square \partial_z^{\alpha} \Phi) (\tau \partial_{\tau} \partial_z^{\alpha} \Phi)| \leq 
        100 \kappa_p \sum_{\Phi \in \Jac} \left( 2 \tau^2 \mathbb{T}_{\mu\nu}[\partial_z^{\alpha} \Phi] T^{\mu} N^{\nu} + \kappa_p^2 (\partial_z^{\alpha} \Phi)^2 \right) \\[0.5em]
        + D \epsilon^{1/2}\tau^{1/2} \sum_{\substack{\Phi \in \Jac \\ \beta = |\alpha|}} [2 \tau^2 \mathbb{T}_{\mu\nu}[\partial_z^{\beta} \Phi] T^{\mu} N^{\nu} + \kappa_p^2 ( \partial_z^{\beta} \Phi )^2 ]
        + \left |  \sum_{\substack{L \geq 2, |\alpha_{\ell}| \geq 1 \\ \sum |\alpha_{\ell}| = K}} \mathbf{Q}_{\bullet} (\partial_z^{\alpha_1} \vec{J}, \ldots, \partial_z^{\alpha_L} \vec{J}) \cdot \partial_z^{\alpha} \tau \partial_{\tau} \Phi
        \right |.
    \end{multline}
    At zeroth order we also have the simpler bound
    \begin{equation} \label{eq:bulk_estimate_0}
        \sum_{\Phi \in \Jac } |\Omega W (\kappa_p^2 \Phi + \tau^2 \square \Phi) (\tau \partial_{\tau} \Phi)| \leq 100 \kappa_p \sum_{\Phi \in \Jac} \left( 2 \tau^2 \mathbb{T}_{\mu\nu}[\Phi] T^{\mu} N^{\nu} + \kappa_p^2 \Phi^2 \right).
    \end{equation}
\end{corollary}

\begin{proof}
    For any $\Phi \in \Jac$, we combine Lemma~\ref{lem:commutation} and Lemma~\ref{lem:inhom} to find that
    \begin{align*}
        \tau^2 \square \partial_{z}^{\alpha} \Phi 
        &= [\tau^2 \square, \partial_z^{\alpha} \Phi] + \partial_z^{\alpha} \Phi \\[0.5em]
        &= \mathbf{L}_{\Phi}(\partial_z^{\alpha} \vec{J}) + \mathbf{B}_{\bullet}(\partial_z^{\alpha} \vec{J}, \vec{J}_0) + \sum_{\substack{L \geq 2, |\alpha_{\ell}| \geq 1 \\ \sum{|\alpha_{\ell}| = L }}} \mathbf{Q}_{\bullet} (\partial_z^{\alpha_1} \vec{J}, \ldots, \partial_z^{\alpha_L} \vec{J}),
    \end{align*}
    where we simply define $\mathbf{B}_{\bullet} = \mathbf{B}_{comm, \bullet} + \mathbf{B}_{in, \bullet}$ and $\mathbf{Q}_{\bullet} = \mathbf{Q}_{comm, \bullet} + \mathbf{Q}_{in, \bullet}$.

    Thereby, we have that
    \begin{align*}
        \Omega W (\kappa_p^2 \partial_z^{\alpha} \Phi + \tau^2 \square \partial_z^{\alpha} \Phi) (\tau \partial_{\tau} \partial_z^{\alpha} \Phi) 
        &= \underbrace{\Omega W ( \kappa_p^2 \partial_z^{\alpha} \Phi + \mathbf{L}_{\Phi}( \partial_z^{\alpha} \vec{J})) \cdot \tau \partial_{\tau} \partial_z^{\alpha} \Phi}_{\mathrm{(I)}} \\[0.5em]
        &\hspace{-5em} + \underbrace{\mathbf{B}_{\bullet}(\partial_z^{|\alpha|} \vec{J}, \vec{J}_0) \cdot W \Omega \tau \partial_{\tau} \partial_z^{\alpha} \Phi}_{\mathrm{(II)}}
        + \underbrace{\sum_{\substack{L \geq 2, |\alpha_{\ell}| \geq 1 \\ \sum{|\alpha_{\ell}| = L }}} \mathbf{Q}_{\bullet} (\partial_z^{\alpha_1} \vec{J}, \ldots, \partial_z^{\alpha_L} \vec{J}) \cdot W \Omega \tau \partial_{\tau} \partial_z^{\alpha} \Phi}_{\mathrm{(III)}}.
    \end{align*}
    The term $\mathrm{(III)}$ is kept as the final term in \eqref{eq:bulk_estimate}, upon abusing notation and absorbing the $\Omega W$ into $\mathbf{Q}_{\bullet}$. For the term $\mathrm{(II)}$, we use the bounds \eqref{eq:apriori_bounds} to write
    \[
        \mathrm{(II)} \leq D \epsilon^{1/2} \tau^{1/2} \cdot \frac{1}{4} | \partial_z^{\alpha} \vec{J} |_{\ell^2}^2 \leq D \epsilon^{1/2} \tau^{1/2} \sum_{|\beta| = |\alpha|} \sum_{\Phi \in \Jac} [ 2 \tau^2 \T_{\mu\nu} [ \partial_z^{\beta} \Phi ] T^{\mu} N^{\nu} + \kappa_p^2 (\partial_z^{\beta} \Phi)^2 ], 
    \]
    where we used \eqref{eq:coercivity}. This is exactly the penultimate term in \eqref{eq:bulk_estimate}.
        
    The main contribution, in the first line of \eqref{eq:bulk_estimate}, arises from the linear term $\mathrm{(I)}$. To bound this term, we can expand out $\mathbf{L}_{\Phi}(\partial_z^{\alpha} \vec{J})$ via \eqref{eq:linear_W} and \eqref{eq:linear_V}, sum over $\Phi \in \Jac$ then use Young's inequality repeatedly. We give some examples to illustrate this procedure:
    \begin{itemize}
        \item 
            The contribution of the second term in \eqref{eq:linear_W} is
            \begin{align*}
                | 2 \kappa_p \Omega W^3 \partial_z^{\alpha} \tau \partial_{\tau} W \cdot \partial_z^{\alpha} \tau \partial_{\tau} W | 
                &= 2 \kappa_p W^2 \cdot \Omega W (\partial_z^{\alpha} \tau \partial_{\tau} W)^2 \\
                &\leq 4 \kappa_p \cdot 2 \tau^2 \T_{\mu\nu}[\partial_z^{\alpha} W] T^{\mu} N^{\mu},
            \end{align*}
            using \eqref{eq:coercivity} in the final estimate.

        \item
            The contribution of the fourth term in \eqref{eq:linear_W}, is, using the Cauchy-Schwarz inequality for the summation in $i$ and then Young's inequality,
            \begin{align*}
                \sum_i| 2 \kappa_p \Omega V_i W^2 \partial_z^{\alpha} \tau \partial_{\tau} V^i \cdot \partial_z^{\alpha} \tau \partial_{\tau} W |
                &\leq \kappa_p W \left( \sum_i |V_i|^2 \right )^{1/2} \cdot \Omega W \left( (\tau \partial_{\tau} W)^2 + \sum_i (\tau \partial_{\tau} V^i)^2 \right) \\
                &\leq 2 \kappa_p \cdot \sum_{\Phi \in \Jac} 2 \tau^2 \T_{\mu\nu}[\partial_z^{\alpha} \Phi] T^{\mu} N^{\mu}.
            \end{align*}

        \item
            The contribution of the fourth term in \eqref{eq:linear_V}, together with summation in $i$ and $j$ is, upon using Cauchy-Schwarz for both summations, estimated by
            \begin{align*}
                \sum_{i, j} | 2 \kappa_p \Omega W V_i V_j \partial_{z}^{\alpha} V^j \cdot \partial_z^{\alpha} \tau \partial_{\tau} V^i |
                &\leq 2 \kappa_p \Omega W \left( \sum_i |V_i|^2 \right ) \left( \sum_j (\partial_z^{\alpha} V^j)^2 \right)^{1/2} \left( \sum_i (\partial_z^{\alpha} \tau \partial_{\tau} V^i)^2 \right)^{1/2} \\
                &\leq 2 \left( \sum_j \kappa_p^2 (\partial_z^{\alpha} V^j)^2 + \sum_i 2 \tau^2 \T_{\mu\nu}[\partial_z^{\alpha} V^i] T^{\mu} N^{\mu}\right).
            \end{align*}

        \item
            The contribution of the additional linear terms due to $\kappa_p^2 \partial_z^{\alpha} \Phi$ in $\mathrm{(I)}$ are estimated by
            \begin{align*}
                \sum_{\Phi \in \Jac} \Omega W | \kappa_p^2 \partial_z^{\alpha} \Phi \cdot \tau \partial_{\tau} \partial_z^{\alpha} \Phi |
                &\leq \sum_{\Phi \in \Jac} \kappa_p \sqrt{\Omega W} | \kappa_p \partial_z^{\alpha} \Phi | \cdot |\sqrt{\Omega W} \partial_z^{\alpha} \tau \partial_{\tau} \Phi | \\
                &\leq \kappa_p \sum_{\Phi \in \Jac} \left( \kappa_p^2 (\partial_z^{\alpha} \Phi)^2 + 2 \tau^2 \T_{\mu\nu}[\partial_z^{\alpha} \Phi] T^{\mu} N^{\mu} \right).
            \end{align*}
    \end{itemize}
    Systematically going through the remaining terms, and recalling that we consider only $\kappa_p \geq 1$, one eventually determines that
    \[
        \mathrm{(I)} \leq 100 \kappa_p \sum_{\Phi \in \Jac} \left( \kappa_p^2 (\partial_z^{\alpha} \Phi)^2 + 2 \tau^2 \T_{\mu\nu}[ \partial_z^{\alpha} \Phi] T^{\mu} N^{\nu} \right),
    \]
    as desired. The zeroth order estimate \eqref{eq:bulk_estimate_0} follows similarly, upon directly using the equations \eqref{eq:nlw_W} and \eqref{eq:nlw_V}.
\end{proof}


\subsection{Bootstrap assumptions and energy estimates} \label{for:sub:energy}

In what remains of \cref{sec:forward}, we often refer to the following \emph{bootstrap asssumptions}:
\begin{subequations}\label{eq:bootstrap_all}
	\begin{gather}
		\label{eq:vw_bootstrap}
		W, W^{-1}, \Omega, \Omega^{-1} \leq \sqrt{2}, \quad V^i \leq \sqrt{2}, \\
		\label{eq:firstorder_bootstrap}
		| \partial_{z^i} \vec{J} |_{\ell^2} \leq 1, \\
		\label{eq:convergent_bootstrap}
		| \vec{J}_0 |_{\ell^2} \leq \epsilon^{1/2} \tau^{1/2}.
	\end{gather}
\end{subequations}
The key proposition of \cref{for:sub:energy} will be the following, recalling the definitions of the domain $\Mhat_{\tau_0, \tau_1}$ and the hypersurfaces $\Shat_{\tau}$ from \eqref{eq:forward_regions}.

\begin{prop}  \label{prop:l2}
    Let $\tau_0<\tau_1$. Suppose that $W, V^i$ satisfy the system \eqref{eq:nlw_Jacobi} in the domain $\Mhat_{\tau_0, \tau_1}$. Furthermore,
    the bootstrap assumptions \eqref{eq:vw_bootstrap}, \eqref{eq:firstorder_bootstrap} and \eqref{eq:convergent_bootstrap} all hold in this region. Then, for all integers $K \geq 0$ and $\tilde{q} \geq 200 \kappa_p$, it holds that 
    \begin{equation} \label{eq:l2}
        \| \Vtop J \|_{H^{K, - \tilde{q}}(\Shat_{\tau_0})}^2 
        + \int_{\tau_0}^{\tau_1} \| \Vtop J \|_{H^{K, - \tilde{q}}(\Shat_{\tilde{\tau}})}^2 \frac{d \tilde{\tau}}{\tilde{\tau}}
        \lesssim_{K, \tilde{q}}
        \| \Vtop J \|_{H^{K, - \tilde{q}}(\Shat_{\tau_1})}^2.
    \end{equation}
\end{prop}

For the proof, it will be helpful to devise the following geometric energies.

\begin{definition} \label{def:energies}
    For $0 < \tau_0 \leq \tau_1$, we define the (zeroth order) energy $\mathcal{E}_{\tau_0}[\Phi]$ of a function $\Phi \in C^{K+1}(\Mhat_{\tau_0, \tau_1})$ on the hypersurface $\Shat_{\tau}$ and the bulk energy $\overline{\mathcal{E}}_{\tau_0, \tau_1}[\Phi]$  as
    \begin{align}
        \mathcal{E}_{\tau_0}[\Phi] 
        &\coloneqq
        \int_{\Shat_{\tau_0}} \J_{\mu} [\Phi] N^{\mu} \, d \mathrm{vol}
        = \int_{\Shat_{\tau_0}} \tau_0^{-2 \tilde{q}} W \left( \tau_0^2 \mathbb{T}_{\mu\nu}[\Phi] T^{\mu} N^{\nu} + \frac{1}{2} \kappa_p^2 \Phi^2 \right) dz,
    \end{align}
    \begin{equation}
        \overline{\mathcal{E}}_{\tau_0, \tau_1}[\Phi] \coloneqq
        \int_{\mathcal{M}_{\tau_0, \tau_1}} \J_{\mu} [\Phi] N^{\mu} \, \frac{d \mathrm{vol}}{\tau} = \int_{\tau_0}^{\tau_1} \mathcal{E}_{\tilde{\tau}}[\Phi] \Omega^{-1} \, \frac{d \tau}{\tau}.
    \end{equation}
    We also often use $\mathcal{E}_{\tau_0}^{(0)}[\Phi]$ and $\overline{\mathcal{E}}_{\tau_0, \tau_1}^{(0)}[\Phi]$ in place of $\mathcal{E}_{\tau_0}[\Phi]$ and $\overline{\mathcal{E}}_{\tau_0, \tau_1}[\Phi]$.

    For $K \geq 1$, we also define the $K$-th order energies (both on hypersurfaces and in the bulk) as
    \begin{equation}
        \mathcal{E}_{\tau_0}^{(K)}[\Phi] = \sum_{|\alpha| = K} \mathcal{E}_{\tau_0}[\partial_z^{\alpha} \Phi], \quad \overline{\mathcal{E}}_{\tau_0, \tau_1}^{(K)}[\Phi] = \sum_{|\alpha|=K} \overline{\mathcal{E}}_{\tau_0, \tau_1}[\partial^\alpha_z\Phi],
    \end{equation}
    and the total energies $\mathcal{E}_{\tau_0}^{(K)}[J]$ and $\overline{\mathcal{E}}_{\tau_0, \tau_1}^{(K)}[J]$ by 
    \begin{equation}
        \mathcal{E}_{\tau_0}^{(K)}[J] = \mathcal{E}_{\tau_0}^{(K)}[W - 1] + \sum_i \mathcal{E}_{\tau_0}^{(K)}[V^i], \quad
        \overline{\mathcal{E}}_{\tau_0, \tau_1}^{(K)}[J] = \overline{\mathcal{E}}_{\tau_0, \tau_1}^{(K)}[W-1] + \sum_i \overline{\mathcal{E}}_{\tau_0, \tau_1}^{(K)}[V^i].
    \end{equation}
    Finally, for notational convenience we define the following. (Similarly for $\Phi$ replaced by $J$.)
    \begin{equation}
        \mathcal{E}_{\tau_0}^{(\leq K)}[\Phi] = \sum_{k = 0}^K \mathcal{E}_{\tau_0}^{(k)}[\Phi], \qquad 
        \overline{\mathcal{E}}_{\tau_0, \tau_1}^{(\leq K)}[\Phi] = \sum_{k = 0}^K \overline{\mathcal{E}}_{\tau_0, \tau_1}^{(k)}[\Phi].
    \end{equation}

\end{definition}

\begin{lemma} \label{lem:energy_coercivity}
    Assuming the bootstrap \eqref{eq:vw_bootstrap}, for any integer $K \geq 0$ the following coercivity holds
    \begin{gather*}
        \| \Vtop \Phi \|_{H^{K, - \tilde{q}}(\Shat_{\tau})}^2 \lesssim_{K, \tilde{q}, p} 
        \mathcal{E}_{\tau_0}^{(\leq K)}[\Phi] \lesssim_{K, \tilde{q}, p}
        \| \Vtop \Phi \|_{H^{K, - \tilde{q}}(\Shat_{\tau})}^2 
    \end{gather*}
    and similarly for $\Phi$ replaced by $J$.
\end{lemma}

\begin{proof}
    Immediate from the definitions of the bounds on $\Omega$ and $W$ from \eqref{eq:vw_bootstrap} and the coercivity \cref{eq:coercivity}.
\end{proof}

\begin{proof}[Proof of \cref{prop:l2}]
    We integrate the divergence identity in Lemma~\ref{lem:current}, to yield that, for $0 < \tau_0 \leq \tau_1$, one has the following inequality for $\mathcal{E}[\Phi]$ and $\overline{\mathcal{E}}[\Phi]$.
    \begin{equation} \label{eq:current_integrated}
    \mathcal{E}_{\tau_0} [\Phi] + 2 \tilde{q} \overline{\mathcal{E}}_{\tau_0, \tau_1} [\Phi] \leq \mathcal{E}_{\tau_1} [\Phi] + \int_{\Mhat_{\tau_0, \tau_1}} \tau^{2\tilde{q}} W(\kappa_p^2 \Phi + \tau^2 \square \Phi) \cdot \tau \partial_{\tau} \Phi \, \frac{ d \mbox{vol}}{\tau}.
\end{equation}
Using \cref{eq:bulk_estimate_0} and summing over $\Phi \in \Jac = \{W, V^1, \ldots, V^n\}$, for the choice $\tilde{q} \geq 200\kappa_p$, we can absorb the bulk term involving $\int_{\Mhat_{\tau_0, \tau_1}}$ into the left hand side, yielding the desired result for $K=0$.

Similarly, using $\partial_z^{\alpha} \Phi$ in place of $\Phi$ in \eqref{eq:current_integrated} for $\Phi \in \Jac$, we can use \eqref{eq:bulk_estimate} and sum over all $|\alpha| = K$ and $\Phi \in \Jac$ to deduce the following, for multilinear forms $\mathbf{Q}_{\bullet}$ with bounded coefficients.
\begin{multline}
    \mathcal{E}^{(K)}_{\tau_0} [J] + 2 \tilde{q} \overline{\mathcal{E}}^{(K)}_{\tau_0, \tau_1}[J] \leq (200 \kappa_p + D \epsilon^{1/2} \tau^{1/2}) \overline{\mathcal{E}}^{(K)}_{\tau_0, \tau_1} [J]  + \mathcal{E}^{(K)}_{\tau_1} [J] \\[0.5em]
    + \sum_{\substack{|\alpha| = K \\ \Phi \in \Jac}} \sum_{\substack{L \geq 2, |\alpha_{\ell}| \geq 1 \\ \sum |\alpha_{\ell}| = K}}\int_{\Mhat_{\tau_0, \tau_1}} \tau^{2 \tilde{q}} \left | \mathbf{Q}_{\bullet} (\partial_z^{\alpha_1} \vec{J}, \ldots, \partial_z^{\alpha_L} \vec{J}) \cdot \partial_z^{\alpha} \tau \partial_{\tau} \Phi \right| \frac{d \mbox{vol}}{\tau}.
\end{multline}
Now, for $\tilde{q} \geq 200 \kappa_p$ and $\epsilon$ sufficiently small, we can absorb the first term on the right hand side into the left. Furthermore, one may apply Young's inequality to the final term on the right hand side and use that, assuming the bootstrap \eqref{eq:vw_bootstrap}, the coefficients of $\mathbf{Q}_{\bullet}$ are bounded, to find that
\begin{equation} \label{eq:current_integrated_2}
    \mathcal{E}^{(K)}_{\tau_0} [J] + \tilde{q} \overline{\mathcal{E}}^{(K)}_{\tau_0, \tau_1}[J] \lesssim \sum_{\substack{L \geq 2, |\alpha_{\ell}| \geq 1 \\ \sum |\alpha_{\ell}| = K}} \int_{\Mhat_{\tau_0, \tau_1}} \tau^{2 \tilde{q}} | \partial_z^{\alpha_1} \vec{J} \cdots \partial_z^{\alpha_L} \vec{J} |^2 \, \frac{d \mbox{vol}}{\tau} + \mathcal{E}^{(K)}_{\tau_1} [J].
\end{equation}
To estimate the multilinear term on the right hand side, it will be useful to apply the following product estimate; this is a mild generalization of \cite[Proposition 3.6]{TaylorPDE3}.
\begin{lemma}  \label{lem:product}
    Let $B \subset \R^n$ be a ball, and let $f_1, f_2, \ldots, f_L$ be suitably regular functions on $B$. Let $K > 0$ and $\sum |\alpha_{\ell}| = K$. Then there exists a constant $D = D(K, L)$, independent of the radius of $B$, such that the following product estimate applies.
    \begin{equation}
        \int_B \left | \partial_{x}^{\alpha_1} f_1 \cdots \partial_{x}^{\alpha_{L}} f_L \right|^2 \, dx \leq D \sum_{\ell = 1}^L \left( \| f_{\ell} \|_{H^K(B)}^2 \cdot \prod_{m \neq \ell} \| f_m \|_{L^{\infty}(B)}^2 \right).
    \end{equation}
\end{lemma}

To apply Lemma~\ref{lem:product} to the multilinear term in \eqref{eq:current_integrated_2}, we set $f_1, \ldots, f_L$ to be each be some first order derivative of $\vec{J}$, and apply bootstrap assumption \eqref{eq:firstorder_bootstrap}, to see that
\[
    \| \partial_z^{\alpha_1} \vec{J} \cdots \partial_z^{\alpha_L} \vec{J} \|^2_{L^2(\Shat_{\tau})} \lesssim \sum_{k = 1}^{K-1} \| \partial_z \vec{J} \|_{L^{\infty}}^{2(K-k)} \| \vec{J} \|^2_{H^k(\Sigma_{\tau})} \lesssim \| \vec{J} \|_{H^{K-1}}^2 \lesssim \tau^{-2 \tilde{q}} \mathcal{E}_{\tau}^{(\leq K - 1)}[J].
\]
Thus, upon expressing the multilinear expression in \eqref{eq:current_integrated_2} as a suitable integral in time of the expression $\| \partial_z^{\alpha_1} \vec{J} \cdots \partial_z^{\alpha_L} \vec{J} \|_{L^2(\Shat_{\tau})}$ and using \cref{lem:energy_coercivity}, we deduce that
\begin{equation} \label{eq:current_integrated_3}
    \mathcal{E}^{(K)}_{\tau_0} [J] + \tilde{q} \overline{\mathcal{E}}^{(K)}_{\tau_0, \tau_1}[J] 
    \lesssim \overline{\mathcal{E}}^{(\leq K - 1)}_{\tau_0, \tau_1}[J] + \mathcal{E}^{(K)}_{\tau_1} [J],
\end{equation}
from which a straightforward induction argument shows that
\begin{equation}
    \mathcal{E}^{(\leq K)}_{\tau_0} [J] + \tilde{q} \overline{\mathcal{E}}^{(\leq K)}_{\tau_0, \tau_1}[J] 
    \lesssim 
    \mathcal{E}_{\tau_1}^{(\leq K)}[J].
\end{equation}
One then derives \eqref{eq:l2} simply upon using \cref{lem:energy_coercivity}.
\end{proof}

\subsection{Interpolation and transport estimates} \label{for:sub:transport}

To improve the bootstrap assumptions we use the first-order equations in \cref{prop:transport}, which we view as ``transport equations'' with respect to $\tau \partial_{\tau}$. To overcome the ``derivative loss'' in the right hand side of these equations, we will apply an interpolation estimate, using the high order $H^s$ control derived in \cref{for:sub:energy}. 

Prior to using the transport equations, one first needs to translate the $\{ t = 1 \}$ initial data assumption in \cref{in:thm:main2}, namely \eqref{in:eq:forward_ass}, to an assumption on some $\Shat_{\tau_1}$ in the framework of \cref{sub:forward_setup}. This will be done via a \emph{Cauchy stability} argument locally in time, together with a Sobolev inequality. We have the following:

\begin{lemma}\label{lem:Cauchy}
    Fix $s_0> \frac{n}{2}+1$.
    Let $\phi = \phi(t, x)$ be a solution of \eqref{in:eq:main} satisfying the initial data assumption \eqref{in:eq:forward_ass}. Then there exists $\tau_1$ with $|1 - \tau_1| \lesssim \epsilon$ such that one has the following estimates for $W, V^i$ on the hypersurface $\Shat_{\tau_1}$, with respect to the $(\tau, z^i)$ coordinate system as in \cref{sub:forward_setup}.
    \begin{equation} \label{eq:stau1_data}
        \| W - 1 \|_{H^{s}(\Shat_{\tau_1})} + \| V^i \|_{H^s(\Shat_{\tau_1})} + \| U^i \|_{C^2(\Shat_{\tau_1})} + \| \mathring{\Omega}^2 \|_{C^2(\Shat_{\tau_1})} \lesssim \epsilon.
    \end{equation}
    Furthermore, for $\phi^{(1)}$, $\phi^{(2)}$ satisfying \eqref{in:eq:main} with initial data $(\phi_0^{(m)}, \phi_1^{(m)})$ on $\{ t = 1 \}$ as described in \cref{in:thm:main2}, one moreover has the corresponding stability estimates for the respective geometric quantities $W_{(m)}, V^i_{(m)}, U^i_{(m)}, \mathring{\Omega}^2_{(m)}$.
    \begin{multline} \label{for:eq:cauchy_stability}
        \| W_{(1)} - W_{(2)} \|_{H^{s-1}(\Shat_{\tau_1})} + \| V_{(1)}^{i} - V_{(2)}^{i} \|_{H^{s-1}(\Shat_{\tau_1})} + \| U_{(1)}^{i} - U_{(2)}^{i} \|_{C^2(\Shat_{\tau_1})} + \| \mathring{\Omega}^2_{(1)} - \mathring{\Omega}^2_{(2)} \|_{C^2(\Shat_{\tau_1})} \\[0.5em]
        \lesssim \norm{\phi_0^{(1)}-\phi_0^{(2)},\phi_1^{(1)}-\phi_1^{(2)}}_{H^{s+2}\times H^{s+1}}.
    \end{multline}
\end{lemma}

\begin{proof}
    Let us denote by $\mathscr{H}_1$ the hypersurface $\{ t(\tau, z^i) = 1, z \in \BB_6\}$ in the $(\tau, z^i)$ coordinate system. Then, for $T = \partial_t$ the standard time translation in Minkowski, and a current defined as
    \[
        \mathbb{J}^{\mu}[\Phi] = (\eta^{-1})^{\mu\nu} \T_{\nu \sigma}[\Phi] T^{\sigma} - \frac{1}{2} T^{\mu} \Phi^2,
    \]
    it follows from the transformation of $\phi = c_p \tau^{-\alpha_p}$ that the initial data assumption \eqref{in:eq:forward_ass} exactly translates to the following geometric estimate.
    \[
        \sum_{\Phi \in \{V^i, W - 1\}} \sum_{|\alpha| \leq s} \int_{\mathscr{H}_1} \J_{\nu}[\partial_z^{\alpha} \Phi] N^{\nu} d\mbox{vol} \lesssim \epsilon^2.
    \]

    Now, let us consider the region $\Mhat \coloneqq \{ \tau > 0, |z| < 2 (\tau + 2) \}$. Note that, for $\epsilon$ sufficiently small, $\mathscr{H}_1$ is a spacelike hypersurface which pierces the outer boundary of $\Mhat$, meaning that one is able to solve the system \eqref{eq:nlw_Jacobi} locally, to the past of $\mathscr{H}_1 \cap \Mhat$ in $\Mhat$. Letting $\tau_1 = \min_{\mathscr{H}_1} \tau$ and using the divergence theorem with respect to $\J^{\mu}[\Phi]$, it is straightforward to verify that one can, in particular, solve \eqref{eq:nlw_Jacobi} to the hypersurface $\Shat_{\tau_1} = \Mhat \cap \{ \tau = \tau_1 \}$, and propagate the estimate
    \[
        \sum_{\Phi \in \{V^i, W - 1\}} \sum_{|\alpha| \leq s} \int_{\Shat_{\tau_1}} \J_{\nu}[\partial_z^{\alpha} \Phi] N^{\nu} d\mbox{vol} \lesssim \epsilon^2.
    \]
    Using an extra step of Sobolev embedding, since $s \geq \frac{n}{2} + 2$, this can be shown to imply \eqref{eq:stau1_data}. The stability estimate \eqref{for:eq:cauchy_stability} is derived by taking differences between the equations; note that one extra derivative is lost as is standard for quasilinear equations.
\end{proof}

In the next lemma, we shall use the results of \cref{for:sub:energy} to propagate low order $L^{\infty}$ estimates to the entire region $\tau_0 \leq \tau \leq \tau_1$, with the goal of eventually improving our bootstrap assumptions.

\begin{lemma} \label{lem:improvement} 
    Let $K\geq200\kappa_p+\frac{n}{2}+2$.
    Consider $W, V^i$ solving the first order system of \cref{prop:transport}. Assume that, at $\tau = \tau_1$, the initial data bound \eqref{eq:stau1_data} holds. 
    Furthermore, suppose that the energy estimate \eqref{eq:l2} holds for $s=K$, along with the bootstrap assumption \eqref{eq:vw_bootstrap}. Then we have, for all $0 < \tau_0 \leq \tau_1$, the following improved $L^{\infty}$ bound:
    \begin{equation} \label{eq:linfty}
        \sum_i |U^i|^2 + |\mathring{\Omega}^2|^2 \leq \epsilon^{2/3}.
    \end{equation}

    In fact, for any multiindex $\alpha$ with $|\alpha| \leq 2$, the same conclusion holds: 
    \begin{equation} \label{eq:linfty_comm}
        \sum_i |\partial_z^{\alpha} U^i|^2 + |\partial_z^{\alpha} \mathring{\Omega}^2 |^2 \leq \epsilon^{2/3}.
    \end{equation}
\end{lemma}

\begin{proof}
    Let $\tilde{q} = 200 \kappa_p$. By the definition of $\Vtop$, and using \eqref{eq:transport_U} and \eqref{eq:transport_omega} to control $\tau \partial_{\tau}$ derivatives via appropriate product estimates, from the initial data assumption \eqref{eq:stau1_data} one can moreover find that, for $K = s \geq 200 \kappa_p + \frac{n}{2} + 2$, that we have
\begin{equation} \label{eq:data_example}
    \| \Vtop J \|_{H^{K, - \tilde{q}}(\Shat_{\tau_1})} \lesssim \epsilon,
\end{equation}
    where we note that the $\tilde{q}$ is harmless since $|\tau - 1| \lesssim \epsilon$ is small thus powers $\tau^{- \tilde{q}}_1$ are negligible. 

    We let $\vec{X}$ represent the vector $(U^i, \mathring{\Omega}^2)^{\intercal}$. Upon expressing the right hand sides of \eqref{eq:transport_U} and \eqref{eq:transport_omega} in terms of $U^i$ and $\mathring{\Omega}^2$, we find that
    \begin{equation} \label{eq:transport_matrix}
        \tau \partial_{\tau} \vec{X} = 
        \underbrace{\begin{bmatrix} 0 & 0 \\ 0 & 2 \kappa_p \end{bmatrix}}_{\mathbf{B}} \vec{X}
        + \mathbf{M}^i \tau \partial_{z^i} \vec{X},
    \end{equation}
    where the entries of the matrices $\mathbf{M}^i$ are explicit functions of $W, V^i$ and thus satisfy $\| \mathbf{M}^i \vec{X} \|_{L^{\infty}} \lesssim \| \vec{X} \|_{L^{\infty}}$, in light of the bootstrap assumption \eqref{eq:vw_bootstrap}.

    A Gagliardo--Nirenberg--Sobolev interpolation formula \cite[Proposition 3.5]{TaylorPDE3} adapted to balls $\Shat_{\tau} \cong \BB_r$ yields that, the following estimate holds.
    \begin{equation}\label{eq:Gagliardo}
        \| \partial_{z} \vec{X} \|_{L^{\infty}(\Shat_{\tau})} \lesssim \| \vec{X} \|_{L^{\infty}}^{1 - \sigma} \| \vec{X} \|_{\dot{H}^K(\Shat_{\tau})}^{\sigma}, \quad \text{ where } \sigma = \frac{1}{K - \frac{n}{2}}.
    \end{equation}
    By expressing $\partial_z^{\alpha} \vec{X}$ in terms of $\partial_z^{\alpha} \vec{J}$ and inserting \eqref{eq:l2}, to deduce that $\| \vec{J} \|_{H^K(\Shat_{\tau})} \leq \tau^{- \tilde{q}} \| \vec{J} \|_{H^{K, - \tilde{q}}(\Shat_{\tau_0})} \lesssim \tau^{- \tilde{q}} \epsilon$, we find that
    \[
        \| \mathbf{M}^i \tau \partial_{z^i} \vec{X} \|_{L^{\infty}(\Shat_{\tau})} \lesssim \| \vec{X} \|_{L^{\infty}}^{1 - \sigma} \cdot \epsilon^{\sigma} \tau^{1 - \tilde{q} \sigma}.
    \]
    Inserting this into \eqref{eq:transport_matrix}, and using that the matrix $\mathbf{B}$ is positive definite with respect to the standard $\ell^2$ norm on $\R^{n+1}$, we thereby deduce that
    \[
        \tau \partial_{\tau} \| \vec{X} \|_{L^{\infty}}^2 \gtrsim \| \vec{X} \|_{L^{\infty}}^{2 - \sigma} \cdot \epsilon^{\sigma} \tau^{1 - \tilde{q} \sigma},
    \]
    which simplifies to
    \[
        \partial_{\tau} ( \| \vec{X} \|_{L^{\infty}}^{\sigma} ) \gtrsim \sigma \epsilon^{\sigma} \tau^{- \tilde{q} \sigma}.
    \]

    Thus, since $K$ is chosen large enough so that $0 < \tilde{q} \sigma < 1$ for our choice $\tilde{q} = 200 \kappa_p$, we may integrate backwards and deduce that 
    \[ 
        \| \vec{X} \|_{L^{\infty}(\Shat_{\tau})}^{\sigma} \lesssim (1 + \sigma) \epsilon^{\sigma} \quad \forall \tau \in (0, 1),
    \]
    from which it is straightforward to deduce \eqref{eq:linfty}. 

    To derive the higher regularity estimate \eqref{eq:linfty_comm}, we first commute the equations \eqref{eq:transport_U} and \eqref{eq:transport_omega} with $\partial_z^{\alpha}$ to find that, 
    for an arbitrary multiindex $\alpha$, one has
    \begin{gather}
        \label{eq:transport_U_comm}
        \tau \partial_{\tau} \partial_z^{\alpha} U^i = W^{-2} \tau \partial_{z^i} \partial_z^{\alpha} W + \sum_{\substack{L \geq 2, |\alpha_{\ell}|\geq 1 \\ \sum |\alpha_{\ell}| = |\alpha| + 1}} \tau \boldsymbol{Q}_{U^i, \bullet} (\partial_z^{\alpha_1} W, \ldots, \partial_z^{\alpha_L} W), \\
        \label{eq:transport_omega_comm}
        \tau \partial_{\tau} \partial_z^{\alpha} \mathring{\Omega}^2 = 2 \kappa_p \partial_z^{\alpha} \mathring{\Omega}^2 + 2 \tau \partial_{z^i} \partial_z^{\alpha} V^i.
    \end{gather}
    The estimate \eqref{eq:linfty_comm} then follows by an analogous argument, together with an induction on $\abs{\alpha}$ to handle the multilinear term in \eqref{eq:transport_U_comm}.
\end{proof}

\begin{prop} \label{prop:existence}
    Suppose that \eqref{eq:stau1_data} holds for $\epsilon$ sufficiently small and for $s=K \geq 200 \kappa_p + \frac{n}{2} + A + 2$, where $A \geq 2$. Then the maximal domain of existence of $W$ and $V^i$ solving \eqref{eq:nlw_W} and \eqref{eq:nlw_V} includes the entire region $\Mhat_{\tau_1}$, and we have, for $\Phi \in \{ W - 1, V^i \}$, the following $L^2$ and $L^{\infty}$ bounds for $\tilde{q} = 200 \kappa_p$ and all $\tau \in (0, \tau_1]$.
    \begin{equation} \label{eq:existence_bounds}
        \| \Vtop \Phi \|_{H^{K, - \tilde{q}}(\Shat_{\tau})} \lesssim_K \epsilon, \qquad \| \Phi \|_{C^2(\Shat_{\tau})} \lesssim_K \epsilon.
    \end{equation}

    Furthermore, identifying the region $\{ \tau = 0, |z| < 4 \} \cong \BB_4 \subset \partial \Mhat$, there exist $C^{A}(\BB_4)$ functions $\mathfrak{W}(z)$ and $\mathfrak{V}^i(z)$ with $\| \mathfrak{W} - 1 \|_{C^A} + \| \mathfrak{V}^i \|_{C^A} \lesssim \epsilon$ such that
    \begin{equation} \label{eq:smooth_convergence}
        W(\tau, z) \to \mathfrak{W}(z), \quad V^i(\tau, z) \to \mathfrak{V}^i(z) \quad \text{ in } C^A(\BB_4) \text{ as } \tau \to 0.
    \end{equation}
    Moreover, one necessarily has $\mathfrak{W}^2(z) - \delta_{ij} \mathfrak{V}^i(z) \mathfrak{V}^j(z) \equiv 1$. 
\end{prop}

\begin{proof}
    The proof follows by a standard bootstrap argument. Supposing that the bootstrap assumptions \eqref{eq:vw_bootstrap}, \eqref{eq:firstorder_bootstrap} and \eqref{eq:convergent_bootstrap} hold on some bootstrap region $\Mhat_{boot} = \Mhat_{\tau_{boot}, 1}$, i.e.~that these assumptions hold for $\tau \in [\tau_{boot}, 1]$. Then the conclusions of \cref{for:sub:energy} and Section~\ref{for:sub:transport} hold for $\tau \in [\tau_{boot}, 1]$. 

    In order to extend the region of existence to $\Mhat = \cup_{\tau_0 > 0} \Mhat_{\tau_0, 1}$, we use Lemma~\ref{lem:improvement} to estimate the right hand sides of \eqref{eq:transport_U_comm} and \eqref{eq:transport_omega_comm} and improve the bootstrap assumptions. Since $A \geq 2$, Lemma~\ref{lem:improvement} tells us that
    \begin{equation} \label{eq:linftyx}
        \| \vec{X} \|_{C^2(\Shat_{\tau})} \leq \epsilon^{2/3} \quad \forall \tau \in [\tau_{boot}, 1],
    \end{equation}
    from which one can immediately use \eqref{eq:vwomega} -- together with its first derivatives -- to show that
    \[
        | \Omega^2 - 1 |, | W - 1 |, |V^i|, |\partial_z W|, |\partial_z V^i| \leq \epsilon^{1/2},
    \]
    which improves upon the bootstrap assumptions \eqref{eq:vw_bootstrap} and \eqref{eq:firstorder_bootstrap}.

    It remains to improve upon the final bootstrap assumption \eqref{eq:convergent_bootstrap}. Note that this follows immediately from showing that all components of the tuple
    \[
        \vec{X}_0 = (\tau \partial_{\tau} \vec{X}, \tau \partial_{z^j} \vec{X}, \tau^2 \partial_{\tau}^2 \vec{X}, \tau^2 \partial_{\tau} \partial_{z^j} \vec{X}, \mathring{\Omega}^2),
    \]
    are bounded by, say, $D \epsilon \tau$. But this is more or less immediate from \eqref{eq:linftyx} and the transport equations in \eqref{eq:transport_matrix}. To be a little more detailed, \eqref{eq:linftyx} immediately yields that $| \tau \partial_{z^j} \vec{X} | \lesssim \epsilon$, which we plug into the tranport equation for $\mathring{\Omega}^2$ to yield
    \[
        | \tau \partial_{\tau} \mathring{\Omega}^2 - 2 \kappa_p \mathring{\Omega}^2 | \lesssim \epsilon \tau.
    \]
    Since $\kappa_p \geq 1$, thie gives also that $|\mathring{\Omega}^2| + |\tau \partial_{\tau} \mathring{\Omega}^2| \lesssim \epsilon$, and the remaining components of $\vec{X}_0$ are similarly bounded, upon taking another $\tau \partial_{\tau}$ derivative or $\partial_{z^j}$ derivative of the equations in \eqref{eq:transport_matrix}. Thus \eqref{eq:convergent_bootstrap} is also improved, yielding existence in the entire domain $\Mhat_{\tau_1}$.

    Next, we use the assumption that $A\geq2$.
    Taking further $\partial_z$ derivatives of the transport equation \eqref{eq:transport_matrix}, and dividing by $\tau$, it is straightforward to find for $\abs{\alpha}\leq A$ that 
    \[
        \partial_{\tau} \partial_z^{\alpha} U^i = O(\tau^{-\tilde{q}\sigma}), \quad \partial_{\tau} (\tau^{-2 \kappa_p} \mathring{\Omega}^2) = O(\tau^{-\tau^{-\tilde{q}\sigma}-2 \kappa_p}),
    \]
    where $\tilde{q}\sigma<1$ is as in the proof of \cref{lem:improvement}. 
    Indeed, we can induct on $U^i,\mathring{\Omega}^2\in C^{\abs{\alpha}-1}$ and apply \cref{eq:Gagliardo} together with the bound $L^{\infty}\left((0,\tau_1);\tau^{-\tilde{q}}H^{K}\right)$ to obtain the boundedness on the right hand side.
    From this, we have convergence of $U^i(\tau, z) \to \mathfrak{u}^i(z)$ and $\mathring{\Omega}^2(\tau, z) \to 0$ as $\tau \to 0$ at the rate $O(\tau)$ in the $C^A$ topology, from which the convergence \eqref{eq:smooth_convergence} follows upon using \eqref{eq:vwomega}.
\end{proof}

\subsection{Recovery of scattering data} \label{for:sub:scattering}
In this section, we take the solutions constructed in \cref{prop:existence} and apply the transport equations \cref{eq:transport_U,eq:transport_omega} to recover a scattering solution as defined in \cref{not:def:scattering_VW}.

\begin{lemma}\label{back:lem:recovery}
	Fix $\tilde{q}\notin\N$.
	Let $U^i,\Omega \in L^\infty([0,\tau_1];\tau^{-\tilde{q}}H^K(\Shat_\tau))$ be a solution of \cref{eq:transport_U,eq:transport_omega} for $K>\tilde{q}+2\kappa_p+\frac{n}{2}+10$ with the bounds \cref{eq:bootstrap_all}.
	Then $(U^i,\Omega)$ is a scattering solution of regularity $K-\floor{\tilde{q}+1}$ with scattering data $\mathfrak{f}\in H^{K-\floor{\tilde{q}+1}}(\Shat_0)$ and $\mathfrak{w}\in H^{K-\floor{\tilde{q}+1}-\floor{2\kappa}}(\hat{\Sigma}_0)$.
\end{lemma}
\begin{proof}
    a)
	We only show the case $2\kappa_p\notin\N$; the only difference when $2\kappa_p\in\N$ is that the final integration picks up an extra logarithmic term and we leave this to the reader. 
	It suffices to show that, for $\Phi\in\{U^i,\Omega\}$, it holds that
	\begin{equation}\label{eq:rec:UOmega}
		\begin{multlined}
			\Phi\in \sum_{j=c_\Phi}^{\floor{2\kappa_p}} \tau^j H^{K-\floor{\tilde{q}+1}-j}(\Shat_0)+c_\Phi \tau^{2\kappa_p}H^{K-\floor{\tilde{q}+1}-\floor{2\kappa_p}}(\hat{\Sigma}_0)\\
			+L^\infty \big([0,\tau_1]; \tau^{\floor{2\kappa_p}+\mathfrak{c}}H^{K-\floor{\tilde{q}+1}-\floor{2\kappa_p}}(\Shat_0)\big)
		\end{multlined}
	\end{equation}
	where $c_{\mathring{\Omega}}=1,c_{U^i}=0$, and $\mathfrak{c}=(\floor{\tilde{q}+1}-\tilde{q})/2$.
	Let us write \cref{eq:transport_omega} as $\partial_\tau \tau^{-2\kappa_p}\mathring{\Omega}=\tau^{-2\kappa_p}\partial_{z^i}V^i$.
	
	First, we claim that, for $ j\leq \floor{\tilde{q}}$, it holds that 
	\begin{equation}\label{eq:rec:step1}
		\Phi\in L^{\infty}\big([0,\tau_1];\tau^{-\tilde{q}+j} H^{K-j}(\Shat_0)\big).
	\end{equation}
	Indeed, this follows easily by induction.
	For $U^i$, we simply note that using \cref{eq:vw_bootstrap,eq:rec:step1} and interpolation, we have $\norm{W^{-2}\partial_{z^i}W}_{H^{K-j-1}}\lesssim \tau^{-\tilde{q}+j}$.
	Provided that $j+1\leq \floor{\tilde{q}}$ it follows by a simple integration that $\cref{eq:rec:step1}$ holds for $U^i$ with $j+1$.
	For $\mathring{\Omega}^2$, we instead obtain $\partial_\tau \tau^{-2\kappa_p}\mathring{\Omega}^2\in L^{\infty}\big([0,\tau_1];\tau^{-\tilde{q}+j-2\kappa_p} H^{K-j}(\Shat_0)\big)$.
	Using the induction hypothesis once more and integration when $j+1\leq \floor{\tilde{q}}$ implies \cref{eq:rec:step1}.
	
	Next, we claim that, for $l < \lfloor 2 \kappa_p \rfloor$, it holds that
    \begin{equation} \label{eq:rec:Omega2}
        \Phi\in Z_\Phi^{K,l}:=\sum_{j=c_\Phi}^l \tau^j H^{K-\floor{\tilde{q}+1}-j}(\Shat_0)+L^{\infty}\big([0,\tau_1];\tau^{\min(\mathfrak{c}+l,2\kappa_p)} H^{K-\floor{\tilde{q}+1+l}}(\hat{\Sigma}_0)\big).
    \end{equation}
    We shall prove \eqref{eq:rec:Omega2} by induction. 
	For $l=-1$, this is already proved.
	The induction step for $U^i$ follows for with the extra integral formula 
	\[ \Phi\in Z_\Phi^{K,l}, \partial_\tau U^i\in Z_U^{K-1,l} \implies U\in Z_U^{K,l+1}.
	\]
	For $\mathring{\Omega}^2$, we use that the extra terms obtained by the integration in $\tau$ yields a term that can be included in the $L^\infty$ term in the expansion, so that
	\[\partial_\tau \tau^{-2\kappa_p}\mathring{\Omega}^2\in \tau^{-2\kappa_p}Z_U^{K-1,l}\implies \tau^{-2\kappa_p}\mathring{\Omega}^2\in \tau^{-2\kappa_p}Z_{\mathring{\Omega}}^{K,l}.\
	\]
    Finally performing one more integration for $\mathring{\Omega}^2$ yields \cref{eq:rec:UOmega}.
	
	To obtain full $\Hb$ regularity for the error, note that we can use equations \cref{eq:transport_U,eq:transport_omega} to exchange regularity in $\partial_{z^i}$ for regularity with respect to $\tau\partial_\tau$. 
\end{proof}

\begin{cor}
	Let $\phi_0,\phi_1,s$ be as in \cref{in:thm:main2}.
    Then $\phi$ is a singular solution of regularity $s-202\kappa_p-1$.
\end{cor}

\begin{proof}
    The existence of a scattering solution $(W,V^i)$ as in \cref{not:def:scattering_VW} follows from \cref{lem:Cauchy,prop:existence,back:lem:recovery}.
    We may use \cref{not:lem:comparison} to recover a singular solution for $\phi$.
\end{proof}

To conclude this section, we also explain how to derive stability estimates for the differences between two different solutions $(U^i_{(j)}, \mathring{\Omega}^2_{(j)})$ for $j = 1, 2$, which, after taking into account \eqref{for:eq:cauchy_stability}, will be used to find \eqref{in:eq:sing_Cauchy}, thus concluding the proof of \cref{in:thm:main2_precise}.

\begin{lemma}[Difference estimates]\label{lem:difference}
    For any pair of $(U^i_{(j)},\Omega_{(j)})$ for $j\in\{1,2\}$ as in \cref{back:lem:recovery}, denoting $\mathfrak{f}_{(j)},\mathfrak{w}_{(j)}$ the corresponding scattering solution, it holds that
    \begin{equation}\label{back:eq:comparing_scat}
        \norm{\mathfrak{f}_{(1)}-\mathfrak{f}_{(2)}}_{H^{K-1-\floor{\tilde{q}+1}}(\Shat_0)}+\norm{\mathfrak{w}_{(1)}-\mathfrak{w}_{(2)}}_{H^{K-1-\floor{\tilde{q}+1}-\floor{2\kappa_p}}(\hat{\Sigma}_0)}\lesssim \norm{(U^i_{(1)}-U^i_{(2)},\mathring{\Omega}^2_{(1)}-\mathring{\Omega}^2_{(2)})}_{H^{K-1}(\hat{\Sigma}_1)}.
    \end{equation}
    This in particular yields \cref{in:eq:sing_Cauchy}.
\end{lemma}

\begin{proof}[(Sketch proof)]

	The estimate \eqref{back:eq:comparing_scat} will be derived by taking differences between the equations for $(U^i_{(m)}, \mathring{\Omega}^2_{(m)})$. 
	We define $(U_{\Delta},\Omega_\Delta)=(U_{(1)}-U_{(2)},\Omega_{(1)}-\Omega_{(2)})$, and also $(V^i_{\Delta}, W_{\Delta})$ similarly.
	Denote by $J_{(m)} = (W_{(m)} , V_{(m)} ^i)^{\intercal}$ $J=(J_{(1)},J_{(2)} )^{\intercal}$ and let $J_{\Delta}$ be the difference.
	We prove a (degenerate) energy estimate for the differences $(V^i_{\Delta}, W_{\Delta})$, which we simply sketch below, as it mostly follows the same two steps as the stability proof of \cref{prop:l2,lem:improvement}.
    Then, the same steps as in \cref{lem:difference} yields the result.
	
	Recall the form of the wave operator from \cref{eq:wave_op}, noting that we will have two different wave operators $\Box_1, \Box_2$ from the two different metric coefficients $(V^i_{(1)}, W_{(1)})$ and $(V^i_{(2)}, W_{(2)})$. By taking differences of the equation \eqref{eq:nlw_W} for $m= 1, 2$, we derive an equation as follows:
	\begin{equation} \label{eq:nlw_W_diff}
		\tau^2 \Box_1 W_{\Delta} = - \tau^2 (\Box_1 - \Box_2) W_2 + \kappa_p \left(\mathring{\Omega}^2_{\Delta} W_{(1)} -  (1 - \Omega^2_{(2)}) W_{\Delta} -  W_{\Delta} \tau \partial_{\tau} \Omega^2_{(1)} - W_{(2)} \tau \partial_{\tau} \mathring{\Omega}^2_{\Delta}\right).
	\end{equation}
    The difference between the wave operator on the right hand side can be written using the multilinear forms of \cref{not:multiforms}, where smooth dependence on $J_{(m)}$ for $m\in\{1,2\}$ are hidden in the coefficients of the forms $\mathbf{B}_1, \mathbf{B}_2$:
	\begin{equation} \label{eq:quasilinear_difference}
		- \tau^2 (\Box_1 - \Box_2) \Phi = \mathbf{B}_1(J_{\Delta},\tau^2 \partial^2 \Phi) + \mathbf{B}_2(\tau \partial J_{\Delta},\tau \partial \Phi).
	\end{equation}
    Note that, for $\Phi$ a component of $J_{(m)}$ for some $m \in \{1, 2\}$, coefficients $\tau^2 \partial^2 \Phi$ and $\tau \partial \Phi$ are elements of $\vec{J}_0$ as in \cref{def:j_0}, at least upon using the equations \eqref{eq:nlw_Jacobi} to express $\tau^2 \partial_{\tau}^2 \Phi$ in terms of other terms in $\vec{J}_0$.

    Similarly, we obtain higher order estimates by commuting with $\partial_z^{\alpha}$ as in \cref{cor:wave_error}. For $|\alpha| = K$, one finds that, for suitably defined $\mathbf{B}_{\bullet}$ and $\mathbf{Q}_{\bullet}$, that
	\begin{nalign}\label{eq:difference_commuted}
        \tau^2\Box_1\partial^\alpha_z W_\Delta
        &=\partial^\alpha\Box_1 W_\Delta+[\Box_1,\partial^\alpha_z]W_\Delta\\[0.5em]
        &\begin{aligned}=\mathbf{L}(\partial_z^{\alpha} \vec{J}_\Delta) + \mathbf{B}_{\bullet}(\partial_z^{\alpha} \vec{J}_\Delta, \vec{J}_0) + \mathbf{B}_1(J_{\Delta}, \tau^2 \partial^2 \partial_z^{\alpha} W_2) \\[0.5em] \hspace{3em} + \sum_{\substack{L \geq 1, |\beta| \geq 0, |\alpha_{\ell}| \geq 1 \\ |\beta| + \sum{|\alpha_{\ell}| = K }}} \mathbf{Q}_{\bullet} (\partial_z^{\beta} \vec{J}_\Delta, \partial_z^{\alpha_1} \vec{J}, \ldots, \partial_z^{\alpha_L} \vec{J}),\end{aligned}
	\end{nalign}
    where we recall that $\vec{J}$ represents the concatenation of $(\kappa_p J, \tau \partial_{\tau} J, \tau \partial_{z^k} J)$ over all components of $J$, and $\vec{J}_{\Delta}$ is similar. 
    Note we have singled out the term $\mathbf{B}_1(J_{\Delta}, \tau^2 \partial^2 \partial_z^{\alpha} W_2)$, which is top order in derivatives and reflects the quasilinearity of the system \eqref{eq:nlw_Jacobi}. 

    Using \eqref{eq:nlw_W_diff} and the analogous equation for $V^i_{\Delta}$, much like in \cref{for:sub:energy}, we can  derive an energy estimate for all $|\alpha| \leq K$, $\tilde{q} \geq 200 \kappa_p$ and $0 < \tau_0 < \tau_1$ as below.
	\begin{multline}\label{eq:difference_H1_error}
        \| \Vtop \partial_z^\alpha J_{\Delta} \|^2_{H^{0, -\tilde{q}}(\hat{\Sigma}_{\tau_0})} + \tilde{q}\int_{\tau_0}^{\tau_1} \| \Vtop \partial_z^\alpha J_{\Delta} \|^2_{H^{0, - \tilde{q}}(\hat{\Sigma}_{\tilde{\tau}})} \frac{d \tilde{\tau}}{\tilde{\tau}}  \\
        \lesssim \| \Vtop J_{\Delta} \|^2_{H^{K, -\tilde{q}}(\hat{\Sigma}_{\tau_1})} + \int_{\tau_0}^{\tau_1} \norm{\tau^2\Box_1\partial^\alpha_z W_\Delta}_{{H^{0, - \tilde{q}}}(\hat{\Sigma}_{\tilde{\tau}})}\frac{\dd\tilde{\tau}}{\tilde{\tau}}.
	\end{multline}
    
    Since, by \cref{prop:l2} and \cref{back:lem:recovery}, we already have pointwise a priori control on (up to two derivatives of) $J_{(m)}$, we get that, for $K=0$, the following energy estimate holds.
	\begin{equation}\label{eq:difference_H1}
	    \| \Vtop J_{\Delta} \|^2_{H^{0, -\tilde{q}}(\hat{\Sigma}_{\tau_0})} +\tilde{q} \int_{\tau_0}^{\tau_1} \| \Vtop J_{\Delta} \|^2_{H^{0, - \tilde{q}}(\hat{\Sigma}_{\tilde{\tau}})} \frac{d \tilde{\tau}}{\tilde{\tau}} \lesssim \| \Vtop J_{\Delta} \|^2_{H^{0, -\tilde{q}}(\hat{\Sigma}_{\tau_1})}.
    \end{equation}
    We hope to show that a similar estimate holds at all orders. By repeating the procedure in \cref{for:sub:energy} to estimate the right hand side in \eqref{eq:difference_commuted}, one can first derive the following, for $\tilde{K} \geq 0$ and $\tilde{q} \geq 200 \kappa_p$.
	\begin{multline}\label{eq:difference_toporder}
        \| \Vtop J_{\Delta} \|^2_{H^{\tilde{K}, -\tilde{q}}(\hat{\Sigma}_{\tau_0})} +\tilde{q} \int_{\tau_0}^{\tau_1} \| \Vtop J_{\Delta} \|^2_{H^{\tilde{K}, - \tilde{q}}(\hat{\Sigma}_{\tilde{\tau}})} \frac{d \tilde{\tau}}{\tilde{\tau}} \lesssim \| \Vtop J_{\Delta} \|^2_{H^{\tilde{K}, -\tilde{q}}(\hat{\Sigma}_{\tau_1})} \\[0.3em]
        + \| \vec{J}_{\Delta} \|_{L^{\infty}(\Mhat_{\tau_1})}^2 \left( \int_{\tau_0}^{\tau_1} \| \Vtop J \|_{H^{\tilde{K}, - \tilde{q}}(\Shat_{\tilde{\tau}})}^2 + \| \tau^2 \partial^2 J \|_{H^{\tilde{K}, - \tilde{q}}(\Shat_{\tilde{\tau}})}^2 \frac{ d \tilde{\tau}}{\tilde{\tau}}  \right).
    \end{multline}
    We comment that the terms appearing in the second line of \eqref{eq:difference_toporder} appear after applying the product estimate \cref{lem:product} to the multilinear term $\mathbf{Q}_{\bullet}$ in \eqref{eq:difference_commuted} and also from the term $\mathbf{B}_1$.

    We now estimate the integral expression arising in the second line, by using \eqref{eq:existence_bounds}, or more precisely the combination of this and \eqref{eq:l2}, to get that, for $\tilde{K} \leq K$ as in \cref{prop:existence},
    \[
        \int_{\tau_0}^{\tau_1} \| \Vtop J \|_{H^{K, - \tilde{q}}(\Shat_{\tilde{\tau}})}^2 \frac{d \tilde{\tau}}{\tilde{\tau}} \lesssim \epsilon^2.
    \]
    In fact, provided that one chooses $\tilde{K} \leq K-1$, one can also estimate $\| \tau^2 \partial^2 J \|_{H^{\tilde{K}, - \tilde{q}}(\Shat_{\tilde{\tau}})}$ by the derivative--losing norm $\tau \| \Vtop J \|_{H^{K, - \tilde{q}}(\Shat_{\tilde{\tau}})}$, where one may have to once again insert the equations \eqref{eq:nlw_Jacobi} in the case $\partial^2$ is represented by $\partial_{\tau}^2$. Omitting the details, one therefore finds that
	\begin{equation}\label{eq:difference_toporder_2}
        \| \Vtop J_{\Delta} \|^2_{H^{\tilde{K}, -\tilde{q}}(\hat{\Sigma}_{\tau_0})} +\tilde{q} \int_{\tau_0}^{\tau_1} \| \Vtop J_{\Delta} \|^2_{H^{\tilde{K}, - \tilde{q}}(\hat{\Sigma}_{\tilde{\tau}})} \frac{d \tilde{\tau}}{\tilde{\tau}} \lesssim \| \Vtop J_{\Delta} \|^2_{H^{\tilde{K}, -\tilde{q}}(\hat{\Sigma}_{\tau_1})} + \epsilon^2 \| \vec{J}_{\Delta} \|_{L^{\infty}(\Mhat_{\tau_1})}^2.
    \end{equation}

    To estimate the $\| \vec{J}_{\Delta} \|_{L^{\infty}(\Mhat_{\tau_1})}$ term appearing on the right, we consider the transport equations of the differences 
	\begin{align} 
		\tau \partial_{\tau} U^i_\Delta &= W^{-2}_{(1)} \tau \partial_{z^i} W_{\Delta}-W_\Delta\frac{W_{(1)}+W_{(1)}}{W_{(1)}^{2}W_{(2)}^{2}}\tau \partial_{z^i}W_{(1)},\\
		\tau \partial_{\tau} \mathring{\Omega}^2_{\Delta} &= 2 \kappa_p \mathring{\Omega}^2_{\Delta}  + 2 \tau \partial_{z^i} V^i_\Delta.
	\end{align}
    That is, if $\vec{X}_{\Delta} = (U_{\Delta}^i, \mathring{\Omega}^2_{\Delta})$, then we have that, for some matrices $\tilde{\mathbf{M}}^i$ depending only on $J$,
    \begin{equation} \label{eq:transport_matrix_differences}
        \tau \partial_{\tau} \vec{X}_{\Delta} = 
        \begin{bmatrix} 0 & 0 \\ 0 & 2 \kappa_p \end{bmatrix}\vec{X}_{\Delta}
        + \tilde{\mathbf{M}}^i \tau \partial_{z^i} \vec{X}_{\Delta}.
    \end{equation}
    Applying the method of \cref{lem:improvement}, we thereby find that, for $\tilde{K} \geq 200 \kappa_p + \frac{n}{2} + 1$ and $\tilde{q} = 200 \kappa_p$, one has the pointwise bound
    \begin{equation} \label{eq:difference_pointwise}
        \| \vec{J}_{\Delta} \|_{L^{\infty}(\Mhat_{\tau_1})}^2 \lesssim \sup_{\tilde{\tau} \in (\tau_0, \tau_1)} \| \Vtop J_{\Delta} \|_{H^{\tilde{K}, - \tilde{q}}(\Shat_{\tilde{\tau}})}^2.
    \end{equation}

    Choosing $\epsilon$ to be sufficiently small in \eqref{eq:difference_toporder_2}, we can insert \eqref{eq:difference_pointwise} to get that, for $\tilde{q} = 200 \kappa_p$, 
    \[
        \sup_{\tilde{\tau} \in (\tau_0, \tau_1)} \| \Vtop J_{\Delta} \|_{H^{K-1, - \tilde{q}}(\Shat_{\tilde{\tau}})}^2  \lesssim \| \Vtop J_{\Delta} \|_{H^{K-1, - \tilde{q}}(\Shat_{\tau_1})}^2.
    \]
    Note that the right hand side is equivalent to the right hand side of \eqref{back:eq:comparing_scat} for $\tau_1 > 0$. We then complete the proof by again using the transport equations for the differences and proceeding as in \cref{back:lem:recovery}.
\end{proof}

	\section{Localised stability and regularity}\label{sec:local} 
	
In this section, we apply \cref{in:thm:main1_precise,in:thm:main2_precise} to deduce \cref{in:cor:general,in:cor:smoothness}. 
The general strategy is that though the general ODE blowup solution in \cref{in:cor:general} is not necessarily close to that of the model (spatially homogeneous) ODE blow-up solution, one can use the scaling, translation and Lorentz boost symmetries of the equation \eqref{in:eq:main} to zoom into a neighborhood of the singularity, so that $\phi$ will actually closely resemble the model ODE blow-up in suitable norms, in particular such that one may apply \cref{in:thm:main2_precise}.

The arguments presented will be largely soft, and not be as computationally intensive as \cref{sec:backward,sec:forward}. 
We begin  in \cref{gen:lem:Cauchy}with generalising \cref{in:thm:main2_precise} to place the perturbation on arbitrary spacelike hypersurfaces, which is a Cauchy stability argument.
Then, we use the symmetries of \cref{in:eq:main} to localise the stability close to a general spacelike hypersurface and prove \cref{in:cor:general}.
Finally, we use the same localisation argument to propagate the regularity all the way to the singularity and prove \cref{in:cor:general,in:cor:smoothness}.

\begin{lemma}\label{gen:lem:Cauchy}
	Let $\underline{\phi}=t^{-\alpha_p}c_p$ be the ODE blow-up and $s,s_0$ as in \cref{in:thm:main2_precise}.
    Fix a $H^{s+n+2}$ spacelike hypersurface $\Sigma\subset\{t\in(2,4)\}\cap\{\abs{x}<15\}$ with unit normal $N$ satisfying $N(\dd t)\in[1/2,2]$ and $\| \Sigma \|_{H^{s+n+2}} \leq C$. Then there exists $\epsilon_{C} > 0$, such that, for $\phi_0\in H^{s+2}(\Sigma),\phi_1\in H^{s+1}(\Sigma)$ satisfying
	\begin{equation}
        \norm{\phi_0-\underline{\phi}|_{\Sigma}}_{H^{s_0+1}}+\norm{\phi_1-\partial_t\underline{\phi}|_{\Sigma}}_{H^{s_0}(\Sigma)}\leq \epsilon_2\leq \epsilon_{C},
	\end{equation}
    the correspoding solution $\phi$ to \eqref{in:eq:main} to the past of $\Sigma$ with initial data $\phi|_{\Sigma}=\phi_0$ and $\partial_t\phi|_{\Sigma}=\phi_1$ has the following properties: 
	there exist $f,\psi\in H^{s-202\kappa_p}$ such that $\phi$ can be extended to the domain $\M=\{\abs{x}<3(1+t),f(x)<t\}\cap I^-(\Sigma)$, and exhibits ODE type blow-up on the spacelike hypersurface $\{t=f(x)\}\cap\M$ with auxiliary scattering data $\psi$.
\end{lemma}
\begin{proof}
	This is immediate using Cauchy stability to propagate $\phi$ to (possibly a subset of) the appropriate $\{t=1\}$ hypersurface, from which one can apply \cref{in:thm:main2_precise}.
\end{proof}

Similarly, we reduce \cref{in:cor:general} to stability result localised to a small neighbourhood of the blow-up hypersurface.
\begin{lemma}\label{gen:lemma:local}
	Let $s,f,\psi,\phi,\tstar,\Mt$ be as in \cref{in:cor:general}.
    Then there exist $\delta,\epsilon>0$ such that, for all $\phi_0,\phi_1$, satisfying on $\Sigma_\delta=\{\tstar=\delta\}$
	\begin{equation}\label{gen:eq:gen_pert}
		\norm{(\phi_0-\phi)|_{\Sigma_\delta}}_{H^{s-30\kappa_p+1}(\Sigma_\delta)}+\norm{(\phi_1-\partial_t\phi)|_{\Sigma_\delta}}_{H^{s-30\kappa_p}(\Sigma_\delta)}\leq\epsilon_2<\epsilon,
	\end{equation}
	there exists $\bar{f},\bar{\psi}\in H^{s-232\kappa_p}(\BB_2)$ and a solution, $\bar{\phi}$, of \cref{in:eq:main} with initial data $(\phi_0,\phi_1)$ in $\bar\M=\{\abs{x}\leq 5(t+1/3)\}\cap \{\tstar\leq\delta\}\cap\{t>\bar{f}\}$ with $\bar{\phi}$ exhibiting ODE blow up towards $\{t=\bar{f}\}$ with auxiliary scattering data $\bar\psi$ which moreover satisfy
    \begin{equation}
        \norm{f-\bar{f},\psi-\bar\psi}_{H^{s-232\kappa_p}}\lesssim\epsilon_2.
    \end{equation}
\end{lemma}

Before proving \cref{gen:lemma:local}, we use it to conclude our result:
\begin{proof}[Proof of \cref{in:cor:general} assuming \cref{gen:lemma:local}]
    Let $\bar\M_{\mathrm{large}},\bar\M_{\mathrm{small}}$ be the existence region in \cref{in:cor:general} and \cref{gen:lemma:local} respectively.
    See \cref{fig:cor1.6} for an illustration.
    Note that $\bar\M_{\mathrm{large}}\setminus \bar\M_{\mathrm{small}}$ does not contain the singular surface for $\phi$, and thus $\phi$ and suitable derivatives are bounded in this region.
    Therefore, from Cauchy stability, it follows that \cref{in:eq:gen_pert} implies \cref{gen:eq:gen_pert} for $\epsilon_1$ sufficiently small.
    Thus, we can apply \cref{gen:lemma:local} to obtain the stability in the restricted region $\bar\M_{\mathrm{large}}\cap\bar\M_{\mathrm{small}}$, which yields the result.
\end{proof}

\begin{figure}
    \centering
    
\begin{tikzpicture}

\begin{axis}[
    axis lines=none,
    domain=-3:3,
    samples=100,
    width=15cm,
    height=9cm,
    xmin=-2.5, xmax=2.5,
    ymin=-0.5, ymax=1,
    thick
]

    \fill [lightgray, opacity=0.4]
        plot[domain = -2:-1.2] ({\x}, {0.5*(\x+2)})
        -- plot[domain = -1.2:1.3] ({\x}, {(1/200)*(\x-6)*(\x-2)*(\x+2) + 0.3})
        -- plot[domain = 1.3:2] ({\x}, {-0.5*(\x-2)})
        -- plot[domain = 2:-2] ({\x}, {(1/200)*(\x-6)*(\x-2)*(\x+2)})
        -- cycle;

    \fill [pattern=north east lines, pattern color=gray]
        plot[domain = -1.5:1.4] ({\x}, {(1/200)*(\x-6)*(\x-2)*(\x+2) + 0.15})
        -- (0.9, 0.1)
        .. controls (0.3, 0.15) and (-0.2, 0.08) .. (-0.8, 0.08)
        -- cycle;
        ;

    \addplot[domain=-2:-1.2] {0.5*(x+2)} node[above left, midway] {\small {slope} $\frac{1}{2}$};
    \addplot[domain=1.3:2] {-0.5*(x-2)} node[above right, midway] {\small $\mathcal{C}_1$};

    \addplot[domain=-1.5:0] {-0.2*x -0.08} node[below left, midway] {\small slope $\frac{1}{5}$};
    \addplot[domain=0:1.4] {0.2*x - 0.08} node[below right, midway] {\small $\mathcal{C}_2$};

    \addplot[black, dotted, very thick] {(1/200)*(x-6)*(x-2)*(x+2)};
    \addplot[black, dashed, domain=-1.6:1.6] {(1/200)*(x-6)*(x-2)*(x+2) + 0.15} node[above, midway] {\small $\Sigma_{\mathbf{t}_1/2}$};
    \addplot[black, dashed, domain=-1.1:1] {(1/200)*(x-6)*(x-2)*(x+2) + 0.05} node[above, midway] {\small $\Sigma_{\delta}$};
    \addplot[black, very thick, domain=-1.2:1.3] {(1/200)*(x-6)*(x-2)*(x+2) + 0.3} node[above, midway] {\small $\Sigma_{\mathbf{t}_1}$, slope $<\frac{1}{10}$};

    \draw[red, thick] (0.75, 0.05) -- (1.05, 0.05) -- (1.05, 0.17) -- (0.75, 0.17) -- (0.75, 0.05);

    \node at (axis cs:1.5,-0.05) {\small $\Sigma_0$};
\end{axis}
\end{tikzpicture}

\captionsetup{justification = centering}
\caption{The solution constructed in \cref{in:thm:main1_precise} exists in the region bounded by $\mathcal{C}_1$, $\Sigma_{\tstar_1}$ and $\Sigma_0$ (indicated in gray), while the perturbed solution described in \cref{in:cor:general} will be constructed between $\Sigma_{\tstar_1/2}$, $\mathcal{C}_2$ and a new ODE blow-up hypersurface (indicated by the hatching). We have also indicated the slopes of the cones in the diagram. The red boxed region will be detailed in \cref{fig:zoomed} below.}
\label{fig:cor1.6}
\end{figure}
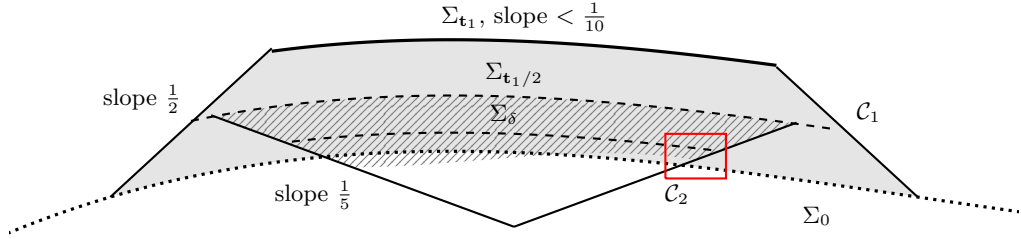


\begin{figure}
    \centering
    \makebox[\textwidth][c]{
\begin{tikzpicture}[scale=0.7]

\begin{axis}[
    axis lines=none,
    domain=-6:6,
    samples=200,
    width=14cm,
    height=10cm,
    xmin=-7, xmax=7,
    ymin=-3, ymax=7,
    thick]

    \fill [pattern=north east lines, pattern color=gray]
        (-6, 4.6) .. controls (0, 3.2) .. (4.6, 3.1)
        -- (0, 0.2)
        .. controls (-3, 0.3) .. (-6, 1.2)
        -- cycle;

    \draw[black, dotted, very thick] (-6, 1.6) .. controls (0, 0.2) .. (6, 0.0)
        node[below, pos=0.7] {$\Sigma_0$};
    \draw[black, dashed] (-6, 4.6) .. controls (0, 3.2) .. (6, 3.0)
        node[above, midway] {$\Sigma_{\delta}$};

    \draw[black, thick] (6, 4.0) -- (-2, -1.0)
        node[below right=-0.1cm, pos=0.5] {$\mathcal{C}_2$, slope $1/5$};

    \node [circle, draw, fill=red, inner sep=0.01mm] (p) at (-1.6, 0.6) { . };
    \node [below = 1mm of p] {\color{red} $p$};
    \node at (0, 6) {Standard $(t, x)$ coordinates};
    \node at (0, -2) {};
\end{axis}
\end{tikzpicture}%
\hspace{1em}%
\begin{tikzpicture}[scale=0.7]

\begin{axis}[
    axis lines=none,
    domain=-6:6,
    samples=200,
    width=14cm,
    height=10cm,
    xmin=-7, xmax=7,
    ymin=-3, ymax=7,
    thick]

    \fill [pattern=north east lines, pattern color=gray]
        (-6, 3.1) .. controls (0, 2.8) .. (5, 4.15)
        -- (0.05, -0.2)
        .. controls (-3, -0.5) .. (-6, -0.1)
        -- cycle;
    
    \fill [pattern=north west lines, pattern color=gray]
        (-2.6, 0.83) -- (-0.6, 0.83)
        -- (0.65, 3.10)
        .. controls (-1.6, 2.9) .. (-3.8, 3.0) 
        -- cycle;

    \draw[black, dotted, very thick] (-6, 0.1) .. controls (0, -0.2) .. (6, 1.4)
        node[below, pos=0.7] {$\Sigma_0$};
    \draw[black, dashed] (-6, 3.1) .. controls (0, 2.8) .. (6, 4.4)
        node[above, midway] {$\Sigma_{\delta}$};

    \draw[black, thick] (6, 5.0) -- (-2, -2.0)
        node[below right=-0.1cm, pos=0.4] {$\mathcal{C}_2$, slope $\in (1/10, 3/10)$};

    \draw[black, very thick] (-1.6, -1) -- (1.4, 4.5);
    \draw[black, very thick] (-1.6, -1) -- (-4.6, 4.5)
        node[below left, pos=0.45] {slope $1/3$};
    \draw[black, very thick] (-2.6, 0.83) -- (-0.6, 0.83);

    \node [circle, draw, fill=red, inner sep=0.01mm] (p) at (-1.6, -0.05) { . };
    \node [below = 1mm of p] {\color{red} $p$};
    \node at (0, 6) {Boosted coordinates};
    \node at (-1.6, 2) {$\mathcal{M}_{p}^{1/4}$};
\end{axis}
\end{tikzpicture}
}
\captionsetup{justification = centering}
\caption{On the left, we provide a zoom into the red boxed region of \cref{fig:cor1.6} where \cref{gen:lemma:local} is proved. On the right, we illustrate the same region in Lorentz boosted coordinates, so that $\Sigma_0$ has vanishing gradient at $p \in \Sigma_0$. We apply Cauchy stability in the cross hatched region $\mathcal{M}_p^{1/4}$, while stability in the cone below the region follows from \cref{in:thm:main2_precise}, as described below.}
\label{fig:zoomed}
\end{figure}
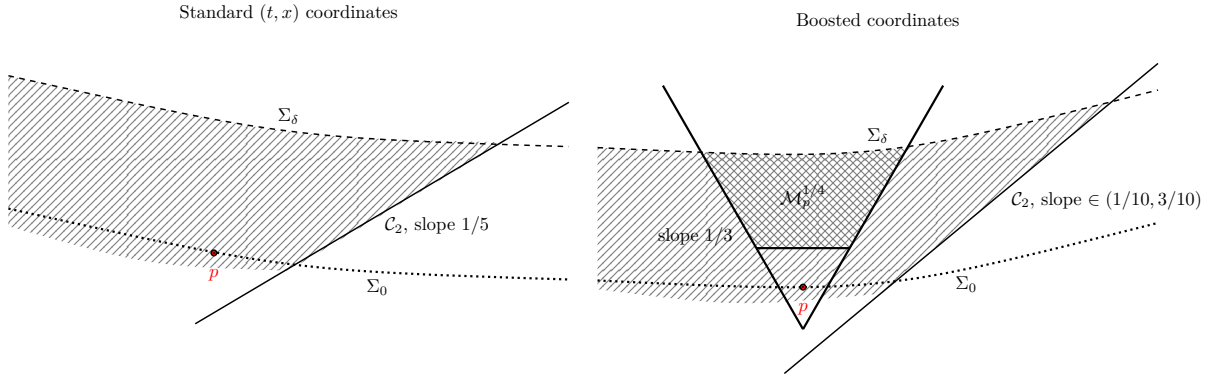
%
%
For the proof of \cref{gen:lemma:local}, we will apply the following scaling symmetry for solutions of the equation \eqref{in:eq:main}.

\begin{definition} \label{def:rescale}
    Let $\phi: \mathcal{D} \to \R$ be a solution of the equation \eqref{in:eq:main}. Then, for any $\delta > 0$, the \emph{$\delta$-rescaling of $\phi$} is defined as the function $\phi_{\delta}: \delta^{-1} \mathcal{D} \to \R$ defined as:
    \begin{equation}\label{gen:eq:scaling}
        \phi_\delta(\bar{t}, \bar{x})=\delta^{\alpha_p}\phi(\bar{t}\delta, \bar{x}\delta).
    \end{equation}
    It is straightforward to check that $\phi_{\delta}$ is also a solution of \cref{in:eq:main}.
\end{definition}

\begin{proof}[Proof of \cref{gen:lemma:local}]
	Let us write $\underline{\M}=\{\abs{x}\leq 5(t+1/3)\}\cap \{\tstar<\delta\}$ and $\underline{\Sigma}_0=\underline{\M}\cap\Sigma_0$ be the portion of the singularity for $\phi$ restricted to this region.
	
    \vspace{0.5em}
	\emph{Step 1 (Geometry):}
	Let $p\in\underline{\Sigma}_0$ be a point on the singularity.
    Using a Lorentz boost, of size at most $|\mathbf{v}| \leq \sup |\partial f| <1/10$, and spacetime translation symmetries, one may assume without loss of generality that $p = (0, 0)$ with $f(0)=\partial f(0)=0$.
	This transformation changes the $\Hb$ norms of $\phi$ given in \cref{in:eq:expansion} at most by a constant factor depending on $\abs{\mathbf{v}}$ and $s$.
    In the new coordinates it still holds that $\abs{\partial f }<1/5$.
	
	We consider the region $\M^{a}_p=\{\abs{x}<3(\delta/10+t)\}\cap \{\tstar<\delta\}\cap\{t>a\delta\}$, for values $-1 \leq a \leq 1$.
	From \cref{fig:zoomed}, it is clear that $\M^{-1}_p\subset \underline{\M}$, thus we have an explicit control of $\phi$ in $\M^{-1}_p$ provided by the asymptotic expansion \cref{in:eq:expansion} within $\underline{\M}$.
	Furthermore, it also holds that $t\sim \delta$ in $\M^{1/4}_p$.
	
    \vspace{0.5em}
	\emph{Step 2 (Scaling):}
    Let us now consider the $\delta$-rescaling of $\phi$, and, as in \cref{def:rescale}, use rescaled coordinates $\bar{t}$ and $\bar{x}$ to describe the domain of $\phi_{\delta}$. From Step 1, it follows that in the rescaled $\M^{1/4}_p$, it holds that $\abs{\bar{x}}\lesssim 1$ and $\bar{t}\sim 1$.
    Using \cref{gen:eq:scaling} and \eqref{in:eq:expansion}, we obtain that
	\begin{equation}\label{gen:eq:rescaled}
		\phi_\delta(\bar{t},\bar{x})=\underbrace{\delta^{\alpha_p}c_p(\bar{t}\delta-f(\bar{x}\delta))^{-\alpha_p}(1-\abs{\partial f(\delta \bar{x})}^2)^{\alpha_p/2}}_{\phi_{\delta,\textrm{m}}}+\mathring{\phi}_\delta.
	\end{equation}
    Here $\phi_{\delta, \mathrm{m}}$ arises from the leading term $\phi_0$ in \eqref{in:eq:expansion}, and the remaining terms $\mathring{\phi}_\delta$ satisfy, for sufficiently small $\delta$, the bound
	\begin{equation} 
		\norm{\delta^{-\alpha_p}(\bar{t}\delta-f(\bar{t}\delta))^{\alpha_p-1}\{1,\bar{t}\partial_{\bar{t}},\partial_{\bar{x}}\}^{s-30\kappa_p}\mathring{\phi}_\delta}_{L^2(\M^{1/4}_p)}\lesssim_{f,\psi} 1,
    \end{equation}
    which rearranges to
    \begin{equation} \label{gen:eq:bound1}
		\norm{\{1,\bar{t}\partial_{\bar{t}},\partial_{\bar{x}}\}^{s-30\kappa_p}\mathring{\phi}_\delta}_{L^2(\M^{1/4}_p)}\lesssim_{f,\psi}\delta.
    \end{equation}

	Next, we analyse the main part $\phi_{\delta,\textrm{m}}$ of $\phi_\delta$ by expanding $f$ in $\M^{1/4}_p$ to get the estimates
	\begin{equation}
        \norm{\{1,\partial_{\bar{x}}\}^{s-30\kappa_p}\partial_{\bar{x}} f(\bar{x}\delta)}_{L^2(\M^{1/4}_p)}\lesssim \delta, \qquad
		\norm{\{1,\partial_{\bar{x}}\}^{s-30\kappa_p} f(\bar{x}\delta)}_{L^2(\M^{1/4}_p)} \lesssim \delta^2.
    \end{equation}
	Therefore, we may formally expand $\phi_{\delta,\textrm{m}}$ as
	\begin{equation}
		\phi_{\delta,\textrm{m}}=c_p(\bar{t}-f(\bar{x}\delta)/\delta)^{-\alpha_p}(1-\abs{\partial f(\bar{x}\delta)}^2)^{\alpha_p/2}=c_p \bar{t}^{-\alpha_p}(1-\bar{t}^{-1}\mathcal{O}(\delta))^{-\alpha_p}(1-\mathcal{O}(\delta))^{\alpha_p/2}.
	\end{equation}
    Thus, in the region $\M^{1/4}_p$ where $|\bar{x}|, \bar{t}, \bar{t}^{-1} \lesssim 1$, one can use suitable Sobolev estimates to derive that
	\begin{equation}\label{gen:eq:bound2}
		\norm{\{1,\bar{t}\partial_{\bar{t}},\partial_{\bar{x}}\}^{s-30\kappa_p}\Big(\phi_{\delta,\textrm{m}}-c_p\bar{t}^{-\alpha_p}\Big)}_{L^2(\M^{1/4}_p)}\lesssim_f\delta.
	\end{equation}
	
    \vspace{0.5em}
    \emph{Step 3 (Local stability):}
	Let $\epsilon_{s, C}$ be the small constant from \cref{gen:lem:Cauchy}, and one may fix $C = 10^4$ for concreteness. 
    Let us note that by \cref{in:eq:cor_assumptions} the hypersurface $\Sigma_\delta=\{\tstar=\delta\}=\{t-f(x)=\delta\}$ is $H^s$ regular.
    Via a trace theorem, the bounds \cref{gen:eq:bound1,gen:eq:bound2} are sufficient to conclude that, for $\delta$ sufficiently small, it holds that
	\begin{equation}
        \norm{(\phi_\delta-c_p\bar{t}^{-\alpha_p})}_{H^{s_0-30\kappa_p-1}(\Sigma_\delta)}+\norm{(\partial_t\phi_\delta+\alpha_pc_p\bar{t}^{-1-\alpha_p})}_{H^{s_0-30\kappa_p-2}(\Sigma_\delta)}\leq \frac{\epsilon_{ 10^4}}{2}.
	\end{equation}
	By assumption, as in \cref{gen:eq:gen_pert}, the initial data of $\bar{\phi}$ is $\epsilon_2$ close to that of $\phi$.
    Choosing $\epsilon_2<\frac{\epsilon_{s, 10^4}}{2}$ we obtain from \cref{gen:lem:Cauchy} that $\bar{\phi}$ forms an ODE blow-up in $\M^{-1}_p\cap\{\bar{t}>\bar{f}(x)\}$ for some $\bar{f}\in H^{s-232\kappa_p}$.
	Furthermore from \cref{in:eq:sing_Cauchy} we also have the bound
	\begin{equation}
		\norm{f-\bar{f}}_{H^{s_0-232\kappa_p}}\lesssim\epsilon_2,
	\end{equation}
	together with the propagation of higher order norm.
    
	\emph{Step 4 (Global stability):}
    Using uniqueness of the solution, we may patch together the different $\bar{\phi}$ generated in Steps 1 to 3, to generate a solution $\bar{\phi}$ in a region $\bigcup_{p\in\underline{\Sigma}_0}\M^{-1}_p$.
	In the rest of the region, we can apply Cauchy stability to conclude that the solution remains bounded.
\end{proof}

We can apply the same scaling argument to propagate regularity all the way to the singularity:

\begin{proof}[Proof of \cref{in:cor:smoothness}]
    Since the smoothness is a local statement, pick a point $p\in \Sigma_0$ and assume without loss of generality that $f(p)=\partial f(p)=0$.
    Consider $\mathcal{M}_p^\delta=\{\abs{x}<3(\delta/10+t)\}\cap \{\tstar<\delta\}\cap\{t>\delta/4\}$ for $\delta$ sufficiently small to be fixed later.
    
    First, let us notice that $\phi$ is $H^s$ regular in $\mathcal{M}_p^\delta$ from propagation of regularity for \cref{in:eq:main}.
    That is, for any spacelike hypersurface $S\subset\mathcal{M}^\delta_p$ it holds that $(\phi,\partial_t\phi)\in H^{s+1}(S)\times H^{s}(S)$.
    
    As in \emph{Step 2} of the proof of \cref{gen:lemma:local}, we may study the rescaled solution $\phi_\delta$ and the corresponding error $\mathring{\phi}_\delta$ as in \cref{gen:eq:rescaled}.
    The same argument as there implies that the for $\delta$ sufficiently small, we may use \cref{gen:lem:Cauchy} to conclude that $f\in H^{s-202\kappa_p}$.
\end{proof}

Finally, we can use the more restrictive assumption of $f\in C^{1,\alpha}$ to obtain the same propagation of regularity:

\begin{proof}[Proof of \cref{in:cor:alpha}]
    The strategy of this proof will be show that, by comparing to a smoother ansatz, the weak asymptotic expression \eqref{eq:weak_asymptotics} together with the assumed smoothness away from the singularity is enough to meet the requirements of \cref{in:thm:main2_precise}, upon using the scaling and Poincar\'e symmetries of the equation. As we will see, it is enough that the a priori $C^{1, \alpha}$ regularity of $f$ and \eqref{eq:weak_asymptotics} is ``better than scaling''.
    
    \vspace{0.5em}
    \emph{Step 1 (Time coordinate $\zeta$):}
    In order to construct our smoother ansatz, we begin by smoothing out the boundary defining function of the singularity $\Sigma = \{ \tstar = 0 \} \cap \mathcal{M}$.
	Pick some non-zero positive $\psi\in C^\infty_c(\M)$ supported away from $\Sigma$ in $\tilde{\BB}=\{t\geq f(x), t^2+\abs{x}^2\leq 10\}$ and  consider the solution of the (elliptic) Dirichlet problem 
	\begin{equation}
        (\partial_t^2+\delta^{ij}\partial_{x^i}\partial_{x^j})\zeta=\psi,\qquad
		\zeta|_{\partial\tilde{\BB}}=0.
	\end{equation}
    From \cite[Theorem 8.33, Theorem 8.34]{gilbarg_elliptic_2001}, we know that a unique solution $\zeta$ exists and satisfies $\zeta\in C^{1,\alpha}(\tilde{\BB})$. We note the following consequences of elliptic theory:
    \begin{enumerate}[(i)]
        \item
            By interior regularity and the strong maximum principle, $\zeta \in C^{\infty}(\mathring{\tilde{\BB}})$ and $\zeta > 0$ in $\mathring{\tilde{\BB}}$.
        \item
            By Hopf's lemma, it follows that $\partial_t \zeta > 0$ on the past boundary $\partial \tilde{\BB} \cap \{ \tstar = 0 \}$. Further, using that $\{ \tstar = 0 \} = \{ t = f(x) \}$ is spacelike, we have that, in some neighborhood $\tilde{\mathcal{M}} = \mathcal{M} \cap \{ \tstar \leq \tstar_0 \} \cap \{ \zeta \leq \zeta_0 \}$ with $\tstar_0, \zeta_0$ sufficiently small, one has
            \[
                \eta^{-1}( \mbox{d} \zeta, \mbox{d} \zeta ) \leq - \frac{(\partial_t \zeta)^2}{2} < 0,
            \]
            so $\zeta$ acts as a time function in $\tilde{\mathcal{M}}$ which synchronizes the singularity at $\{ \zeta = 0 \}$.

        \item
            Assuming, moreover, that $\supp \psi \cap \tilde{\mathcal{M}} = \emptyset$, we derive higher order $L^{\infty}$ estimates using the following elliptic estimate: for some $p \in \tilde{\mathcal{M}}$ with $r = d(p, \Sigma)$,
            \[
                \| \partial^{k+1} \zeta \|_{L^{\infty}(\BB_{r/4}(p))} \lesssim r^{-k} \| \partial \zeta - \partial \zeta(p) \|_{L^{\infty}(\BB_{r/2}(p))} \lesssim r^{\alpha - k},
            \]
            where in the final step we used the $C^{1,\alpha}$ regularity of $\zeta$ in $\tilde{\BB}$. Using that $r \sim \zeta$ in $\tilde{\mathcal{M}}$, one has
	        \begin{equation}\label{gen:eq:boundary_def}
                \zeta,\partial\zeta\in L^\infty(\M), \quad \text{and } \zeta^{k-\alpha-1}\partial^{k}\zeta\in L^\infty (\M) \quad \forall k\geq2.
	        \end{equation}
    \end{enumerate}
	From propagation of regularity for the wave equation, it follows that the solution $\phi$ is $H^s$ regular up to $\{\zeta=\zeta_0/2\}$.
	
    \vspace{0.5em}
	\emph{Step 2 (Smooth ansatz):}   
    Let us define $\phi_0 \coloneqq c_p\zeta^{-\alpha_p}\abs{\eta^{-1}(\dd\zeta,\dd\zeta)}^{\frac{1}{p-1}}$. 
    Note that the definition of $\phi_0$ matches that of $\tilde{\phi}_0$ with $\tstar$ in place of $\zeta$, since $\eta^{-1}(\dd\tstar,\dd\tstar)=1-\abs{\partial f}^2$.
    Indeed, one easily checks that $\lim_{\tstar\to0}\phi_0/\tilde{\phi}_0=1$.
    We will now compute $\Box \phi_0$ in the region $\tilde{\mathcal{M}}$ as follows.
    \begin{align}
        \Box \phi_0 
        &= - \alpha_p c_p \nabla_{\mu} \left( \zeta^{-\alpha_p - 1} \nabla^{\mu} \zeta | \eta^{-1}( \mbox{d} \zeta, \mbox{d} \zeta )|^{\frac{1}{p-1}} 
        + \zeta^{-\alpha_p} |\eta^{-1}(\mbox{d}\zeta, \mbox{d}\zeta)|^{-\frac{p-2}{p-1}} \partial \zeta * \partial^2 \zeta \right) \nonumber \\[0.5em]
        &= - \alpha_p (\alpha_p + 1) c_p \zeta^{- \alpha_p - 2} |\eta^{-1} (\mbox{d} \zeta, \mbox{d} \zeta)|^{\frac{p}{p-1}} + \zeta^{-\alpha_p - 1} |\eta^{-1}(\mbox{d} \zeta, \mbox{d} \zeta)|^{- \frac{p-2}{p-1}} \partial \zeta * \partial \zeta * \partial^2 \zeta \nonumber \\[0.5em]
        &\hspace{2cm} + \zeta^{- \alpha_p} \left( |\eta^{-1}(\mbox{d} \zeta, \mbox{d} \zeta)|^{-\frac{p-2}{p-1}} \partial^3 \zeta * \partial \zeta + |\eta^{-1}(\mbox{d} \zeta, \mbox{d} \zeta)|^{- \frac{2p-3}{p-1}} \partial \zeta * \partial \zeta * \partial^2 \zeta * \partial^2 \zeta \right) \nonumber \\[0.5em]
        &= - \phi_0^p + O(\zeta^{-\alpha_p - 2 + \alpha}),
    \end{align}
    where in the final step we make use of \eqref{gen:eq:boundary_def} and that $|\eta^{-1}(\mbox{d} \zeta, \mbox{d} \zeta)| \sim 1$ in $\tilde{\mathcal{M}}$. In fact, one can take more derivatives to find that, for any multiindex $\beta$ (including both time and spatial derivatives), 
    \begin{equation} \label{gen:eq:forcing_estimate}
        \partial^{\beta} P[\phi_0] = O(\zeta^{- \alpha_p - 2 + \alpha - |\beta|}).
    \end{equation}

    Let us next write $\mathring{\phi} \coloneqq \phi-\phi_0$. To estimate this in $L^{\infty}$, we note first that since $\zeta = 0$ at $\{ t = f(x) \}$ and $\zeta$ is $C^{1, \alpha}$ in $\tilde{\mathcal{M}}$, one must have that 
    \[
        \partial_{x^i} \zeta = - \partial_{x^i} f \partial_t \zeta + O(\zeta^{\alpha}), \quad \zeta = (t - f(x)) \partial_t \zeta + O(\zeta^{1 + \alpha}),
    \]
    from which we get that $ \frac{\zeta}{\sqrt{-\eta^{-1}(\mbox{d}\zeta, \mbox{d}\zeta)}} = \frac{\tstar}{\sqrt{1 - |\partial f|^2}} + O(\zeta^{1 + \alpha})$.
    Putting this to the $- \alpha_p$th power, we have $\phi_0 - \tilde{\phi}_0 = O(\zeta^{-\alpha_p + \alpha})$, or equivalently
    \begin{equation}\label{gen:eq:phi_ring_sharp}
        \mathring{\phi}\in \zeta^{-\alpha_p+\alpha}L^\infty(\M)
    \end{equation}

    \vspace{0.5em}
    \emph{Step 3 (High regularity estimates for $\mathring{\phi}$):}
    Consider the following equation for $\mathring{\phi}$:
	\begin{equation}\label{gen:eq:phi_ring}
		\Box\mathring{\phi}=\mathcal{N}_{\phi_0}[\mathring{\phi}]+ \phi_0^{p-1}\mathring{\phi}-P[\phi_0].
	\end{equation}
	We claim that for all multi-indices $\abs{\beta}\leq s$, including both $t$ and $x$ derivatives,
	\begin{equation}\label{gen:eq:phi_ring_uniform}
        \zeta^{\alpha_p-1/2-\alpha/2+\abs{\beta}}\partial^\beta \mathring{\phi}\in L^2(\M).
	\end{equation}
	
    Let us recall the definition of the energy momentum tensor $\T_{\mu\nu}[\Phi]$ from \cref{eq:energymomentum}, and define the current $\J^\mu[\Phi]=\zeta^{q}(\eta^{-1})^{\mu\nu}\T_{\nu\sigma}[\Phi]T^\sigma$, where $T = \partial_{t}$ is the usual time translation vector field in Minkowski space. 
	As in \cref{back:lem:energy_est}, we obtain that 
	\[
        \mathrm{div} \J[\mathring{\phi}]=q\zeta^{q-1}\T_{\nu\sigma}[\mathring{\phi}] T^{\sigma}(\eta^{-1})^{\mu\nu}(\dd\zeta)_\mu+\zeta^{q}\partial_t\mathring{\phi} \left(\Box\mathring{\phi}\right).
	\]
    Using the bounds on $\zeta$ in $\{\zeta<\zeta_0\} \subset \tilde{\mathcal{M}}$, we get that the first term is negative coercive with a bound
	\[
        \partial_t\zeta \zeta^{q-1} \left( (\partial_t\mathring{\phi})^2+\delta^{ij}\partial_{x^i}\mathring{\phi}\partial_{x^j} \mathring{\phi}\right) \leq-10\zeta^{q-1}\T_{\nu\sigma}[\mathring{\phi}]T^{\sigma}(\eta^{-1})^{\mu\nu}(\dd\zeta)_\mu\leq100 \partial_t\zeta \zeta^{q-1} \left( (\partial_t\mathring{\phi})^2+\delta^{ij}\partial_{x^i}\mathring{\phi}\partial_{x^j} \mathring{\phi}\right).
	\]
	
	Let us set $q=2\alpha_p+2-\alpha$ and define $\Sigma_{\zeta_1}=\{\zeta=\zeta_1\}\cap\M$.
    We apply the divergence theorem in $\D_{\zeta_1,\zeta_2}=\tilde{\M}\cap\{\zeta\in(\zeta_1,\zeta_2)\}$ for $0<\zeta_1<\zeta_2\leq \zeta_0$ to obtain
	\begin{equation}
		\int_{\Sigma_{\zeta_2}}\J^\mu[\mathring{\phi}] \geq \int_{\Sigma_{\zeta_1}}\J^\mu[\mathring{\phi}]-\int_{\D_{\zeta_1,\zeta_2}}\mathrm{div}\J[\mathring{\phi}].
	\end{equation}
	Applying Cauchy-Schwarz, we obtain that 
	\begin{equation}
		\int_{\D_{\zeta_1,\zeta_2}} \zeta^q\partial_t\mathring{\phi}\Box\mathring{\phi}\lesssim 10^{-2}\int_{\D_{\zeta_1,\zeta_2}} \zeta^{q-1}(\partial_t\mathring{\phi})^2\abs{\partial_t\zeta} +100\int_{\D_{\zeta_1,\zeta_2}} \zeta^{q+1}\abs{\partial_t\zeta}^{-1}(\Box\mathring{\phi})^2.
	\end{equation}
    Using the a priori $\mathring{\phi}, \zeta^2 P[\phi_0] \in \tstar^{-\alpha_p+\alpha}L^\infty(\M)$ estimates from Step 2, and using that $|\mathcal{N}_{\phi_0} f| \lesssim \zeta^{-2 + \alpha_p} |f|^2$, we derive that \cref{gen:eq:phi_ring_uniform} holds for $\abs{\beta}=1$.
	
	Commuting \cref{gen:eq:phi_ring} with $\partial\in\{\partial_t,\partial_x\}$, we obtain
	\begin{equation}
		\Box\partial\mathring{\phi}=\phi_0^{p-1}\partial\mathring{\phi}+\mathring{\phi}\partial\phi_0^{p-1}+\partial\mathcal{N}_{\phi_0}[\mathring{\phi}]-\partial P[\phi_0].
	\end{equation}
	Using \cref{gen:eq:boundary_def} we have the uniform estimates $\zeta^{2+\abs{\beta}}\partial^{\beta}\phi_0^{p-1}\in L^\infty(\M)$ for all $\abs{\beta}$
	and in particular
	\begin{equation}
		\int_{\D_{0,\zeta_1}}\zeta^{q+1}\left(\mathring{\phi}\partial\phi_0^{p-1}\right)\lesssim\int_{\D_{0,\zeta_1}} \zeta^{q-3}\mathring{\phi}^2.
	\end{equation}
    Note that, from \eqref{gen:eq:forcing_estimate}, we also have $\zeta^{\alpha_p + 1 - \frac{\alpha}{2} +1}\partial P[\phi_0]\in L^2(\M)$.
	Using the same energy estimate as before with $q = 2 \alpha_p + 4 - \alpha$ yields \cref{gen:eq:phi_ring_uniform} for $k=2$.
    The general case follows similarly by induction, together with known estimates on $\partial^{\beta} \phi_0$ and $\partial^{\beta} P[\phi_0]$ from Steps 1 and 2.
	
    \vspace{0.5em}
	\emph{Step 4 (Conclusion):}
    Finally, we apply Cauchy stability, scaling and local existence steps as in the proof of \cref{gen:lemma:local}. Let $p$ be a point on the singularity, and, for $\delta$ small, consider the region $\M^{1/4}_p \subset \tilde{\M}$ as in \cref{gen:eq:bound2}.
    Let us define $\phi_\delta$ as the $\delta$-rescaling of $\phi$ as in \cref{def:rescale}, and split it as
	\begin{equation}
		\phi_\delta(\bar{t},\bar{x})=\delta^{\alpha_p}\phi_0(\delta \bar{t},\delta\bar{x})+\mathring{\phi}_\delta.
	\end{equation}
    The bounds \cref{gen:eq:phi_ring_sharp,gen:eq:phi_ring_uniform} together with interpolation imply, for any $k \geq 0$, the following estimate (recall that $\bar{t} \sim 1$ and $\zeta \sim \delta$ in $\M_p^{1/4}$):
	\begin{multline}
        \norm{\{1, \bar{t} \partial_{\bar{t}}, \partial_{\bar{x} }  \}^k\mathring{\phi}_{\delta}}_{L_{\bar{x},\bar{t}}^2(\M^{1/4}_p)}\lesssim \norm{\{1, \bar{t} \partial_{\bar{t}}, \partial_{\bar{x} }  \}^{k'+k}\mathring{\phi}_{\delta}}_{L_{\bar{x},\bar{t}}^2(\M^{1/4}_p)}^{\frac{k}{k'+k}}\norm{\mathring{\phi}_{\delta}}_{L_{\bar{x},\bar{t}}^2(\M^{1/4}_p)}^{\frac{k'}{k'+k}}
        \\
        \lesssim_{k,k',\zeta}  \delta^{\frac{k}{k'+k}(\alpha/2-\frac{n}{2})+\frac{k'}{k'+k}\alpha}\leq \delta^{\alpha/4},
	\end{multline}
    where we choose $k'$ such that $\frac{nk}{2(k'+k)}\leq\alpha/4$ in order to negate the $\frac{n}{2}$ loss in the first factor of the second line from the changed volume form.
    Using the assumption on $s$, we can in particular take $k=200\kappa_p+5+\frac{n}{2}$ to obtain $\delta^{\alpha/4}$ smallness.
    
	Similar estimates apply to the leading order part; from the estimates of Steps 1 and 2 we get the following variant of \cref{gen:eq:bound2}: for all $k\geq0$
	\begin{equation}
        \norm{\{1,\bar{t}\partial_{\bar{t}}, \partial_{\bar{x}}\}^{k}\Big(\delta^{\alpha_p} \phi_0(\delta \bar{t}, \delta \bar{x}) - c_p\bar{t}^{-\alpha_p}\Big)}_{L_{\bar{x},\bar{t}}^2(\M^{1/4}_p)}\lesssim \delta^{\alpha/4}.
	\end{equation}
	Via \cref{gen:lem:Cauchy} we first obtain the improvement $f\in H^{5+\frac{n}{2}}$ with $\delta^{\alpha/4}$ size control, and the $f\in H^{s-200\kappa_p}$ norm is propagated irrespective of smallness. 
    Here $s$ is the assumed regularity of $\phi$ away from the singularity, inherited from our assumption on the initial data.
\end{proof}

	\pagebreak
    \printbibliography

\end{document}